\documentclass{amsart}
\usepackage{color}
\usepackage{amscd}
\usepackage{nicefrac}
\usepackage{graphicx}
\usepackage{hyperref}
\usepackage{amssymb}
\usepackage{euscript}
\usepackage{tikz}
\usetikzlibrary{calc}
\usepackage[all]{xy}

\numberwithin{equation}{section}
\renewcommand{\subsection}[1]{\hspace{-\parindent}\refstepcounter{subsection}{\bf (\arabic{section}\alph{subsection}) #1.}\addcontentsline{toc}{subsection}{\bf #1.}}

\newenvironment{nouppercase}{%
  \renewcommand{\uppercasenonmath}[1]{}}{}
\newcommand{\mybox}[1]{\parbox[t]{37em}{#1}}
\allowdisplaybreaks

\theoremstyle{plain}
\newtheorem{thm}{Theorem}[section]

\newtheorem{theorem}[thm]{Theorem}

\newtheorem{corollary}[thm]{Corollary}

\newtheorem{definition}[thm]{Definition}

\newtheorem{remark}[thm]{Remark}
\newtheorem{convention}[thm]{Convention}
\newtheorem{proposition}[thm]{Proposition}
\newtheorem{example}[thm]{Example}
\newtheorem{non-example}[thm]{Non-example}

\newtheorem{lemma}[thm]{Lemma}
\newtheorem{setup}[thm]{Setup}

\newtheorem{scholium}[thm]{Scholium}

\newtheorem*{claim*}{Claim} 
\newtheorem*{lemma*}{Lemma}
\newtheorem*{theorem*}{Theorem}
\newtheorem*{conjecture*}{Conjecture}
\newtheorem{application}[thm]{Application}

\newcommand{\bC}{{\mathbb C}}

\newcommand{\bG}{{\mathbb G}}

\newcommand{\bK}{{\mathbb K}}

\newcommand{\bP}{{\mathbb P}}
\newcommand{\bQ}{{\mathbb Q}}
\newcommand{\bR}{{\mathbb R}}

\newcommand{\bZ}{{\mathbb Z}}

\newcommand{\scrA}{\EuScript A}
\newcommand{\scrB}{\EuScript B}
\newcommand{\scrC}{\EuScript C}
\newcommand{\scrD}{\EuScript D}

\newcommand{\scrF}{\EuScript F}
\newcommand{\scrG}{\EuScript G}

\newcommand{\scrK}{\EuScript K}

\newcommand{\scrO}{\EuScript O}
\newcommand{\scrP}{\EuScript P}
\newcommand{\scrQ}{\EuScript Q}
\newcommand{\scrR}{\EuScript R}

\newcommand{\scrT}{\EuScript T}

\newcommand{\frakb}{\mathfrak{b}}
\newcommand{\frakc}{\mathfrak{c}}

\newcommand{\frakg}{\mathfrak{g}}

\newcommand{\frakt}{\mathfrak{t}}
\newcommand{\frakw}{\mathfrak{w}}

\newcommand{\half}{{\textstyle\frac{1}{2}}}

\newcommand{\iso}{\cong}
\newcommand{\htp}{\simeq}
\newcommand{\smooth}{C^\infty}

\renewcommand{\hom}{\mathit{hom}}

\newcommand{\cornerbar}[1]{
  \tikz[baseline=(n.base)]{\node(n)[inner sep=1pt]{$#1$};
    \draw[line cap=round](n.north west)--(n.north east)--(n.south east);
  }
}

\newcommand{\cornerdoublebar}[1]{
 \tikz[baseline=(n.base)]{\node(n)[inner sep=1pt]{$#1$};
    \draw[line cap=round](n.north west)--(n.north east)--(n.south east);
    \draw[line cap=round]($(n.north east) + (.065,0)$)--($(n.south east)+(.065,0)$);	
  }
}

\newcommand{\cornersubbar}[1]{ 
 \begin{tikzpicture}[baseline=(n.base)]
    \node(n)[inner sep=1pt]{{$\scriptstyle #1$}};
    \draw[line cap=round](n.north west)--(n.north east)--(n.south east);
  \end{tikzpicture}
}

\newcommand{\calnablaq}{
\mathcal{V}\hspace{-.45em}\raisebox{0.1ex}{$\bar{\,\,}_{\raisebox{-.1ex}{$\scriptstyle q$}}$}\hspace{-.1em}}

\newcommand{\relativenablaq}{\dot{\calnablaq}}

\newcommand{\oPerpSymbol}{\begin{tikzpicture}[scale=0.112]
  \draw (0,-.2)--(0,1); \draw (-.98,-.2)--(.98,-.2);
  \draw (0,0) circle [radius=1];
\end{tikzpicture}}
\newcommand{\operp}{\mathbin{\raisebox{-.8pt}{\oPerpSymbol}}}

\title[LEFSCHETZ FIBRATIONS]{\Large\larger\rm Fukaya $A_\infty$-structures associated to\\ Lefschetz fibrations. VIII}
\author{Paul Seidel}

\begin{document}
\begin{nouppercase}
\maketitle
\end{nouppercase}
\begin{abstract}
We use Lefschetz pencil methods to derive structural results about Fukaya categories of Calabi-Yau hypersurfaces; in particular, concerning their dependence on the Novikov parameter.
\end{abstract}

\section{Introduction\label{sec:intro}}

\fbox{\parbox{40.75em}{%
{\bf Warning to the reader.} {\em The arguments in this paper cite two preprints \cite{seidel21,seidel21b} which do not at present exist. Until that is remedied, the proofs cannot be considered complete, and this paper should be regarded as a preliminary research announcement.}}
}

\subsection{Overview}
Take a closed symplectic manifold of dimension $(2n-2)$, which is a symplectic Calabi-Yau. By this, we mean that it has zero first Chern class and integral symplectic class; for instance, it could come from a smooth projective Calabi-Yau variety and a choice of ample line bundle. The Fukaya categories of such manifolds are central objects in symplectic topology and mirror symmetry. By construction, the Fukaya category involves a formal parameter, the Novikov parameter $q$ (one could call it a formal ``family of categories'' parametrized by $q$). Intuitively, $-\log(q)$ controls the scale of the symplectic form (hence, $q \rightarrow 0$ corresponds to the large volume limit). A priori, the Fukaya category can involve arbitrary series in $q$, but one would naturally like to constrain its structure on some deeper level. 

Fundamental progress in this direction was achieved in Ganatra-Perutz-Sheridan's paper \cite{ganatra-perutz-sheridan15}, where they characterize $q$ intrinsically, among all its possible reparametrizations, in terms of the noncommutative Hodge theory associated to the Fukaya category (this is partly based on earlier ideas of Kontsevich and Barannikov, as well as the canonical coordinates in classical, meaning enumerative, mirror symmetry). 
When applying this characterization, one starts with an explicit algebraic model for the Fukaya category written in terms of an a priori unknown parameter $t = f(q)$, a kind of result that can be obtained in many examples by showing that the Fukaya category is a versal deformation of its $q = 0$ limit (e.g.\ \cite{seidel03b,sheridan11b}). The Ganatra-Perutz-Sheridan argument then determines the inverse map of $f$, thereby allowing one to recover the natural parametrization by $q$. As a highly desirable byproduct, this approach computes certain Gromov-Witten invariants, for instance providing a new proof of the classical formula for rational curves on the quintic threefold. The structural implications of this method for the Fukaya category depend on the nature of the algebraic model. In the simplest possible situation, if that model were polynomial in $t$, then the resulting picture of the Fukaya category is that it is polynomial in $f(q)$. However, that kind of information has to be obtained on a case-by-case basis, typically being read off from the mirror geometry.

This paper pursues a different approach, which eschews noncommutative Hodge theory, relying instead on noncommutative counterparts of classical algebro-geometric notions (notably that of divisor). In a nutshell one could say that, while Ganatra-Perutz-Sheridan think of the Hodge theory of the mirror, and implement the resulting insights in terms that involve only the Fukaya category, we go through a philosophically parallel process concerning the geometry of the mirror. To be clear: mirror symmetry does not actually appear in our results, it only provides the motivation. Moreover, while the outcome of our considerations has some overlap with that of Ganatra-Perutz-Sheridan, the formalism itself is entirely different.
%

Concretely, we start with a monotone symplectic manifold (such as a Fano variety) carrying an anticanonical Lefschetz pencil, and take our Calabi-Yau to be a member of that pencil. That additional geometric structure will be crucial throughout. In the end, we prove that the $q$-dependence in a subcategory of the Fukaya category is restricted to explicitly given finitely generated rings of functions of $q$. In the simplest version of our result, the dependence turns out to be polynomial in some $t = f(q)$. The function $f$ is given in terms of a Schwarzian differential equation involving Gromov-Witten invariants. Our argument does not compute those invariants (of course, there are other methods available for that). The $q$-dependence statement is part of a deeper description of the Fukaya category, which greatly constrains its structure. 

%

\subsection{Setup\label{subsec:setup}}
Let's introduce notation for the various manifolds that are part of our setup (this follows the idiosyncratic conventions from \cite{seidel15}). Write $\bC| = \bC \cup \{\infty\}$ for the projective line. The graph of our Lefschetz pencil, which one can think as being obtained by blowing up the base locus of the pencil, is a Lefschetz fibration
\begin{equation} \label{eq:cornerbar-p}
\cornerbar{p}: \cornerbar{E} \longrightarrow \bC|.
\end{equation}
We assume that the fibre at $\infty$ is smooth, and call that $\bar{M}$. From the blowup, $\cornerbar{E}$ inherits an exceptional divisor $\delta E|$, whose intersection with $\bar{M}$ we denote by $\delta M$. It comes with a canonical isomorphism $\delta E| \iso \bC| \times \delta M$. For ease of reference, here is a list of the manifolds we have mentioned, and other related ones, arranged in increasing order of dimensions:
\begin{equation}
\begin{array}{l|l} 
\delta M & \parbox{25em}{a closed $(2n-4)$-manifold, which represents the symplectic class in $\bar{M}$} \vspace{.1em}
\\ \hline
\bar{M} & \parbox{25em}{the symplectic Calabi-Yau, a closed $(2n-2)$-manifold, and the fibre at $\infty$ of \eqref{eq:cornerbar-p}} 
\\ \hline 
M = \bar{M} \setminus \delta M & \parbox{25em}{a noncompact exact symplectic $(2n-2)$-manifold} 
\\ \hline
\delta E| & \parbox{25em}{the exceptional divisor, isomorphic to $\bC| \times \delta M$} 
\\ \hline
\delta E = \delta E| \setminus \delta M & \parbox{25em}{where we think of $\delta M$ as lying over $\infty \in \bC|$} \\ \hline
\cornerbar{E} & \parbox{25em}{A closed $2n$-manifold, and the total space of \eqref{eq:cornerbar-p}} \\ \hline
\bar{E} = \cornerbar{E} \setminus \bar{M} & \parbox{25em}{which comes with the restriction of \eqref{eq:cornerbar-p}, a Lefschetz fibration $\bar{p}: \bar{E} \rightarrow \bC$ containing $\delta E$ as a fibrewise symplectic hypersurface} \\ \hline
E = \bar{E} \setminus \delta E & \parbox{25em}{which comes with an exact Lefschetz fibration $p: E \rightarrow \bC$.}
\end{array}
\end{equation}

\begin{convention} \label{th:connected}
We will always assume that $\bar{M}$, and therefore $\bar{E}$, is connected. We will not always impose a connectedness assumption on $\delta E$ (or equivalently $\delta M$), even though later, one of our results will require it (Theorem \ref{th:main}).
\end{convention}

The simplest starting point for our discussion is the last-mentioned structure, namely the exact Lefschetz fibration $p$. Choose a basis of vanishing cycles, which is an ordered collection of Lagrangian spheres $(V_1,\dots,V_m)$ in $M$. Let $\scrB$ be the full $A_\infty$-subcategory of the Fukaya category of $M$ formed by these objects. One can think of it as a single $A_\infty$-algebra,
\begin{equation} \label{eq:b-algebra}
\scrB = \bigoplus_{i,j=1}^m \mathit{CF}^*(V_i,V_j),
\end{equation}
where $\mathit{CF}^*$ are the spaces of Floer cochains (more precisely, this is an $A_\infty$-algebra over the semisimple ring $R = \bQ^{\oplus m}$). Let's pass from $M$ to its compactification $\bar{M}$, and work relative to $\delta M$. That gives rise to a deformation of $\scrB$ (part of the relative Fukaya category).
Concretely, the deformation is given by operations
\begin{equation} \label{eq:mu-b}
\mu_{\scrB_q}^d \in \mathit{Hom}^{2-d}(\scrB^{\otimes d},\scrB)[[q]],
\end{equation}
which reduce to the previous ones if we set $q = 0$. One can think of them as forming an $A_\infty$-structure on $\scrB_q = \scrB[[q]]$ (generally speaking, relative Fukaya categories involve a $\mu^0$ curvature term, but in our case one can arrange that this is zero). The main question is, naively speaking: what series in $q$ appear in this deformation? Of course, Floer-theoretic structures on the cochain level depend on many auxiliary choices, and are unique only up to quasi-isomorphism. Hence, a better way of phrasing the question is to ask whether one can find a model, within the quasi-isomorphism class of $\scrB_q$, such that only certain functions of $q$ appear in it.

\begin{remark}
We use Fukaya categories with rational coefficients ($\bQ$ for $\scrB$, and $\bQ[[q]]$ for $\scrB_q$). This is not necessary for the definition of those $A_\infty$-structures, where one can work integrally. However, in the course of our construction, differential equations with respect to $q$ play a crucial role; to make their theory well-behaved, one has to be able to integrate formal power series. It seems possible that this might work for $\scrB$ defined over $\bZ$, and $\scrB_q$ defined over $\bZ[[q,\frac12 q^2,\frac16 q^3,\dots]]$; but the benefits seemed too slim to bother. 
\end{remark}

\subsection{Polynomiality\label{subsec:1-variable}}
We begin with a form of our result which, while not the most general or precise one, is more easily accessible.

\begin{theorem} \label{th:1-variable}
Suppose that the assumptions \eqref{eq:simply-connected} (topological) and \eqref{eq:as} (on the Gromov-Witten theory of $\cornerbar{E}$) hold. Then, that Gromov-Witten theory produces a function
\begin{equation} \label{eq:initial}
f \in \bQ[[q]], \; \; f(0) = 0, \;\; f'(0) \neq 0,
\end{equation}
such that: there is a quasi-isomorphic model for $\scrB_q$, in which all operations \eqref{eq:mu-b} have coefficients which are polynomials in $f$.
\end{theorem}

The first assumption is
\begin{equation} \label{eq:simply-connected}
H^1(\bar{M};\bZ) = 0.
\end{equation}
This is technical and presumably unnecessary, but it simplifies our task by removing the need to keep track of certain bounding cochains. Let's turn to the enumerative side. Take $A \in H_2(\cornerbar{E};\bZ)$ which has intersection number $k$ with $\bar{M}$. Let 
\begin{equation}
z_A \in H^{4-2k}(\cornerbar{E};\bQ)
\end{equation}
be the genus zero once-pointed Gromov-Witten invariant for curves in class $A$ (informally, this is the class of the cycle represented by $A$-curves). One can assemble these cohomology classes into formal series
\begin{equation} \label{eq:zk}
z^{(k)} = \sum_{A \cdot [\bar{M}] = k} q^{\delta E| \cdot A} z_A \in H^{4-2k}(\cornerbar{E};\bQ((q))).
\end{equation}
Out of the three nontrivial cases $k = 0,1,2$, we will only need two:
\begin{align} \label{eq:z1}
& z^{(1)} \in [\delta E|] q^{-1} + H^2(\cornerbar{E};\bQ[[q]]), \\
\label{eq:z2}
& z^{(2)} \in H^0(\cornerbar{E};\bQ[[q]]) = \bQ[[q]].
\end{align}
The leading term in \eqref{eq:z1} comes from trivial sections lying inside $\delta E| \iso \bC| \times \delta M$. For \eqref{eq:z2},  classes $A$ which have negative intersection number with $\delta E|$ don't contribute, since no stable maps representing them go through a generic point. The enumerative assumption is:
\begin{equation} \label{eq:as}
\parbox{35em}{$z^{(1)}$ lies in the linear subspace of $H^2(\cornerbar{E};\bQ((q)))$ spanned by the classes of the fibre $\bar{M}$ and exceptional divisor $\delta E|$.}
\end{equation}
Because of \eqref{eq:z1}, this assumption allows one to write
\begin{equation} \label{eq:write-z1}
q^{-1}[\delta E|] = \psi z^{(1)} - \eta [\bar{M}] \quad \text{ with }
\psi \in 1 + q\bQ[[q]], \quad \eta \in \bQ[[q]].
\end{equation}
We use the power series $\psi$, $\eta$ and $z^{(2)}$ to write down a Schwarz\-ian equation
\begin{equation} \label{eq:schwarz-eq}
S_qf + 8z^{(2)} \psi^2 + \left(\eta - \frac{\psi'}{\psi}\right)' + \frac12 \left( \eta - \frac{\psi'}{\psi} \right)^2 = 0.
\end{equation}
The function $f$ from \eqref{eq:initial} is a solution of this equation. Even though this characterizes $f$ only up to a certain ambiguity \eqref{eq:conformal}, that is irrelevant, as the statement holds for any solution (Remark \ref{th:doesnt-matter}). 

\begin{scholium} \label{th:scholium}
We need to recall some elementary background. The Schwarz\-ian differential (see e.g.\ \cite{osgood98, ovsienko-tabachnikov09}) is
\begin{equation} \label{eq:schwarz}
S_qf = \left(\frac{f''}{f'}\right)' - \frac{1}{2} \left( \frac{f''}{f'} \right)^2 = \frac{f'''}{f'} - \frac{3}{2} \left(\frac{f''}{f'}\right)^2.
\end{equation}
where $f' = df/dq$. Consider Schwarzian equations for a series $f$ as in \eqref{eq:initial}, of the form
\begin{equation} \label{eq:schwarz-2}
S_qf + g = 0 \quad \text{with some given } g \in \bQ[[q]].
\end{equation}
Any such equation has solutions, which can be found by recursively solving for the Taylor coefficients (starting with arbitrary initial values $f(q) = a_1 q + a_2 q^2 + \cdots$, $a_1 \in \bQ^\times$, $a_2 \in \bQ$). In complex geometry, the Schwarzian characterizes a function up to composition with conformal transformations on the left. In our formal situation, a solution of \eqref{eq:schwarz-2} in the class \eqref{eq:initial} is unique up to
\begin{equation} \label{eq:conformal}
f \longmapsto \frac{f}{a+bf}, \quad a \in \bQ^\times, \; b \in \bQ.
\end{equation}
Schwarzian equations appear in the following context. Consider the linear differential equation
\begin{equation} \label{eq:g-equation}
s'' - (g/2) s= 0.
\end{equation}
This has a basis of solutions
\begin{equation} \label{eq:y0y1}
s_0, s_1 \in \bQ[[q]], \quad \text{with } s_0(0) \neq 0, \; s_1(0) = 0, \; s_1'(0) \neq 0.
\end{equation}
Then, $f = s_1/s_0$ is a solution of \eqref{eq:schwarz-2}. This gives another proof of the existence of solutions, and also makes \eqref{eq:conformal} more intuitive (one can rescale the $s_k$, and also add a multiple of $s_1$ to $s_0$). Slightly more generally, given any equation
\begin{equation} \label{eq:g-equation-2}
s'' + h s' - (k/2) s = 0,
\end{equation}
we can substitute $\exp(-\!\int h/2 \, \mathit{dq}) s$ instead of $s$, which reduces it to the form \eqref{eq:g-equation} with 
\begin{equation} \label{eq:new-g}
g = k + h' + h^2/2.
\end{equation}
Hence, quotients of solutions to \eqref{eq:g-equation-2} still satisfy \eqref{eq:schwarz-2} for that choice of $g$ (one could also derive that directly, without the reduction step). 
\end{scholium}

In our context, \eqref{eq:schwarz-eq} arises from the linear second order equation
\begin{equation} \label{eq:2nd-order}
s'' + \left( \eta - \frac{\psi'}{\psi} \right) s' - 4z^{(2)} \psi^2 s = 0.
\end{equation}
It was shown in \cite{seidel16} that this equation arises naturally in the symplectic cohomology of $\bar{E}$. From there, through several intermediate stages, it makes its way to controlling the structure of $\scrB_q$. What actually happens is that the quasi-isomorphism class of $\scrB_q$ is completely determined by the following information: the directed subalgebra of $\scrB$ (which is at least in principle amenable to computation in many situations, see e.g.\ \cite{seidel04}); a small amount of extra information (which can be read off if one knows $\scrB_q$ up to first order in $q$); and the function $f$, which is an ``external'' input from Gromov-Witten theory. We refer to Section \ref{sec:outline} for further discussion. In particular, Theorem \ref{th:1-variable} will be derived in Section \ref{subsec:directed-2}, by combining two results stated there (Propositions \ref{th:1-variable-2} and \ref{th:1-variable-3}).

\begin{corollary} \label{th:convergent}
Suppose that assumptions \eqref{eq:simply-connected} and \eqref{eq:as} hold, and moreover, that the Gromov-Witten invariants $z^{(1)}$ and $z^{(2)}$ are locally convergent in $q$. Then there is an $A_\infty$-category quasi-isomorphic to $\scrB_q$, and which, for some $\epsilon>0$, is defined over the ring of $q$-series with (rational coefficients and) convergence radius $\geq \epsilon$.
\end{corollary}

In that situation, the coefficients of \eqref{eq:2nd-order} are locally convergent, hence so are its solutions and the function $f$. That makes Corollary \ref{th:convergent} an immediate consequence of Theorem \ref{th:1-variable}. (This kind of local convergence property of the Fukaya category has been known to hold in some instances, but those are based on homological mirror symmetry, which can't lead to a general statement of the kind made above.)

\subsection{Examples\label{subsec:examples}}
Take the pencil of cubics on $\bC P^2$. Since $\bar{M}$ is a torus, \eqref{eq:simply-connected} doesn't hold; but as we've mentioned, that is a technical condition, and we will explain how to work around it in this particular case (Remark \ref{th:elliptic-surface}). Assumption \eqref{eq:write-z1} holds by \cite[Lemma 6.2]{seidel15}. Starting with the explicit expressions in \cite[Section 6]{seidel15}, which are derived from \cite{bryan-leung97}, one can write all the functions in \eqref{eq:schwarz-eq} in terms of the (quasimodular) Eisenstein series $E_2$:
\begin{align}
& \frac{\psi'}{\psi} = \frac{E_2(q^3)-1}{2q}, \\
& \eta = -\frac{\psi'}{\psi}, \\
& 4 z^{(2)}\psi^2  = \alpha^2 - \alpha' + 2\alpha \frac{\psi'}{\psi}, 
\quad \text{where} \quad \alpha = \frac{E_2(q^3) - 9 E_2(q^9)}{8q}.
\end{align}
Using Ramajunam's formula $12 qE_2'(q) = E_2(q)^2 - E_4(q)$ and relations from \cite[Table 1]{maier11}, this turns into a remarkably simple expression for \eqref{eq:schwarz-eq}:
\begin{equation} \label{eq:e4}
S_q f + \frac{E_4(q^3)-1}{2q^2} = 0.
\end{equation}
A solution is
\begin{equation} \label{eq:eta9}
f = \frac{1}{(\eta(q)/\eta(q^9))^3+3} = q - 5 q^4 + 32 q^7 - 198 q^{10} + 1214 q^{13} + \cdots
\end{equation}
The fact that this satisfies \eqref{eq:e4} can be derived from modular form computations. Namely, $1/f$ is a hauptmodul for $\Gamma_0(9)$, and therefore $S_\tau(f) = S_\tau(1/f) = 2\pi^2 E_4(q^3)$ by \cite[Equation 7.2]{mckay-sebbar00}. Note that what appeared here is the Schwarzian with respect to $\tau = \log(q)/2\pi i$. An elementary change-of-variables formula says that $S_\tau f = 4\pi^2 (\half - q^2 S_q f)$, and that leads to \eqref{eq:e4}. The reader might want to look at \cite[Section 7]{seidel15}, which explains the relation to mirror symmetry; there, our $1/f$ appears as \cite[Equation 7.13]{seidel15}. Conceptually, the outcome of this discussion is that the modular $q$-dependence of the Fukaya category of the torus is directly related to the modularity of Gromov-Witten invariants of the rational elliptic surface.

\begin{remark} \label{th:picard-fuchs}
Suppose that we have a family $C_t$ of elliptic curves. The Picard-Fuchs equation for their periods is a linear second order differential equation. The modulus $\tau = \tau(C_t)$, which is the quotient of two periods, therefore satisfies a Schwarzian equation in $t$. Inverting the relationship, one gets a Schwarzian equation for the parameter $t$ as a function of $\tau$, or equivalently of $q = e^{2\pi i \tau}$. 
One can apply this to the mirror of $\bC P^2$, meaning the (fibrewise compactified) family $C_t = \{x+y+x^{-1}y^{-1} = t^{-1}\}$, following e.g.\ \cite{klemm-lian-roan-yau93, verrill01}. The outcome is the equation
\begin{equation} \label{eq:pf-schwarzian}
S_q t + \frac{(216t^3+1)}{2t^2(27t^3-1)^2} \Big(\frac{dt}{dq} \Big)^2 - \frac{1}{2q^2} = 0,
\end{equation}
which is indeed satisfied by the function $t = f$ from \eqref{eq:eta9} (compare \cite[Table 3]{verrill01}). Even though it arrives at the same solution, and also uses the Schwarzian, this method is substantially different from ours, even on a pedestrian computational level. Differential equations arising from periods, such as \eqref{eq:pf-schwarzian} and the underlying Picard-Fuchs equation, are self-contained in the sense that they only involve rational expressions in the unknown function and its derivatives. In constrast, both \eqref{eq:2nd-order} and \eqref{eq:schwarz-eq} contain an ``external'' transcendental term. A comparison of the two approaches yields additional relationships, whose meaning is not evident (for instance, by looking at \eqref{eq:schwarz-eq} and \eqref{eq:pf-schwarzian}, one sees that $E_4(q^3)/q^2$ is a rational expression in $f$ and $f'$).
%

For higher-dimensional Calabi-Yau manifolds, the Picard-Fuchs equations are of order $>2$, but they can still give rise to Schwarzian equations for mirror maps under certain circumstances. For one-parameter families of lattice-polarized $K3$ surfaces, one can write down equations similar to \eqref{eq:pf-schwarzian}, see for instance \cite[Equation 4.17]{lian-yau96}. In that case, the underlying phenomenon is that the Picard-Fuchs equation can be written as the symmetric square of a second order equation \cite[Theorem 5]{doran00}. A similar approach for one-parameter families of Calabi-Yau threefolds yields a more complicated result, where the Schwarzian equation has an extra term involving the Yukawa coupling, see \cite[Equation 4.24]{lian-yau96} and more generally \cite[Section 4]{doran00}. The appearance of this term makes the equation superficially more similar to ours; however, there is still no simple relation between the two, as far as this author can see.
\end{remark}

Our second example is the pencil of quintic hypersurfaces in $\bC P^4$. In that case, $Y = \cornerbar{E}$ is a hypersurface of bidegree $(1,5)$ in $X = \bC P^1 \times \bC P^4$. The Gromov-Witten invariants of such hypersurfaces can be computed using quantum Lefschetz \cite{givental96,lee01b,gathmann03}. We give a brief summary of how that works out in this particular case (for the benefit of non-specialists, and also because the computation raises a question, see Remark \ref{th:circular}). Let $x_1,x_2$ be the standard generators of $H^2(X;\bZ)$; $q_1 , q_2$ the corresponding quantum parameters; and $\hbar^{-1}$ another formal parameter. We assign degrees $|q_1| = 2$, $|q_2| = 0$, $|\hbar^{-1}| = -2$. On the side of $Y$, one looks at the (degree zero) expression
\begin{equation} \label{eq:jj}
J = 1 + \sum_{\substack{d_1,d_2 \geq 0 \\ (d_1,d_2) \neq (0,0)}} q_1^{d_1}q_2^{d_2} \sum_{m \geq 0} \hbar^{-m-2} z_{d_1,d_2,m,Y} \in H^*(Y;\bQ[[q_1,q_2,\hbar^{-1}]]),
\end{equation}
where $z_{d_1,d_2,m,Y}$ is the one-pointed Gromov-Witten invariant for rational curves of degree $(d_1,d_2)$, with an $m$-fold power of the gravitational descendant inserted. By definition, the functions $\psi, \eta, z^{(2)}$ from \eqref{eq:schwarz-eq} can be read off from the coefficients of \eqref{eq:jj} with $m = 0$: specifically, 
\begin{equation} \label{eq:read-off}
\begin{aligned}
& \psi^{-1} (25 q^{-1} x_2^2  + 5\eta x_1x_2) = \iota_! \big( \sum_{d_2 \geq 0} q^{5d_2-1} z_{1,d_2,0,Y} \big), \\
& z^{(2)} (x_1+5x_2) = \iota_! \big( \sum_{d_2 > 0} q^{5d_2-2} z_{2,d_2,0,Y} \big).
\end{aligned} 
\end{equation}
On the side of $X$, the corresponding expression comes with a twist involving the cohomology class $[Y] = x_1 + 5x_2$, and can be computed explicitly (see e.g.\ \cite{givental00}):
\begin{equation} \label{eq:i-function}
\begin{aligned}
I & \;=  (x_1+5x_2) + \sum_{\substack{d_1,d_2 \geq 0 \\ (d_1,d_2) \neq (0,0)}} y_1^{d_1}y_2^{d_2} {\textstyle \prod_{i=0}^{d_1+5d_2} (x_1 + 5x_2 + i \hbar)  \sum_{m \geq 0} \hbar^{-m-2} z_{d_1,d_2,m,X} }\\
& = (x_1+5x_2) \sum_{d_1,d_2 \geq 0} (y_1/\hbar)^{d_1}y_2^{d_2} \frac{\prod_{i=1}^{d_1+5d_2} ((x_1+5x_2)/\hbar + i)}{\prod_{i_1=1}^{d_1} (x_1/\hbar+i_1)^2 \prod_{i_2=1}^{d_2} (x_2/\hbar+i_2)^5}.
\end{aligned}
\end{equation}
Here, $y_i$ are parameters of the same degrees as the $q_i$. The quantum Lefschetz theorem says that 
\begin{equation} \label{eq:quantum-lefschetz}
 e^{-g - (k + l_1 x_1 + l_2 x_2)/\hbar} I =\big( \iota_! J\big)_{q_i = y_i e^{l_i}},
\end{equation}
where $\iota: Y \hookrightarrow X$ is the inclusion. Generally speaking, $g,k,l_1,l_2$ are functions of $(y_1,y_2)$, involving powers $y_1^{d_1}y_2^{d_2}$ with $d_1,d_2 \geq 0$, $(d_1,d_2) \neq (0,0)$. They are constrained by preserving degrees in \eqref{eq:quantum-lefschetz}, meaning that $g,l_1,l_2$ are functions of $y_2$ only, while $k$ is $y_1$ times a function of $y_2$. One can determine these functions completely in terms of $I$, by looking at specific terms in \eqref{eq:quantum-lefschetz}: 
\begin{equation}
\begin{aligned}
& e^g = \textstyle \sum_d  y_2^d \frac{(5d)!}{(d!)^5}, \\
& e^g k = \textstyle \sum_d y_1y_2^d  \frac{(5d)!}{(d!)^5}\, (5d+1), \\
& e^g l_1 =  \textstyle \sum_d y_2^d\frac{(5d)!}{(d!)^5}\, (1 + \frac12 + \cdots + \frac{1}{5d}), \\
& e^g l_2 = \textstyle \sum_d y_2^d \frac{(5d)!}{(d!)^5}\,5 (\frac{1}{d+1} + \cdots + \frac{1}{5d}),
\end{aligned}
\end{equation}
which leads to
\begin{align} 
& y_1 
= q_1 (1 - 274 q_2 - 50747 q_2^2 - 56404664 q_2^3 + \cdots), \\
& y_2 = q_2 - 770 q_2^2 + 171525 q_2^3 - 8162300 q_2^4 +\cdots.
\label{eq:quintic-mirror}
\end{align} 
Note that $g_2$ and $l_2$ agree with their counterparts for the quintic threefold. Hence, \eqref{eq:quintic-mirror} is the mirror map for the quintic threefold, as it appears in enumerative (and, by \cite{ganatra-perutz-sheridan15}, also in homological) mirror symmetry.
The outcome of applying \eqref{eq:quantum-lefschetz} and comparing that to \eqref{eq:read-off} is that
\begin{equation}
\begin{aligned}
& \psi = 1 - 496 q^5 - 335007 q^{10} - 365737016 q^{15} + \cdots, 
\\
& \eta =  940 q^4 +  + 856700 q^9 + 1097543500 q^{14} + \cdots,
\\
& z^{(2)} = 600 q^3 + 1508700 q^8 + 3364924200 q^{13} + \cdots.
\end{aligned}
\end{equation}
Then, \eqref{eq:schwarz-eq} 
has a solution
\begin{equation}
f = q - 154 q^6 - 13127 q^{11} - 17106304 q^{16} + \cdots
\end{equation}
This is a version of the mirror map \eqref{eq:quintic-mirror}, at least as far we can say given the available numerical data (more precisely, it is obtained from \eqref{eq:quintic-mirror} by substituting $q_2 = q^5$, and then considering $f = y_2^{1/5}$).

\begin{remark} \label{th:circular}
The computation above has an unsatisfactory circular nature. We apply quantum Lefschetz to $\cornerbar{E} \subset \bC P^1 \times \bC P^5$ in order to extract the coefficients for our Schwarzian equation, which then provides the function $f(q)$. On the other hand, quantum Lefschetz already contains the expression $q_2 = y_2 e^{l_2}$, whose inverse is the classical mirror map \eqref{eq:quintic-mirror}. As we have seen, that and $f$ appear to be essentially the same. There should be a reason for this purely in terms of the quantum Lefschetz formalism, and which applies generally to Calabi-Yau hypersurfaces (inside manifolds with $b_2 = 1$, because we're in a single-variable context). In other words, it should be true in general that Theorem \ref{th:1-variable} is compatible with classical mirror symmetry. However, this is beside the thrust of the present paper.
\end{remark}

\subsection{Consequences}
Theorem \ref{th:1-variable} has implications whose statements do not involve vanishing cycles or relative Fukaya categories. We don't want to engage into an extensive discussion, but we can show a little of what that means. Write $\scrF_\bK(\bar{M})$ for the Fukaya category of $\bar{M}$ in the ordinary sense \cite{fooo}, which is defined over the algebraically closed Novikov field 
\begin{equation} \label{eq:novikov-field}
\begin{aligned}
& \bK = \big\{ c = \textstyle \sum_r c_r q^r, \;:\; r \in \bR, \; c_r \in \bC, \text{ and for any $R \in \bR$,} \\[-.3em] & \qquad\qquad \qquad \text{there are only finitely many $c_r \neq 0$ with $r \leq R$}. \big\}
\end{aligned}
\end{equation}
Objects of $\scrF_{\bK}(\bar{M})$ are Lagrangian submanifolds $L \subset \bar{M}$ which are graded, {\em Spin}, and come with a bounding cochain (a solution of the curved Maurer-Cartan equation), as well as a flat complex line bundle.

\begin{corollary} \label{th:algebraic-closure}
Assume that Theorem \ref{th:1-variable} applies, and let $f$ be the function that appears there. Take the full $A_\infty$-subcategory of $\scrF_{\bK}(\bar{M})$ consisting of those $L \subset M$ such that $H^1(L) = 0$. Then, that subcategory is quasi-isomorphic to one which is defined over the algebraic closure $\overline{\bQ(f)} \subset \bK$.
\end{corollary}

Such $L$, when made into objects of the Fukaya category, have the property that their Floer cohomology
\begin{equation}
\mathit{HF}^*_{\bK}(L,L) = H^*(\mathit{hom}_{\scrF_{\bK}(\bar{M})}(L,L))
\end{equation}
vanishes in degree $1$ (in other words, they are rigid objects). Our basis of vanishing cycles gives rise to a full subcategory of $\scrF_{\bK}(\bar{M})$ which is quasi-isomorphic to $\scrB_q \otimes_{\bQ[[q]]} \bK$, and which split-generates the entire Fukaya category. The split-generation statement is a consequence of the long exact sequence in \cite{oh11} and the grading argument from \cite[Corollary 5.8]{seidel04}. With that in mind, Corollary \ref{th:algebraic-closure} follows from Theorem \ref{th:1-variable} and \cite[Lemma A.14]{seidel15}. (It might be interesting to study generalizations to non-rigid objects, where instead of a single object one would have to consider a versal family; however, we will not pursue that here). As promised, let's consider some more specific implications on the level of Floer cohomology.

\begin{application} 
Take $L$ as in Corollary \ref{th:algebraic-closure}. Suppose that $\mathit{HF}^*_\bK(L,L)$ is one-dimensional in some degrees $i_1$, $i_2$, $i_3 = i_1+i_2$. Choose generators $x_1$ and $x_2$ in the first two degrees, and assume that their Floer-theoretic product $x_1 \cdot x_2$ is nonzero. If we write
\begin{equation}
(x_2 \cdot x_1) = g \,(-1)^{i_1i_2} (x_1 \cdot x_2) \in \mathit{HF}_\bK^{i_1+i_2}(L,L),
\end{equation} 
then $g \in \bK$ is independent of the choice of generators (and in general nonzero, due to the noncommutativity of the product). Corollary \ref{th:algebraic-closure} implies that $g \in \overline{\bQ(f)}$.
\end{application}

\begin{application} 
Take a symplectic automorphism $\phi: \bar{M} \rightarrow \bar{M}$ (equipped with a grading). Consider the fixed point Floer cohomology of its iterates, denoted by $\mathit{HF}_\bK^*(\phi^r)$, with the pair-of-pants product. The degree zero part of these groups forms a graded algebra over $\bK$,
\begin{equation} \label{eq:phi-ring}
\bigoplus_{r \geq 0} \mathit{HF}^0_\bK(\phi^r).
\end{equation}
(See \cite{aldi-zaslow06, gross-siebert19} for appearances of specific instances of this algebra in mirror symmetry.) In the situation of Theorem \ref{th:1-variable}, it turns out that \eqref{eq:phi-ring} is defined over $\overline{\bQ(f)}$. To see why that is true, one uses the product $\bar{M}^2 = \bar{M} \times \bar{M}$, where the sign of the symplectic form on the second factor is reversed. The graph $\Gamma_\phi = \{(\phi(x),x)\}$ is a Lagrangian submanifold of $\bar{M}^2$. All graphs can be made into objects of the Fukaya category in a preferred way (up to quasi-isomorphism), and
\begin{equation}
\mathit{HF}_\bK^*(\phi) = \mathit{HF}^*_\bK(\Gamma_\phi, \Gamma_{\mathit{id}}).
\end{equation}
There are exact triangles in $\scrF_K(\bar{M}^2)$ associated to Dehn twists along any Lagrangian sphere $S$,
\begin{equation} \label{eq:surgery}
\xymatrix{\phi(S) \times S \ar[rr]^-{\text{canonical}} && \Gamma_{\phi} \ar[rr] && \Gamma_{\phi\tau_S}
\ar@/^1pc/[llll]^-{[1]}
 }
\end{equation}
(The version of this statement for monotone symplectic manifolds is in \cite{wehrheim-woodward15, mak-wu15} and, as pointed out in the latter reference, one can think of it as a mild generalization of the exact triangle for Lagrangian surgery \cite{fukaya-oh-ohta-ono07}, hence prove it by cobordisms \cite{biran-cornea13}; that approach generalizes to the Calabi-Yau case.) By repeated use of that, for $S = V_i$, and the same grading trick as in \cite[Corollary 5.8]{seidel04}, one can show that $\scrF_\bK(\bar{M}^2)$ is split-generated by products $V_i \times V_j$. In the situation of Theorem \ref{th:1-variable}, all graphs are rigid objects, since we have assumed that $H^1(\bar{M}) = 0$. Hence, Corollary \ref{th:algebraic-closure} implies that the full subcategory of all graphs is defined over $\bQ(f)$, which in particular yields the result we have stated.
\end{application}

\subsection{Structural properties\label{subsec:mirror}}
Underlying Theorem \ref{th:1-variable} is a more detailed statement. Let's first introduce the directed version of \eqref{eq:b-algebra},
\begin{equation} \label{eq:simple-directed}
\tilde{\scrC} = \bigoplus_k \bQ \tilde{e}_{V_k} \oplus \bigoplus_{i<j} \mathit{CF}^*(V_i,V_j).
\end{equation}
The second summand inherits an $A_\infty$-structure from that of $\scrB$, and we then add the $\tilde{e}_{V_k}$ artificially, as strict identity endomorphisms of each $V_k$. Next, consider a kind of double of $\tilde{\scrC}$, namely
\begin{equation} \label{eq:trivial-extension}
\tilde{\scrC} \oplus \tilde{\scrC}^\vee[1-n],
\end{equation}
where $\tilde\scrC^\vee$ is the dual graded vector space, which appears in \eqref{eq:trivial-extension} with its grading shifted up by $1-n$. Let's give \eqref{eq:trivial-extension} an additional grading (called weight grading for lack of a better name; it has nothing to do with Hodge theory), where the first summand has weight $0$, and the second weight $-1$.

\begin{theorem} \label{th:model}
In the situation of Theorem \ref{th:1-variable}, there is a quasi-isomorphic model for $\scrB_q$ which lives on the space \eqref{eq:trivial-extension}, and which has the following properties. All $A_\infty$-operations are nondecreasing with respect to weights; the weight $0$ part of those operations is the trivial extension algebra constructed from $\tilde\scrC$; and the part that increases weights by $r$ has coefficients which are polynomials of degree $\leq r$ in $f$. Moreover, in this context, knowledge of the weight $\leq 1$ part determines the entire $A_\infty$-structure, up to quasi-isomorphism.
\end{theorem}

The trivial extension algebra mentioned in the theorem is an $A_\infty$-structure one can put on \eqref{eq:trivial-extension}, using the $A_\infty$-structure on $\tilde{\scrC}$ and the structure (derived from that) of of $\tilde{\scrC}^\vee$ as the dual diagonal $A_\infty$-bimodule over $\scrC$. The trivial extension algebra structure has weight $0$ (is homogeneous with respect to weights). 

If one wants to use Theorem \ref{th:model} as a concrete tool to compute $\scrB_q$ in examples, the first step is to determine \eqref{eq:simple-directed}. This $A_\infty$-structure is defined inside the exact symplectic manifold $M$, and because of directedness, it has only finitely many nonzero $A_\infty$-operations; nevertheless, computing it is quite cumbersome, because anticanonical Lefschetz pencils tend to have a large number of vanishing cycles. However, one can also apply this method to obtain information about more degenerate (and hence simpler) situations. It is hard to explain this without digging further into the details, but we can outline what it means in one instance.

\begin{example}
Consider the Lefschetz fibration (with $(n+1)^n$ critical points)
\begin{equation} \label{eq:toric-fibration}
\begin{aligned}
& \bar{E} = \{(z,x) \in \bC \times \bC P^n \;:\; x_0^{n+1} + \cdots + x_n^{n+1} = z x_0\cdots x_n\}, \\
& \bar{p}: \bar{E} \longrightarrow \bC, \;\; \bar{p}(x,z) = z, \\
& \delta E = \bC \times \{x_0\cdots x_n = 0\} \subset \bar{E}.
\end{aligned}
\end{equation}
As usual, we set $E = \bar{E} \setminus E$ and $p = \bar{p}|E$. Temporarily departing from our usual notation, write $\bar{M}$ for some smooth fibre of $\bar{p}$, and $\delta M = \delta E \cap \bar{M}$. Even though $\delta E$, and therefore $\delta M$, are not smooth (they have normal crossing singularities), one can still define $\scrB$ and its deformation $\scrB_q$, see \cite{sheridan11b}. We will use $\bC$ instead of our usual $\bQ$ as a coefficient field in Floer cohomology, because that is more familiar from a mirror symmetry viewpoint. The directed subcategory \eqref{eq:simple-directed} is particularly simple to understand: writing $\sim$ for derived equivalence,
\begin{equation} \label{eq:mirror-projective}
\tilde{\scrC} \sim \scrD^b\mathit{Coh}_\Gamma(\bC P^n).
\end{equation}
Here, $\Gamma \iso (\bZ/n\!+\!1)^{n-1}$ acts on projective space by diagonal matrices whose entries are roots of unity, and which have determinant $1$. One considers the corresponding category of equivariant coherent sheaves, and its bounded derived category as a dg category (which is the notation $\scrD^b$ above). One can prove \eqref{eq:mirror-projective} by noticing that $p: E \rightarrow \bC$ is a fibrewise $\Gamma$-cover of the standard mirror of projective space, and then apply \cite{futaki-ueda14}; this in fact yields an equivariant exceptional collection on $\bC P^n$ which explicitly corresponds to a basis of vanishing cycles for $p$. As one consequence of \eqref{eq:mirror-projective}, one can use algebraic geometry to carry out certain computations in the category of $\tilde{\scrC}$-bimodules (we refer to Sections \ref{subsec:directed} and \ref{subsec:algebra} for the notation):
\begin{equation}
\label{eq:mirror-hh-computation}
\begin{aligned} 
& H^0(\mathit{hom}_{(\tilde{\scrC},\tilde{\scrC})}( \tilde{\scrC}^\vee[-n], \tilde{\scrC}))
\iso \mathit{Hom}^\Gamma_{\bC P^n \times \bC P^n}( \scrO_{\Delta} \otimes \scrK, \scrO_{\Delta})
\\ &
= H^0(\bC P^n, \scrK^{-1})^{\Gamma} 
= \bC \cdot y_0^{n+1} \oplus \bC \cdot y_1^{n+1} \oplus \cdots \oplus \bC \cdot y_n^{n+1} + \bC \cdot (y_0\cdots y_n).
\end{aligned}
\end{equation}
The computation takes place in the category of equivariant sheaves on $\bC P^n \times \bC P^n$; $\scrO_{\Delta}$ is the structure sheaf of the diagonal; $\scrK$ is the canonical bundle; and $y_0,\dots,y_n$ are the standard coordinates on $\bC^{n+1}$. In the same way, 
\begin{equation} \label{eq:negative-degree-ext}
\begin{aligned}
& H^*(\mathit{hom}_{(\tilde{\scrC},\tilde{\scrC})}( (\tilde{\scrC}^\vee[-n])^{\otimes r}, \tilde{\scrC}))
\iso \mathit{Ext}_{\bC P^n \times \bC P^n}^{\Gamma,*}( \scrO_{\Delta} \otimes \scrK^{\otimes r}, \scrO_{\Delta}) = 0
\\ & \qquad \qquad \qquad \qquad \qquad \qquad \qquad \qquad
\;\; \text{for $*<0$ and any $r \geq 1$.}
\end{aligned}
\end{equation}

Think of \eqref{eq:toric-fibration} as coming from the pencil of degree $(n+1)$ hypersurfaces on $\bC P^n$ generated by 
\begin{align}
& x_0^{n+1} + \cdots + x_n^{n+1}  = 0, \\
& x_0\cdots x_n = 0.  \label{eq:singular-fibre}
\end{align}
Because the latter hypersurface is singular, this is not an instance of our usual framework (as defined in Section \ref{subsec:setup}). However, one can slightly perturb \eqref{eq:singular-fibre} to obtain a Lefschetz pencil. For $n>2$, that pencil satisfies \eqref{eq:simply-connected} and \eqref{eq:as} for simple topological reasons; and the $n = 2$ case is the first example from Section \ref{subsec:examples}. The perturbation gives rise to new vanishing cycles, but the old vanishing cycles remain, and their Floer-theoretic structures are unaffected up to quasi-isomorphism. More precisely, one can choose a basis of vanishing cycles for the Lefschetz fibration which consists of the old $(n+1)^n$ ones plus others added on at the end; and then, the full subcategory formed by the old vanishing cycles, up to quasi-isomorphism, is the $\scrB$ (or $\scrB_q$) we have been considering. This implies that all the statements in Theorem \ref{th:model} hold for the situation of \eqref{eq:toric-fibration}, except possibly for the last one; however, that last one can be proved by inspecting its origin (Corollary \ref{th:concrete-classification-2}) and applying \eqref{eq:negative-degree-ext}.

The outcome of this is that, in order to determine $\scrB_q$, only one additional piece of information is necessary, namely its weight $1$ part. Moreover (see again Corollary \ref{th:concrete-classification-2}), that part is determined by an element $\epsilon_q \in H^0(\mathit{hom}_{(\tilde{\scrC},\tilde{\scrC})}(\tilde{\scrC}^\vee[-n], \tilde{\scrC}))[[q]]$. In terms of \eqref{eq:mirror-hh-computation} we can write
\begin{equation}
\epsilon_q = s_0 + f(q) s_1,
\end{equation}
where $s_0,s_1 \in H^0(\bC P^n, \scrK^{-1})^{\Gamma}$, and $f$ is a solution \eqref{eq:schwarz-eq}. Suppose that one knew two additional facts:
\begin{align} \label{eq:more-info-1}
& \text{$s_0$ is a nonzero multiple of $y_0\cdots y_n$;} \\
\label{eq:more-info-2}
& \text{$\textstyle s_1 = c\sum_k y_k^{n+1} + $ (a multiple of)$ y_0\cdots y_n$, with $c \neq 0$.}
\end{align}
The first of these is not hard to derive, at least in principle. The map $s_0$ can be read off by setting $q = 0$, which means that it is part of the structure of $\scrB$, hence inherits additional symmetries from $\pi_1(E) \iso \bZ^n$, which necessarily make it a multiple of $y_0\dots y_n$; and then, one could argue indirectly, namely that a relation between the localisation of $\tilde\scrC$ along $s_0$ and the wrapped Fukaya category of $E$ implies that $s_0$ is nonzero. For \eqref{eq:more-info-2}, one could again use additional symmetries to argue that all $y_k^{n+1}$ must come with the same coefficient. The remaining part would amount to showing that $\scrB_q$ is not $q$-independent up to quasi-isomorphism (for which there are again possible indirect approaches, which we will not attempt to discuss here).

Assuming \eqref{eq:more-info-1}, \eqref{eq:more-info-2} hold, Theorem \ref{th:model} together with the automatic split-generation result from \cite{perutz-sheridan15, ganatra16} and an easy algebro-geometric argument imply that 
 \begin{equation} \label{eq:hms}
\begin{aligned}
&\scrF_{\bK}(\bar{M}) \sim \scrD^b\mathit{Coh}_{\Gamma}(X_q), \\
& \qquad \qquad
X_q = \big\{ y_0^{n+1} + \cdots + y_n^{n+1} = w(y_0\cdots y_d), \;
w = a/f(q) + b\} \subset \bP^n(\bK)
\end{aligned}
\end{equation}
for some $a \in \bC^*$, $b \in \bC$. Here, $\sim$ stands for derived equivalence where the left side has been Karoubi-completed (perfect modules). The statement \eqref{eq:hms} is a form of homological mirror symmetry for Calabi-Yau hypersurfaces in projective space, which includes the determination of the mirror map up to two unknown constants $a,b$; this remaining ambiguity goes back to  \eqref{eq:conformal}. The outline of argument given here is incomplete, since our idea for \eqref{eq:more-info-1} is based on results not readily available in the literature, and we have given no indication of how to approach \eqref{eq:more-info-2}; the only purpose of this discussion was to give the reader an idea of how the amount of computation needed would compare to other approaches, such as the combination of \cite{sheridan11b} and \cite{ganatra-perutz-sheridan15}.
\end{example}

\subsection{Novikov rings\label{subsec:multivariable}}
We'll now discuss what to do if assumption \eqref{eq:as} fails. Take
\begin{equation}
\label{eq:h-lattice}
H = H_2(\cornerbar{E};\bZ)/\mathit{torsion} \iso \bZ^{r+2}, \;\; r = b_2(\cornerbar{E}) - 2.
\end{equation}
The intersection pairings with $\bar{M}$ and $\delta E|$ give homomorphisms $H \rightarrow \bZ$. We use them to define a graded Novikov ring $\Lambda$:
\begin{equation} \label{eq:x-series}
\begin{aligned} 
& \Lambda^i = \big\{ x = \textstyle \sum_A x_A q^A \;:\; A \in H \text { with } 2(A \cdot \bar{M}) = i; \; x_A \in \bQ; \; \text{ and for any $C \in \bZ$,} \\[-.3em] & \qquad\qquad \qquad \qquad\text{there are only finitely many $x_A \neq 0$ with $A \cdot \delta E| \leq C$}. \big\}
\end{aligned}
\end{equation}
One can use the same intersection numbers to collapse $\Lambda$ back to a simpler graded ring:
\begin{equation} \label{eq:specialize}
\begin{aligned}
& K: \Lambda \longrightarrow \bQ[t,t^{-1}]((q)), \; |t| = 2, \\
& K(q^A) = t^{A \cdot \bar{M}} q^{A \cdot \delta E|},
\end{aligned}
\end{equation}
Take a class $z \in H^2(\cornerbar{E};\Lambda^j)$, which is an expression as in \eqref{eq:x-series} but with coefficients $z_A \in H^2(\cornerbar{E};\bQ)$. One can associate to it a derivation of degree $j$,
\begin{equation} \label{eq:z-derivation}
\begin{aligned}
& \partial_z: \Lambda^i \longrightarrow \Lambda^{i+j}, \\
& \partial_z(q^{A_2}) = \textstyle \sum_{A_1} (z_{A_1} \cdot A_2) q^{A_1+A_2}.
\end{aligned}
\end{equation}

We find it convenient to introduce more explicitly the filtration with respect to which $\Lambda$ is complete. Namely, set
\begin{align}
&
H_{\geq k} = \{A \in H \;:\; A \cdot \delta E| \geq k\}, \\
&
H_k = \{A \in H\;:\; A \cdot \delta E| = k\},
\end{align}
and consider the graded subspace $\Lambda_{\geq k} \subset \Lambda$ consisting of series \eqref{eq:x-series} where the sum is only over $A \in H_{\geq k}$. Clearly,
\begin{equation}
\begin{aligned}
& \Lambda_{\geq k}/\Lambda_{\geq k+1} = \bQ[H_k], \\
& K(\Lambda_{\geq k}) \subset q^k\bQ[t,t^{-1}][[q]].
\end{aligned}
\end{equation}
From now on, we will assume (mostly for convenience) that 
\begin{equation}
\label{eq:connected} \text{$\delta M$ is connected.}
\end{equation}
As a consequence, there is a distinguished homology class
\begin{equation} \label{eq:a-star-class}
A_* \in H_2(\cornerbar{E};\bZ), \text{ represented by the constant sections in $\delta E| \iso \bC| \times \delta M$.}
\end{equation}
When regarded modulo torsion, this is an element $A_* \in H_{-1}$ of degree $2$. 
Consider more specifically the derivations \eqref{eq:z-derivation} associated to classes 
\begin{equation} \label{eq:z-field}
z \in [\delta E|] q^{A_*} + H^2(\cornerbar{E};\Lambda_{\geq 0}^2) \subset
H^2(\cornerbar{E};\Lambda_{\geq -1}^2).
\end{equation}
These satisfy $\partial_z(\Lambda^i_{\geq 0}) \subset \Lambda^{i+2}_{\geq 0}$. That is a special case ($k = 0$) of the observation that
\begin{equation} \label{eq:differential-term}
f \in \Lambda_{\geq k}\;\; \Longrightarrow\;\; \partial_z f \in k q^{A_*} f + \Lambda_{\geq k}.
\end{equation}
An elementary argument shows:

\begin{lemma} \label{th:z-function}
Any $f_0 \in \bQ[H_0]$ has a unique extension $f \in f_0 + \Lambda_{\geq 1}$ satisfying $\partial_z f = 0$. (And that extension operation preserves degrees.)
\end{lemma}

There is also an associated Schwarzian differential operator $S_z$, which is \eqref{eq:schwarz} with the derivative replaced by $\partial_z$. More precisely, this operator has the form
\begin{equation} \label{eq:z-schwarzian}
S_z: \{f \in \Lambda_{\geq 1}^i \text{ such that $q^{A_*}f \in \Lambda_{\geq 0}$ is invertible}\} \longrightarrow \Lambda_{\geq 0}^4.
\end{equation}
Applying the Schwarzian to such $f$ makes sense because $\partial_z f \in q^{A_*}f + \Lambda_{\geq 1}$ is invertible. $S_zf$ is unchanged under transformations
\begin{equation} \label{eq:schwarz-transformations-2}
f \in \Lambda^i_{\geq 1} \longmapsto \frac{f}{a+bf} \in \Lambda^{i-k}_{\geq 1} \;\; \text{with $a \in \Lambda_{\geq 0}^{k}$ invertible, 
$b \in \Lambda_{\geq 0}^{k-i}$, $\partial_z a = \partial_z b = 0$.}
\end{equation}
Each orbit of the transformations \eqref{eq:schwarz-transformations-2} contains a unique $f \in q^{-A_*} + \Lambda_{\geq 3}^{-2}$. Moreover,
\begin{equation}
S_z(f+\tilde{f}) \in S_zf - \partial_z^3 \tilde{f} + \Lambda_{\geq k-2}^4 \;\; 
\text{ for $f \in q^{-A_*} + \Lambda_{\geq 3}^{-2}$ and $\tilde{f} \in \Lambda_{\geq k}^{-2}$, $k \geq 3$.}
\end{equation}
Using that, one can solve Schwarzian equations order to order in $\Lambda_{\geq k}$. The outcome is:

\begin{lemma} 
Given $g \in \Lambda_{\geq 0}^4$ and any even $i$, there is a solution of $S_zf = g$ with $|f| = i$.
Moreover, any two such solutions are related by \eqref{eq:schwarz-transformations-2}.
\end{lemma}

\subsection{The general result\label{subsec:multivariable-2}}
We use the same enumerative invariants as before, but in a version which separates out the contribution from different homology classes. To avoid confusion, we use the subscript $\Lambda$ for the new version, which is
\begin{equation}
z^{(k)}_\Lambda = \sum_{A \cdot \bar{M} = k} z_A q^A \in H^{4-2k}(\cornerbar{E};\Lambda^{2k}).
\end{equation}
The analogues of \eqref{eq:z1}, \eqref{eq:z2} say that
\begin{align}
\label{eq:refined-z1}
& z^{(1)}_\Lambda \in [\delta E|] q^{A_*} + H^2(\cornerbar{E}; \Lambda^2_{\geq 0}), \\
\label{eq:refined-z2}
& z^{(2)}_\Lambda \in H^0(\cornerbar{E};\Lambda^4_{\geq 0}) = \Lambda^4_{\geq 0}.
\end{align}
We will use the first invariant to define a derivation \eqref{eq:z-derivation}, and the second one as the inhomogeneous term in the associated Schwarzian equation
\begin{equation} \label{eq:weird-schwarz-0}
S_{z^{(1)}_\Lambda} f_\Lambda + 8z^{(2)}_\Lambda = 0.
\end{equation}
To make the outcome more explicit, choose $A_1,\dots,A_r \in H$, which form a basis for the subspace of classes $A$ with $A \cdot \bar{M} = A \cdot \delta E| = 0$. By Lemma \ref{th:z-function}, there are unique 
\begin{equation} \label{eq:g-lambda-k}
g_{\Lambda,k} \in q^{A_k} + \Lambda_{\geq 1}^0 \subset (\Lambda_{\geq 0}^0)^\times
, \;\; \partial_{z^{(1)}_\Lambda} g_{\Lambda,k} = 0. 
\end{equation}
Take those functions as well as a solution of \eqref{eq:weird-schwarz-0} of degree $0$, and write
\begin{align}
\label{eq:specialize-2}
& f = K(f_{\Lambda}) \in \bQ[[q]], \;\; f(0) = 0, \; f'(0) \neq 0, \\
\label{eq:specialize-fk}
& g_k = K(g_{\Lambda,k}) \in \bQ[[q]]^{\times}.
\end{align}

\begin{theorem} \label{th:main}
Suppose that \eqref{eq:connected} and \eqref{eq:simply-connected} hold. Then, there is a quasi-isomorphic model for $\scrB_q$, which is defined over $\bQ[f,g_1^{\pm 1},\dots,g_r^{\pm 1}] \subset \bQ[[q]]$, where $f$ and $g_k$ are from \eqref{eq:specialize-2}, \eqref{eq:specialize-fk}.
\end{theorem}

Here, the notation $\bQ[f,g_1^{\pm 1},\dots,g_r^{\pm 1}]$ just stands for the subring of $\bQ[[q]]$ which those functions generate; they are not algebraically independent in general, so the structure of the subring will vary from case to case. (There is  a corresponding analogue of Theorem \ref{th:model}, which we will not state here.)

\begin{remark} \label{th:tangent}
We will now explain how this is compatible with our previous Theorem \ref{th:1-variable}. Suppose that Assumption \eqref{eq:as} holds. In that case, our derivation fits into a commutative diagram
\begin{equation} \label{eq:kappa-diagram}
\xymatrix{
\Lambda_{\geq 0}^i \ar[d]_-{\partial_{z^{(1)}_\Lambda}} \ar[rr]^-{K} && 
t^{i/2}\bQ[[q]] \ar[d]^-{\partial_{q,t}} \\
\Lambda_{\geq 0}^{i+2} \ar[rr]_-{K} &&
t^{i/2+1} \bQ[[q]],
}
\end{equation}
where, with the same functions as in \eqref{eq:write-z1}, 
\begin{equation}
\partial_{q,t} = t\psi^{-1}( \partial_q + (i/2) \eta).
\end{equation}
As a consequence, we have $\partial_q g_k = 0$, which because of $g_k(0) = 1$ implies that these functions are constant equal to $1$. Moreover, $f$ is a solution of
\begin{equation} \label{eq:schwarz-q-t}
S_{q,t} f + 8 z^{(2)} t^2 = 0,
\end{equation}
where $S_{q,t}$ is the Schwarzian operator on $\bQ[t,t^{-1}][[q]]$ associated to $\partial_{q,t}$; concretely,
\begin{equation} \label{eq:compare-schwarzians}
S_{q,t}f = t^2\psi^{-2}\Big(S_q f + \partial_q\Big(\eta - \frac{\psi'}{\psi}\Big) + \half\Big(\eta - \frac{\psi'}{\psi}\Big)^2 \Big)\;\; \text{ for $f \in \bQ[[q]]^\times$.}
\end{equation}
Hence, \eqref{eq:schwarz-q-t} is exactly our original \eqref{eq:schwarz-eq} (thereby providing a more conceptual reformulation of that equation). 
\end{remark}

It may be helpful to think in algebro-geometric terms. $T_0 = \mathit{Spec}(\bQ[H_0])$ is an algebraic torus over $\bQ$. It carries an action of the multiplicative group $\bG_m$, corresponding to the grading of that ring. Moreover, the torus has a distinguished origin, and one can consider the $\bG_m$-orbit $O_0$ through the origin. Then, $T_{\geq 0} = \mathit{Spf}(\Lambda_{\geq 0})$ is a formal thickening of that torus, still carrying the group action. Inside it, \eqref{eq:specialize} describes a formal thickening of the previously considered orbit, which we'll denote by $O_{\geq 0} \subset T_{\geq 0}$. The operators $\partial_z$ correspond to vector fields on $T_{\geq 0}$ which are transverse to $T_0$ (and contracted by the $\bG_m$-action); that explains the structure of the functions annihilated by $\partial_z$ (Lemma \ref{th:z-function}). In the situation from Remark \ref{th:tangent}, the vector field is tangent to $O_{\geq 0}$, allowing us to restrict all computations to that subspace.

\subsection{Structure of the paper} 
For the vast majority of the paper, we focus on Theorem \ref{th:1-variable}. Section \ref{sec:outline} describes its proof, which draws heavily on earlier papers in the series. After that, two ingredients in the proof are examined in more detail. One is the purely algebraic part, which is not new, but receives a streamlined presentation in Sections \ref{sec:noncommutative-divisors}--\ref{sec:connections} (in the first of those sections, the exposition goes beyond the minimum required for our applications). The second ingredient is Hamiltonian Floer cohomology relative to a symplectic divisor, which we gradually build towards in Sections \ref{sec:mclean}--\ref{sec:apply-to-lefschetz}. Here, the aim is to lift certain results from \cite{seidel16} to the relative context. Finally, Section \ref{sec:general} introduces $B$-fields, and uses them to adapt the previous argument to the more general situation of Theorem \ref{th:main}.

{\em Acknowledgments.} I'd like to thank the audiences that gracefully endured me droning on about this project for a decade and a half. Financial support was received by: the Simons Foundation, through a Simons Investigator award and the Simons Collaboration for Homological Mirror Symmetry; and the NSF, through award DMS-1904997.

\section{\label{sec:outline}The argument}
The aim of this section is to make the statement from Theorem \ref{th:1-variable} more precise, and then describe its proof, in particular explaining where the differential equation \eqref{eq:2nd-order} comes in. Rather than proceeding in logical order, we'll start with the Floer-theoretic structures which are superficially closest to the original statement, and then excavate the underlying strata, in a sense going backwards.

\subsection{Warm-up exercises\label{subsec:directed}}
Let's start by revisiting \eqref{eq:b-algebra}, which is defined using Floer theory in the fibre $M$. One can arrange (by a suitable choice of perturbation) that 
\begin{equation} \label{eq:cf-vk}
\mathit{CF}^*(V_k,V_k) = \bQ e_{V_k} \oplus \bQ t_{V_k},
\end{equation}
where $e_{V_k}$ has degree $0$ and $t_{V_k}$ has degree $(n-1)$. Set
\begin{equation} \label{eq:b-directed}
\scrC = \bigoplus_k \bQ e_{V_k} \oplus \bigoplus_{i<j} \mathit{CF}^*(V_i,V_j) \subset \scrB.
\end{equation}

\begin{lemma} \label{th:subalgebra}
$\scrC$ is an $A_\infty$-subalgebra of $\scrB$.
\end{lemma}

\begin{proof}
All we have to check is that $\mu^d_{\scrB}(e_{V_k},\dots,e_{V_k})$ is zero for $d \neq 2$, and a multiple of $e_{V_k}$ for $d = 2$. This is true for degree reasons except for one special case ($d = 1$ and $n-1 = 1$). In that remaining case, $V_k$ is a simple closed curve with Maslov class zero on a punctured torus, and the desired property follows from the fact that $\mathit{HF}^*(V_k,V_k)$ is nonzero.
\end{proof}

In \eqref{eq:simple-directed} we considered a slightly different algebra $\tilde{\scrC}$, obtained from the second summand in \eqref{eq:b-directed} by adjoining strict units. A priori, it may appear that $\tilde{\scrC}$ has fewer nontrivial $A_\infty$-operations than $\scrC$, but in fact:

\begin{lemma} \label{th:directed}
$\tilde{\scrC}$ is isomorphic to $\scrC$.
\end{lemma}

\begin{proof}
For general reasons \cite[Lemma 2.1]{seidel04}, there is a strictly unital $\scrC^{(1)}$, living on the same graded space as $\scrC$, and an $A_\infty$-isomorphism
\begin{equation} \label{eq:make-unital}
\scrC^{(1)} \longrightarrow \scrC. 
\end{equation}
(Readers are hereby alerted to a mistake in the statement of \cite[Lemma 2.1]{seidel04}: the result requires the additional assumption that the identity elements are nonzero in cohomology, used in the first step in the proof; see also Lemma \ref{th:unify} later on. That assumption is clearly satisfied here.) Take the inverse of \eqref{eq:make-unital}, restrict it to the second summand in \eqref{eq:b-directed}, and then extend it again (in the unique way) to a strictly unital $A_\infty$-isomorphism
\begin{equation} \label{eq:make-unital-2}
\tilde{\scrC} \longrightarrow \scrC^{(1)}.
\end{equation}
The composition of \eqref{eq:make-unital} and \eqref{eq:make-unital-2} yields the desired statement. 
\end{proof}
%

As a consequence of Lemma \ref{th:subalgebra}, one can regard $\scrB$ and its quotient $\scrP = (\scrB/\scrC)[-1]$ (we have applied an upwards shift in the grading, for consistency with the notation used in the rest of the paper) as $\scrC$-bimodules. It is helpful to introduce a bit of terminology, which applies to any context where one has an $A_\infty$-subalgebra. Namely, let's introduce the weight filtration and its associated graded space:
\begin{equation} \label{eq:weights}
\begin{aligned}
& W^{-1}\scrB = \scrB, \quad W^0\scrB = \scrC, \quad W^1\scrB = 0, \\
& GrW\scrB = W^0\scrB \oplus (\scrB/W^0\scrB) = \scrC \oplus \scrP[1].
\end{aligned}
\end{equation}
The graded space inherits operations
\begin{equation}
\begin{aligned}
& W^0(\scrB^{\otimes d}) = \scrC^{\otimes d} \longrightarrow W^0\scrB = \scrC[2-d],
\\ 
& W^{-1}(\scrB^{\otimes d})/W^0(\scrB^{\otimes d})[-1] = \bigoplus_{r+s+1=d}
\scrC^{\otimes s} \otimes \scrP \otimes \scrC^{\otimes r} \longrightarrow (\scrB/W^0\scrB)[1-d] = \scrP[2-d].
\end{aligned}
\end{equation}
These describe the structure of $\scrC$ as an $A_\infty$-algebra, together with that of $\scrP$ as a $\scrC$-bimodule.
Of course, $\scrB$ contains more information than its graded space. For instance, from the definition of $\scrP$ as a quotient, we get an exact triangle of $\scrC$-bimodules
\begin{equation} \label{eq:rho-triangle}
\xymatrix{
\scrP \ar[rr] &&
\scrC \ar[rr]_-{\text{inclusion}} && 
\scrB \ar@/_1pc/[llll]_-{\text{projection } [1]}
}
\end{equation}
As background, we'd like to remind the reader that $A_\infty$-bimodules over $\scrC$ form a differential graded category, which we will denote by $(\scrC,\scrC)$ (see Section \ref{subsec:algebra} for a more precise account of our conventions). In \eqref{eq:rho-triangle} we are talking about chain homotopy classes of bimodule homomorphisms, which means about morphisms in the associated cohomology level (triangulated) category. 

\begin{lemma} \label{th:delta-triangle}
There is an exact triangle of $\scrC$-bimodules, involving the dual diagonal bimodule $\scrC^\vee$,
\begin{equation} \label{eq:delta-triangle}
\xymatrix{
\scrC^\vee[-n] \ar[rr]_-{\epsilon} &&
\scrC \ar[rr]_-{\mathrm{inclusion}} && 
\scrB \ar@/_1pc/[llll]_-{[1]}
}
\end{equation}
\end{lemma}

Unlike the previous algebraic observations, this is specific to Floer theory, and in particular, the map $\epsilon$ has a geometric origin. Further discussion is postponed to Section \ref{subsec:fukaya-of-lefschetz}. By comparing \eqref{eq:delta-triangle} and \eqref{eq:rho-triangle}, one sees that $\scrP$ is quasi-isomorphic to $\scrC^\vee[-n]$; and moreover, that quasi-isomorphism relates $\epsilon$ to the corresponding map in \eqref{eq:rho-triangle}, which was defined as the connecting homomorphism of the short exact sequence involving those bimodules. 

\begin{remark}
Readers looking ahead to Section \ref{subsec:fukaya-of-lefschetz} may wonder why we involve Fukaya categories of Lefschetz fibrations in the proof of Lemma \ref{th:delta-triangle}, given that there is a more self-contained approach based on the (weak) Calabi-Yau nature of $\scrF(M)$. The immediate answer is that the Lefschetz fibrations approach provides additional information about $\epsilon$, which is crucial for our argument. From a wider perspective the two approaches should be linked, by a relation between the structure of a noncommutative anticanonical divisor on the Fukaya category of the Lefschetz fibration, and the Calabi-Yau structure on the Fukaya category of its fibre. However, that is beyond the scope of the theory developed here.
\end{remark}

For $r \geq 1$, denote by $(\scrC^\vee[-n])^{\otimes_{\scrC} r}$ the $r$-th tensor power of the bimodule $\scrC^\vee[-n]$. The following result  is proved in \cite{seidel21}. The argument there combines general properties of Fukaya categories of Lefschetz fibrations (injectivity of twisted open-closed string maps) with the more specific geometry of Lefschetz pencils (which we use here for the first time in our discussion). 

\begin{proposition} \label{th:no-negative-degree}
Up to homotopy, there are no nontrivial bimodule homomorphisms from $(\scrC^\vee[-n])^{\otimes_{\scrC} r} \rightarrow \scrC$ of negative degree. In formulae,
\begin{equation} \label{eq:no-negative-degree}
H^*\big(\mathit{hom}_{(\scrC,\scrC)}((\scrC^\vee[-n])^{\otimes_{\scrC} r}, \scrC)\big) = 0
\;\; \text{ for all $r \geq 1$ and $\ast < 0$.}
\end{equation}
\end{proposition}

Because $\scrC$ is proper, and homologically smooth (this is slightly easier to see for $\tilde\scrC$, which is a directed $A_\infty$-algebra in the terminology from \cite{seidel04}, but the two statements are equivalent), the dual diagonal bimodule is invertible with respect to tensor product. As a consequence, tensoring with $\scrC^\vee[-n]$ is a cohomologically full and faithful operation on the category $(\scrC,\scrC)$. In view of that, \eqref{eq:no-negative-degree} implies
\begin{equation} \label{eq:no-negative-degree-2}
H^*\big(\mathit{hom}_{(\scrC,\scrC)}((\scrC^\vee[-n])^{\otimes_{\scrC} r}, \scrC^\vee[-n])\big) = 0 \;\; \text{ for all $r \geq 2$ and $\ast < 0$.}
\end{equation}
Using that and purely algebraic arguments, specifically Corollary \ref{th:concrete-classification} (applied to the inclusion $\scrC \hookrightarrow \scrB$), one gets the following immediate consequence:

\begin{proposition} \label{th:1-variable-1}
$\scrB$ is determined, up to quasi-isomorphism, by $\scrC$ and the bimodule homomorphism $\epsilon$ from \eqref{eq:delta-triangle}. More precisely, what matters is the homotopy class $[\epsilon]$, up to composition with automorphisms of $\scrC^\vee$ on the right.
\end{proposition}


\subsection{The $q$-deformation\label{subsec:directed-2}}
The relative Fukaya category $\scrF_q(\bar{M})$ is an $A_\infty$-deformation of $\scrF(M)$. We use the definition of that category from \cite{sheridan11b, perutz-sheridan20, seidel20}. Very briefly, the construction is achieved by letting the almost complex structures and inhomogeneous terms for perturbed pseudo-holomorphic discs $u: S \rightarrow \bar{M}$ depend on the position of the points $u^{-1}(\delta M) \subset S$. Restricting to our basis of vanishing cycles gives an $A_\infty$-deformation $\scrB_q$ of $\scrB$; we use this definition except in the lowest dimension $n-1 = 1$, which has a specific caveat.

Our concern in that dimension is with the way in which the vanishing cycles become objects of the relative Fukaya category. To put this in a larger context, let's start with a formally enlarged version of the relative Fukaya category, denoted by $\scrF_q(\bar{M})^{(0)}$. Objects are pairs $(L,\alpha)$, where $L$ is an object of $\scrF(M)$ and $\alpha \in q\mathit{CF}^1(L,L)[[q]]$ can be arbitrary. One uses those $\alpha$ to deform all the $A_\infty$-operations. In particular, the curvature term for $(L,\alpha)$ is
\begin{equation} \label{eq:mc-alpha}
\mu^0_{\scrF_q(\bar{M})^{(0)}} = \mu^0_{\scrF_q(\bar{M})} + \mu^1_{\scrF_q(\bar{M})}(\alpha) + \mu^2_{\scrF_q(\bar{M})}(\alpha,\alpha) + \cdots \in q\mathit{CF}^2(L,L) [[q]].
\end{equation}
To get an $A_\infty$-subcategory without curvature, one uses $(L,\alpha)$ where $\alpha$ is a Maurer-Cartan element (also called bounding cochain), which means that \eqref{eq:mc-alpha} is zero. For $L = V_k$, our convention \eqref{eq:cf-vk} means that \eqref{eq:mc-alpha} vanishes for any $\alpha$. While it may seem natural to take $\alpha = 0$, which means to define $\scrB_q$ as a subcategory of $\scrF_q(\bar{M})$ (as we've done in higher dimensions), the outcome is not invariant under changing the auxiliary choices (almost complex structures, inhomogeneous terms) involved. Different choices give rise to $A_\infty$-categories $\scrF_q(\bar{M})$ related by curved $A_\infty$-functors; and the induced $A_\infty$-functor between categories $\scrF_q(\bar{M})^{(0)}$ does not preserve the sub-class of objects with $\alpha = 0$. Instead, we will define $\scrB_q$ as a subcategory of $\scrF_q(\bar{M})^{(0})$ with respect to a particular choice of $(\alpha_1,\dots,\alpha_m)$ for the vanishing cycles $(V_1,\dots,V_m)$, which will be specified later on (see the proof of Lemma \ref{th:quasi-isomorphic-to-directed}). In higher dimensions the issue does not arise, since there, $\mathit{CF}^1(V_k,V_k) = 0$.

\begin{remark} \label{th:workaround-1}
There is a classical definition of $\scrF_q(\bar{M})$ in this dimension ($n - 1 = 1$), still counting perturbed pseudo-holomorphic discs $u$ with powers $q^{u \cdot \delta M}$, but where the Cauchy-Riemann equation is independent of $u^{-1}(\delta M) \subset S$. It yields a category, here denoted by $\scrF_q(\bar{M})^{(1)}$, with no curvature term. One could use that to define $\scrB_q$, but then that the entire subsequent discussion would have to lay down a parallel track specifically for that dimension.

There is also an algebraic trick that relies on that alternative construction to single out a class of Maurer-Cartan elements in our original definition of the relative Fukaya category. Namely, one can define an $A_\infty$-bimodule on which $\scrF_q(\bar{M})^{(0)}$ acts on one side, and $\scrF_q(\bar{M})^{(1)}$ on the other. For concreteness, let's write $L^{(0)} = (L,\alpha)$ for a Lagrangian submanifold viewed as an object of $\scrF_q(\bar{M})^{(0)}$, and $L^{(1)}$ for the same submanifold as an object of $\scrF_q(\bar{M})^{(1)}$. Part of the bimodule structure are chain complexes $\mathit{CF}^*(L^{(0)},L^{(1)})[[q]]$ (note that here, we use a single underlying Lagrangian). By asking that the resulting cohomology should not be $q$-torsion, we can single out a particular equivalence class of $\alpha$ for each $L$. It is in principle possible to show that this is the class which arises in our considerations. However, as before, this would require building at least parts of a parallel theory for $\scrF_q(\bar{M})^{(1)}$. 
\end{remark}

The next higher dimension also seems potentially problematic, since it is the only one in which the curvature term for $V_k$ could be nonzero. However, that is not actually an issue:

\begin{lemma} \label{th:2-curvature}
Take $n-1 = 2$. Then, $V_k$ as an object of $\scrF_q(\bar{M})$ has vanishing curvature term.
\end{lemma}

\begin{proof}
First, we show that the desired property is independent of the choices made in constructing $\scrF_q(\bar{M})$, assuming that these satisfy \eqref{eq:cf-vk}. Suppose that we have two different versions, denoted by $\scrF_q(\bar{M})^{(+)}$ and $\scrF_q(\bar{M})^{(-)}$. They are related by a curved $A_\infty$-functor $\scrG_q$, which is the identity on objects and is a filtered quasi-isomorphism. For any object $L$, this means we have
\begin{equation}
\begin{aligned}
& \scrG_q^0 \in q\mathit{CF}^1(L,L)^{(-)}[[q]], \\
& \scrG_q^1: \mathit{CF}^*(L,L)^{(+)} \longrightarrow \mathit{CF}^*(L,L)^{(-)}[[q]], \\
& \dots
\end{aligned}
\end{equation}
and the first $A_\infty$-functor equation is
\begin{equation} \label{eq:mu0-relation}
\begin{aligned}
& \mu^{0}_{\scrF_q(\bar{M})^{(-)}} + \mu^{1}_{\scrF_q(\bar{M})^{(-)}}(\scrG_q^0) + 
\mu^2_{\scrF_q(\bar{M})^{(-)}}(\scrG_q^0,\scrG_q^0) + \cdots \\ & \qquad \qquad = 
\scrG_q^1(\mu^0_{\scrF_q(\bar{M})^{(+)}}) \in q\mathit{CF}^2(L,L)^{(-)}[[q]].
\end{aligned}
\end{equation}
Setting $L = V_k$, we find that $\scrG_q^0$ must be zero (for degree reasons), and that $\scrG_q^1$ defines an isomorphism between the endomorphism groups in both categories (because it's a filtered quasi-isomorphism and the differentials are zero). Hence, \eqref{eq:mu0-relation} implies that if the curvature term vanishes for one choice, it also does for the other one.

Fix some almost complex structure $J$ which makes $\delta M$ into an almost complex submanifold, and for which there are no holomorphic discs bounding $V_k$ (in this dimension, that is a generic property). When defining the curved $A_\infty$-structure for $V_k$ as an object of $\scrF_q(\bar{M})$, we may keep our perturbed equations close to the classical $J$-holomorphic curve equation (this includes choosing a small Hamiltonian when forming the Floer complex). Then, a Gromov compactness argument shows that $\mu^0$ has only terms of order $q^N$ and higher, for some $N$ (which can be made arbitrarily large, but depends on the choice of perturbation). Now, let's consider two versions of $\scrF_q(\bar{M})$. The same algebraic argument as before shows that if for one choice, the curvature term is of order $q^N$, then the same holds for the other choice. Since that argument can be carried out for arbitrary $N$, the curvature term must be zero.
\end{proof}

\begin{remark} \label{th:workaround-2}
One can avoid the Gromov compactness argument by a workaround in the spirit of Remark \ref{th:workaround-1}. Namely, there is a version of the relative Fukaya category in this dimension, defined in \cite{seidel03b}, which has a priori vanishing curvature terms. Let's denote that by $\scrF_q(\bar{M})^{(1)}$, and write $L^{(1)}$ for $L$ considered as an object of that category. One can also define a bimodule on which $\scrF_q(\bar{M})$ acts on one side, and $\scrF_q(\bar{M})^{(1)}$ on the other. Part of the bimodule structure is a map
\begin{equation} 
\label{eq:cf-right}
\mathit{CF}^0(L,L^{(1)}) \otimes \mathit{CF}^2(L,L) \longrightarrow \mathit{CF}^2(L,L^{(1)})[[q]].
\end{equation}
Set $L = V_k$, and arrange that all Floer groups that appear above satisfy \eqref{eq:cf-vk}. Then, the constant term of \eqref{eq:cf-right} is nonzero; this is just part of the fact that for $q = 0$, we are considering two equivalent versions of $\scrF(M)$. Moreover, given that the differential on $\mathit{CF}^2(L,L^{(1)})$ is zero, and that $\scrF_q(\bar{M})^{(1)}$ has vanishing curvature term, the $A_\infty$-bimodule equations imply that inserting $\mu^0_{\bar\scrF_q(M)} \in q\mathit{CF}^2(L,L)[[q]]$ into \eqref{eq:cf-right} yields zero. That of course implies that this element must itself be zero.
\end{remark}

Take the same subspace $\scrC \subset \scrB$ as before, but now restrict the operations of $\scrB_q$ to it. We denote the outcome by $\scrC_q$. The generalization of Lemma \ref{th:subalgebra} is:

\begin{lemma} \label{th:subalgebra-2}
$\scrC_q$ is an $A_\infty$-subalgebra of $\scrB_q$.
\end{lemma}

\begin{proof}
As before, this mostly follows from degree reasons. The additional ingredients are Lemma \ref{th:2-curvature} and the following fact:
\begin{equation} \label{eq:mu1q-e}
\mu^1_{\scrB_q}(e_{V_k}) = 0 \;\; \text{ for $n-1 = 1$.}
\end{equation}
This is the case where $V_k \subset \bar{M}$ is a simple closed curve on a torus, which has Maslov index zero and therefore represents a primitive homology class, and we have chosen some $\alpha_k$ to make it into an object of $\scrF(\bar{M})^{(0)}$. Denote cohomology level morphisms in that category by $\mathit{HF}^*_q(\cdot,\cdot)$, for the sake of brevity. If we take another such curve $L$ which intersects $V_k$ transversally in a point, then $\mathit{HF}^*_q((V_k,\alpha_k),L) \iso \bQ[[q]]$. Because of the ring structure, this means that $\mathit{HF}^*((V_k,\alpha_k),(V_k,\alpha_k))$ can't be $q$-torsion, and in view of \eqref{eq:cf-vk}, that implies \eqref{eq:mu1q-e}
\end{proof}

Let $\tilde{\scrC}_q$ be the $A_\infty$-deformation of $\tilde{\scrC}$ obtained by taking the restriction of the operations on $\scrB_q$ to the second summand in \eqref{eq:b-directed}, and again adjoining strict units. 

\begin{lemma} \label{th:directed-2}
$\tilde{\scrC}_q$ is isomorphic to $\scrC_q$.
\end{lemma}

\begin{proof}
This is an extension of the proof of Lemma \ref{th:directed}. We again start by constructing \eqref{eq:make-unital}. Given a strictly unital $A_\infty$-algebra, the deformation theory to a curved $A_\infty$-structure over $\bQ[[q]]$ remains the same (in the sense of a bijection of equivalence classes of deformations) if we require that strict unitality be retained. This follows from the quasi-isomorphism between the full and reduced Hochschild complex, by general obstruction theory methods. Hence, in our case one can deform $\scrC^{(1)}$ to some $\scrC^{(1)}_q$, and \eqref{eq:make-unital} to some
\begin{equation} \label{eq:make-unital-3}
\scrC^{(1)}_q \longrightarrow \scrC_q.
\end{equation}
For degree reasons, $\scrC^{(1)}_q$ has zero curvature term, and so does \eqref{eq:make-unital-3}.
Consider its inverse, which also has zero curvature. By restricting and then adding strict units, one can obtain an isomorphism $\tilde{\scrC}_q \rightarrow \scrC^{(1)}_q$, which one combines with \eqref{eq:make-unital-3} as before.
\end{proof}

\begin{lemma} \label{th:delta-triangle-2}
There is an exact triangle of $\scrC_q$-bimodules
\begin{equation} \label{eq:delta-triangle-2}
\xymatrix{
\scrC_q^\vee[-n] \ar[rr]_-{\epsilon_q} &&
\scrC_q \ar[rr]_-{\text{inclusion}} && 
\scrB_q \ar@/_1pc/[llll]_-{[1]}
}
\end{equation}
\end{lemma}

This is a stronger version of Lemma \ref{th:delta-triangle}. As before, it implies that $\scrP_q = (\scrB_q/\scrC_q)[1]$ is quasi-isomorphic to $\scrC_q^\vee$, in a way that relates $\epsilon_q$ to the connecting homomorphism in the obvious bimodule short exact sequence. Further discussion is again postponed to Section \ref{subsec:fukaya-of-lefschetz} (see Lemma \ref{th:a-triangle}). 

By applying a $q$-filtration argument to \eqref{eq:no-negative-degree} and \eqref{eq:no-negative-degree-2}, one sees that
\begin{equation} \label{eq:no-negative-degree-3}
\left.
\begin{aligned}
& H^*\big(\mathit{hom}_{(\scrC_q,\scrC_q)}((\scrC_q^\vee[-n])^{\otimes_{\scrC_q} r}, \scrC_q)\big) = 0,\;\; r \geq 1, \\
& H^*\big(\mathit{hom}_{(\scrC_q,\scrC_q)}((\scrC_q^\vee[-n])^{\otimes_{\scrC_q} r}, \scrC_q^\vee[-n])\big) = 0, \;\; r \geq 2
\end{aligned}
\right\} \text{ for $\ast < 0$.}
\end{equation}
In view of Corollary \ref{th:concrete-classification-2}, this implies the following analogue of Proposition \ref{th:1-variable-1}:

\begin{proposition} \label{th:1-variable-1b}
$\scrB_q$ is determined, up to quasi-isomorphism, by $\scrC_q$ and the bimodule homomorphism $\epsilon_q$ from \eqref{eq:delta-triangle}. More precisely, what matters is the homotopy class $[\epsilon_q]$, up to composition with automorphisms of $\scrC_q^\vee$ on the right.
\end{proposition}


The statement of Proposition \ref{th:1-variable-1b} is somewhat qualitative in nature. We now turn to a more explicit, but less general, version. Namely, suppose that $\scrC_q$ is a trivial $A_\infty$-deformation of $\scrC$. This means that there is an $A_\infty$-homomorphism $\Psi_q: \scrC[[q]] \rightarrow \scrC_q$ which is the identity modulo $q$. Concretely, such a trivialization consists of maps
\begin{equation} \label{eq:trivialization-maps}
\Psi^d_q: \scrC^{\otimes d} \longrightarrow \scrC_q[1-d] \quad \text{for $d \geq 1$, and with }
\Psi^d_q\big|_{q = 0} = \begin{cases} \mathit{id} & d = 1, \\ 0 & \text{otherwise,} \end{cases}
\end{equation}
which relate the two $A_\infty$-structures (in the general theory of $A_\infty$-deformations, such a trivialization would have an additional curvature term, but in our case that's zero for degree reasons). One can carry $\epsilon_q$ across this trivialization, and accordingly think of it as a family of bimodule maps $\scrC^\vee[-n] \rightarrow \scrC$ depending on the parameter $q$. In formulas, the trivialization yields an isomorphism
\begin{equation} \label{eq:delta-is-q-dependent}
H^n\big(\mathit{hom}_{(\scrC_q,\scrC_q)}( \scrC_q^\vee,\scrC_q)\big) \iso H^n\big(\mathit{hom}_{(\scrC,\scrC)}(\scrC^\vee,\scrC)\big)[[q]].
\end{equation}
To exploit the consequences of this, it is convenient to replace the given $\scrB_q$ by a quasi-isomorphic $A_\infty$-algebra, which has both the triviality of $\scrC_q$ and the quasi-isomorphism $\scrB_q/\scrC_q \htp \scrC_q^\vee[1-n]$ built in more directly. That algebra will live on the space
\begin{equation} \label{eq:trivial-q-spaces}
\tilde{\scrB}_q = \scrC[[q]] \oplus (\scrC^\vee[[q]])[1-n]).
\end{equation}
As in our previous discussion of \eqref{eq:trivial-extension}, this comes with an additional grading, where the first summand has weight $0$ and the second weight $-1$.

\begin{proposition} \label{th:1-variable-2}
Assume that the deformation $\scrC_q$ is trivial. Pick a trivialization, and correspondingly think of $\epsilon_q$ as living on the right hand side of \eqref{eq:delta-is-q-dependent}, and suppose that it lies in
\begin{equation} \label{eq:v-subspace}
H^n(\mathit{hom}_{(\scrC,\scrC)}(\scrC^\vee,\scrC)) \otimes V
\subset H^n(\mathit{hom}_{(\scrC,\scrC)}(\scrC^\vee,\scrC))[[q]] \;\;
\text{ for a $\bQ$-subspace $V \subset \bQ[[q]]$.}
\end{equation}
Then there is is an $A_\infty$-structure on \eqref{eq:trivial-q-spaces}, such that:
\begin{align}
& \label{eq:q-independent}
\mybox{$\scrC[[q]] \subset \tilde{\scrB}_q$ is an $A_\infty$-subalgebra (just to reiterate, $\scrC[[q]]$ carries the $q$-linear extension of the $A_\infty$-structure on $\scrC$). Moreover, the induced bimodule structure on $\tilde{\scrB}_q/\scrC[[q]] = (\scrC^\vee[[q]])[1-n])$ is exactly that of the dual diagonal bimodule.} \\
& \label{eq:q-weight}
\mybox{For any $r \geq 0$, the part of the $A_\infty$-structure of $\tilde{\scrB}_q$ of weight $r$ has coefficients which are homogeneous polynomials of degree $r$ in $V$ (thanks to \eqref{eq:q-independent}, we already knew this for $r = 0$).} \\
& \label{eq:trivial-isomorphism}
\mybox{There is a filtered quasi-isomorphism $\tilde\scrB_q \rightarrow \scrB_q$, which extends the given trivialization $\scrC[[q]] \rightarrow \scrC_q$.}
\end{align}
\end{proposition}

\begin{remark}
Readers who find this level of detail overwhelming may be satisfied with a less precise version (which is sufficient in order to obtain Theorem \ref{th:1-variable}). Namely, there is an $A_\infty$-algebra $\tilde{\scrB}_q$ over $\bQ[[q]]$ whose underlying space is explicitly given as $\tilde{\scrB}[[q]]$ for some graded $\bQ$-vector space $\tilde{\scrB}$, such that 
\begin{equation} \label{eq:polydeg}
\mu^d_{\tilde{\scrB}_q} \in \mathit{Hom}^{2-d}(\tilde{\scrB}^{\otimes d}, \tilde{\scrB}) \otimes \mathit{Sym}^{\leq d}(V) \subset \mathit{Hom}^{2-d}(\tilde{\scrB}^{\otimes d}, \tilde{\scrB})[[q]],
\end{equation}
and which is filtered quasi-isomorphic to $\scrB_q$ (the property \eqref{eq:polydeg} follows from our original statement, because that operation has only nonzero pieces of weight $0,\dots,d$). 
\end{remark}

\begin{proof}[Proof of Proposition \ref{th:1-variable-2}] 
This has several steps. Take the maps \eqref{eq:trivialization-maps} and extend them to $\Psi_q^d: \scrB^{\otimes d} \rightarrow \scrB_q[1-d]$, still having the same $q = 0$ reduction property as in \eqref{eq:trivialization-maps} (where of course $\mathit{id}$ is now the identity on $\scrB$). Having done that, there is a unique $A_\infty$-deformation of $\scrB$, which is trivial on $\scrC$, and which turns these maps into an $A_\infty$-isomorphism to the given $\scrB_q$. Denote that structure by $\tilde\scrB_q^{(1)}$. By construction, it contains $\scrC[[q]]$ as an $A_\infty$-subalgebra. From $\Psi_q$ and the corresponding property of $\scrB_q$, the $\tilde{\scrB}_q^{(1)}$-bimodule $\tilde\scrP_q^{(1)} = (\tilde\scrB_q^{(1)}/\scrC[[q]])[-1]$ inherits a quasi-isomorphism $\gamma_q^{(1)}: (\scrC[[q]]^\vee)[-n] \rightarrow \tilde{\scrP}_q^{(1)}$.

Next, apply Lemma \ref{th:bimodule-transfer-2} to $\scrC[[q]] \subset \tilde\scrB_q^{(1)}$ and the map $\gamma_q^{(1)}$. The outcome is another $A_\infty$-algebra $\tilde{\scrB}_q^{(2)}$ containing $\scrC[[q]]$, such that $\tilde{\scrB}_q^{(2)}/\scrC[[q]]$ is exactly equal to $(\scrC[[q]]^\vee)[1-n]$. This comes with a quasi-isomorphism $\tilde\scrB^{(2)}_q \rightarrow \tilde{\scrB}^{(1)}_q$ which is the identity on $\scrC[[q]]$, and which induces $\gamma_q^{(1)}$ on the associated bimodules. As usual, to the inclusion $\scrC[[q]] \subset \tilde{\scrB}^{(2)}_q$ is associated a map of bimodules, of the form
\begin{equation}
\tilde{\epsilon}_q^{(2)}: (\scrC[[q]]^\vee)[-n] \longrightarrow \scrC[[q]].
\end{equation}
Tracing through the construction, one sees that this is precisely $\epsilon_q$, when thought of as lying in the right hand side of \eqref{eq:delta-is-q-dependent}, hence its homotopy class lies in \eqref{eq:v-subspace}.

Now we start working from the other end. Using Lemma \ref{th:field-of-definition}, one can find an $A_\infty$-structure on the space $\tilde{\scrB}_q$ which satisfies \eqref{eq:q-independent} and \eqref{eq:q-weight}, and such that the associated bimodule map is homotopic to $\tilde{\epsilon}_q^{(2)}$. Finally, applying the classification result from Lemma \ref{th:first-order-2}, one can find an $A_\infty$-isomorphism $\tilde{\scrB}_q \rightarrow \tilde{\scrB}^{(2)}_q$ which is the identity on $\scrC[[q]]$, and induces the identity on the quotient bimodule $(\scrC[[q]])^\vee[1-n]$. Note that both of the last-mentioned Lemmas depended on \eqref{eq:no-negative-degree}, \eqref{eq:no-negative-degree-2}. Together with the previous steps, this gives a sequence of quasi-isomorphisms which establishes \eqref{eq:trivial-isomorphism}:
\begin{equation}
\tilde{\scrB}_q \longrightarrow \tilde{\scrB}_q^{(2)} \longrightarrow \tilde{\scrB}_q^{(1)} \longrightarrow \scrB_q.
\end{equation}
\end{proof}

\begin{remark} \label{th:use-tilde-c}
Throughout Proposition \ref{th:1-variable-2}, one can replace $\scrC$ by the quasi-isomorphic $\tilde{\scrC}$ from \eqref{eq:simple-directed}. One can think of this as an application of Lemma \ref{th:v-transfer}; or alternatively, we could have used $\tilde{\scrC}$ throughout the argument, at the price of a slight increase in complexity (because the functor $\tilde{\scrC} \rightarrow \scrB$ is not an inclusion).
\end{remark}

\begin{scholium} \label{th:kaledin}
We need a bit more algebraic background. For an $A_\infty$-deformation $\scrA_q$, let $\mathit{HH}^*(\scrA_q,\scrA_q)$ be the Hochschild cohomology. (One can think of it as the endomorphisms of the diagonal bimodule, even though that's not particularly helpful here.) The $q$-dependence of the $A_\infty$-structure is measured by the Kaledin class (the terminology is from \cite{seidel15}, but the notion itself goes back to \cite{kaledin07, lunts10})
\begin{equation} \label{eq:kaledin}
\kappa_q \in \mathit{HH}^2(\scrA_q,\scrA_q).
\end{equation}
Suppose that the Kaledin class vanishes, and that we choose a bounding cochain for the underlying Hochschild cocycle. More intrinsically, one can consider this as the choice of an $A_\infty$-connection $\calnablaq$, which is a kind of  $q$-differentiation operation with additional $A_\infty$-terms. Homotopy classes of $A_\infty$-connections form an affine space over $\mathit{HH}^1(\scrA_q,\scrA_q)$, meaning that the choice of bounding cochain matters only modulo coboundaries. Moreover, a filtered quasi-isomorphism of curved $A_\infty$-algebras induces a bijection between homotopy classes of connections. As explained in \cite{seidel21b}, an $A_\infty$-connection $\calnablaq$ induces connections, in the ordinary sense of $q$-differentiation operations, on Hochschild cohomology and related groups. These depend only on the homotopy class of $\calnablaq$. The one of interest to us is
\begin{equation} \label{eq:nabla-2}
\nabla^2_q: H^*(\mathit{hom}_{(\scrA_q,\scrA_q)}(\scrA_q^\vee,\scrA_q)) \longrightarrow
H^*(\mathit{hom}_{(\scrA_q,\scrA_q)}(\scrA_q^\vee,\scrA_q)).
\end{equation}
Since we are working over $\bQ$, we can obtain stronger results by, roughly speaking, integrating $A_\infty$-connections. Namely, as shown in \cite{kaledin07,lunts10}, 
the Kaledin class is zero if and only $\scrA_q$ is a trivial $A_\infty$-deformation (of its $q = 0$ part, denoted by $\scrA$). More precisely, for every $A_\infty$-connection $\calnablaq$ there is a unique trivialization $\scrA[[q]] \iso \scrA_q$ under which $\calnablaq$ corresponds to the trivial connection (meaning $d/dq$) on $\scrA[[q]]$. We will give an exposition of this theory in Section \ref{sec:connections}.
\end{scholium}

Let's return to the specific $A_\infty$-algebra $\scrC_q$ arising in our geometric situation.

\begin{proposition} \label{th:1-variable-3}
Assume that \eqref{eq:simply-connected} and \eqref{eq:as} hold. Then the Kaledin class of $\scrC_q$ is zero. Moreover, there is a choice of $A_\infty$-connection such that for the associated operation \eqref{eq:nabla-2}, the element $\epsilon_q$ from \eqref{eq:delta-triangle-2} satisfies
\begin{equation} \label{eq:2nd-order-2}
\nabla_q^2 \nabla_q^2 \epsilon_q + \left( \eta - \frac{\psi'}{\psi} \right) \nabla_q^2 \epsilon_q - 4z^{(2)} \psi^2 \epsilon_q =0.
\end{equation}
 \end{proposition}

This is the main course on our menu, and will be extensively discussed later. By combining it with the previous material, one can quickly complete the proof of the result stated in the Introduction.


\begin{proof}[Proof of Theorem \ref{th:model}] 
Take the $A_\infty$-connection from Proposition \ref{th:1-variable-3}, and use that to trivialize the deformation $\scrC_q$ (see Scholium \ref{th:kaledin}). This induces a commutative diagram, with the vertical arrows as in \eqref{eq:delta-is-q-dependent},
\begin{equation}
\xymatrix{
H^*(\mathit{hom}_{(\scrC_q,\scrC_q)}(\scrC_q^\vee,\scrC_q)) \ar[rr]^-{\nabla_q^2} \ar[d]_-{\iso} &&
H^*(\mathit{hom}_{(\scrC_q,\scrC_q)}(\scrC_q^\vee,\scrC_q)) \ar[d]^-{\iso} 
\\
H^*(\mathit{hom}_{(\scrC,\scrC)}(\scrC^\vee,\scrC))[[q]] \ar[rr]^-{d/dq}  &&
H^*(\mathit{hom}_{(\scrC,\scrC)}(\scrC^\vee,\scrC))[[q]] 
}
\end{equation}
Let's combine that with Proposition \ref{th:1-variable-3}. The outcome is that, if we think of $\epsilon_q$ as lying in $H^n(\mathit{hom}_{(\scrC,\scrC)}(\scrC^\vee,\scrC))[[q]]$, it satisfies
\begin{equation} \label{eq:2nd-order-3} 
\epsilon_q'' + \left( \eta - \frac{\psi'}{\psi} \right) \epsilon_q' - 4z^{(2)} \psi^2 \epsilon_q =0,
\end{equation}
Let $s_0,s_1$ be a basis of solutions of \eqref{eq:2nd-order}, as in \eqref{eq:y0y1}. Then, \eqref{eq:2nd-order-3} implies that there are bimodule maps $\epsilon_0,\epsilon_1: \scrC^\vee[-n] \rightarrow \scrC$ such that
\begin{equation} \label{eq:write-delta}
\epsilon_q = \epsilon_0 s_0 + \epsilon_1 s_1.
\end{equation}
In other words, $\epsilon_q$ lies in \eqref{eq:v-subspace} for $V = \bQ s_0 \oplus \bQ s_1 \subset \bQ[[q]]$. Proposition \ref{th:1-variable-2} then yields an $A_\infty$-structure on $\tilde{\scrB}_q$ quasi-isomorphic to $\scrB_q$, and so that the weight $r$ part has coefficients which are homogeneous polynomials of degree $r$ in $(s_0,s_1)$. We can rescale the operations on $\tilde{\scrB}_q$ by pulling them back by the automorphism which is the identity on $\scrC[[q]]$, and $s_0^{-1}$ times the identity on $(\scrC^\vee[[q]])[1-n])$. This rescales the weight $r$ part by $s_0^{-r}$. Hence, for the pullback structure, the weight $r$ part has coefficients which are degree $\leq r$ polynomials in $f = s_1/s_0$. We know that $f$ is of the form \eqref{eq:initial} and a solution of \eqref{eq:schwarz-eq}. Finally, note that Theorem \ref{th:model} was formulated using $\tilde{\scrC}$ instead of $\scrC$, but that is irrelevant, see Remark \ref{th:use-tilde-c}.
\end{proof}

\begin{remark} \label{th:doesnt-matter}
In the proof above, we could have used any pair $(s_0,s_1)$ of solutions satisfying \eqref{eq:y0y1}. Correspondingly, for our statement to hold, $f$ is not constrained to be a specific solution of \eqref{eq:schwarz-eq}; any solution of the form \eqref{eq:initial} will do. 
\end{remark}

\subsection{The Fukaya category of the Lefschetz fibration\label{subsec:fukaya-of-lefschetz}}
The main role in our argument is actually played by the Fukaya category $\scrF(p)$ of the Lefschetz fibration $p: E \rightarrow \bC$, and its deformation $\scrF_q(\bar{p})$ associated to the fibrewise compactification $\bar{p}: \bar{E} \rightarrow \bC$, with its symplectic divisor $\delta E$. 
This comes with a restriction functor \cite{seidel20}
\begin{equation} \label{eq:restriction}
\scrF_q(\bar{p}) \longrightarrow \scrF_q(\bar{M}).
\end{equation}
To understand that, it is useful to think of $M$ not as being located at infinity, but as the fibre of $p$ over some finite (but large) point of $\bC$; of course, up to isomorphism there is no difference between the two. The construction of the restriction functor involves suitably adapted choices of almost complex structures and Hamiltonian terms. Having done that, it is simply a projection map on the spaces of Floer cochains, retaining only those Lagrangian chords (the perturbed version of intersection points) which lie in $M$. Algebraically, the restriction functor is the simplest kind of $A_\infty$-homomorphism, with no  terms other than the linear one, and constant in $q$.

Each vanishing cycle $V_k \subset M$ in our chosen basis comes from a Lefschetz thimble $\Delta_k \subset E$. Let $\scrA \subset \scrF(p)$ be the resulting full $A_\infty$-category, and $\scrA_q \subset \scrF_q(\bar{p})$ the corresponding part of the relative Fukaya category. We want to set up things in a very specific way, which is compatible with our previous \eqref{eq:cf-vk} via \eqref{eq:restriction}. One can arrange that the endomorphisms of each Lefschetz thimble have the form
\begin{equation} \label{eq:cf-delta}
\mathit{CF}^*(\Delta_k,\Delta_k) = \bQ e_{\Delta_k} \oplus \bQ t_{\Delta_k} \oplus \bQ u_{\Delta_k},
\end{equation}
where the degrees of the generators (in that order) are $0$, $(n-1)$, $n$. Geometrically, these generators are given by points of $\tilde{\Delta}_k \cap \Delta_k$ for a suitably chosen Hamiltonian perturbation $\tilde{\Delta}_k$. The first two generators correspond to intersection points lying in $M$, while the third one is outside. Hence, restriction takes $e_{\Delta_k}$ and $t_{\Delta_k}$ to their counterparts in \eqref{eq:cf-vk}, while it kills $u_{\Delta_k}$. Along similar lines, one can arrange that for $i<j$, all of $\mathit{CF}^*(\Delta_i,\Delta_j)$ is generated by intersection points in $M$, hence maps isomorphically to $\mathit{CF}^*(V_i,V_j)$. Finally, by the definition of a basis of vanishing cycles, the differential $\mu^1_{\scrA}$ is acyclic on $\mathit{CF}^*(\Delta_i,\Delta_j)$ for $i>j$.

\begin{lemma}
The $A_\infty$-deformation $\scrA_q$ satisfies
\begin{align} 
\label{eq:mu0-aq}
& \mu^0_{\scrA_q} = 0\;\; \text{ except possibly if $n = 2$,} \\
\label{eq:mu1-aq}
& \mu^1_{\scrA_q}(e_{\Delta_k}) = 0, \\
\label{eq:mu1-aq2}
& \mu^1_{\scrA_q}(t_{\Delta_k}) = \text{(some $\bQ[[q]]^\times$ multiple of) } u_{\Delta_k}.
\end{align}
\end{lemma}

\begin{proof}
We know (see Lemma \ref{th:2-curvature} for the only nontrivial case) that $V_k$ as an object of $\scrF_q(\bar{M})$ has vanishing curvature term. Because of the structure of the restriction functor, this means that $\mu^0_{\scrA_q}$ would need to consist of multiples of the $u_{\Delta_k}$, which implies \eqref{eq:mu0-aq} for degree reasons. Similarly, \eqref{eq:mu1-aq} is true for degree reasons except possibly if $n = 2$, and in that case it follows from Lemma \ref{th:subalgebra-2}. If we set $q = 0$, then $\mu^1_{\scrA}(t_{\Delta_k})$ must be a $\bQ^\times$-multiple of $u_{\Delta_k}$, simply because the Floer cohomology of $\Delta_k \subset E$ is isomorphic to its ordinary cohomology, and that yields \eqref{eq:mu1-aq2}.
\end{proof}

\begin{lemma} \label{th:quasi-isomorphic-to-directed}
The restriction functor yields an $A_\infty$-homomorphism $\scrR_q: \scrA_q \rightarrow \scrB_q$. Moreover, there is a filtered quasi-isomorphism $\scrA_q \rightarrow \scrC_q$, where $\scrC_q \subset \scrB_q$ is the $A_\infty$-subalgebra structure on the subspace \eqref{eq:b-directed}, which fits into a homotopy commutative diagram
\begin{equation} \label{eq:a-restriction}
\xymatrix{
\scrA_q \ar[rr]^-{\scrR_q} \ar[drr]_-{\htp}  && \scrB_q
\\ &&
\scrC_q \ar@{_{(}->}[u]
}
\end{equation}
\end{lemma}
%

\begin{proof}
Assume first that $n>2$. One defines $\scrR_q$ to be \eqref{eq:restriction} applied to the objects $\Delta_k$. Take
\begin{equation} \label{eq:directed-a}
\tilde{\scrA} = \bigoplus_k \bQ e_{\Delta_k} \oplus \bigoplus_{i<j} \mathit{CF}^*(\Delta_i,\Delta_j) \subset \scrA.
\end{equation}
From \eqref{eq:mu0-aq} and \eqref{eq:mu1-aq}, it follows that this gives an $A_\infty$-subalgebra $\tilde{\scrA}_q \subset \scrA_q$. From \eqref{eq:mu1-aq2} and the acyclicity of $\mathit{CF}^*(\Delta_i,\Delta_j)$ for $i>j$, one sees that the inclusion of that subalgebra is a filtered quasi-isomorphism. By looking at the restriction of $\scrR_q$, we get a commutative diagram
\begin{equation}
\xymatrix{
\scrA_q \ar[rr]^-{\scrR_q}
&&
\scrB_q
\\
\tilde{\scrA}_q \ar@{^{(}->}[u]^-{\htp} \ar[rr]_-{\scrR_q|\tilde{\scrA}_q}^{\iso}
&&
\scrC_q.
\ar@{_{(}->}[u]
}
\end{equation}
Up to homotopy, one can fill this in with a diagonal map and obtain \eqref{eq:a-restriction}.

The remaining case $n = 2$ brings up the previously mentioned issue with defining $\scrB_q$ in that dimension. Let's consider the larger categories $\scrF_q(\bar{p})^{(0)}$ and $\scrF_q(\bar{M})^{(0)}$ whose objects are Lagrangians together with a degree $1$ Floer cochain whose $q=0$ term is zero. This is a purely algebraic construction, so the restriction functor extends (in a somewhat obvious way) to
\begin{equation}
\scrF_q(\bar{p})^{(0)} \longrightarrow \scrF_q(\bar{M})^{(0)}.
\end{equation}
Thanks to \eqref{eq:mu1-aq2}, one can find, for each $1 \leq k \leq m$, some $\alpha_k \in q\mathit{CF}^1(\Delta_k,\Delta_k)[[q]]$ which solves the Maurer-Cartan equation, meaning that the analogue of \eqref{eq:mc-alpha} is zero. Let $\scrA_q^{(0)} \subset \scrF_q(\bar{p})^{(0)}$ be the full subcategory formed by $(\Delta_k,\alpha_k)$. As the definition of $\scrB_q$, we use the full subcategory of $\scrF_q(\bar{p})$ formed by the $V_k$ with the restrictions of the $\alpha_k$. By construction, neither of these has curvature, and the restriction functor yields
\begin{equation} \label{eq:0-restriction}
\scrR_q^{(0)}: \scrA_q^{(0)} \longrightarrow \scrB_q.
\end{equation}
Moreover, there is an obvious (curved) $A_\infty$-homomorphism 
\begin{equation} \label{eq:translation-functor}
\scrT_q: \scrA_q^{(0)} \rightarrow \scrA_q,
\end{equation}
whose curvature term is given by the $\alpha_k$; with the identity as linear term; and all higher order terms equal to zero. One defines $\scrR_q$ by composing \eqref{eq:0-restriction} with the inverse of \eqref{eq:translation-functor} (note that this will cause $\scrR_q$ to have a curvature term). The rest of the argument is as before: one defines a quasi-isomorphic subalgebra $\tilde{\scrA}_q^{(0)} \subset \scrA_q^{(0)}$ as in \eqref{eq:directed-a}, and then $\scrR_q^{(0)}|\tilde{\scrA}_q^{(0)}$ is an isomorphism from that subalgebra to $\scrC_q$. The relation with $\scrA_q$ is again provided by combining this with the inverse of \eqref{eq:translation-functor}.
\end{proof}

\begin{remark} \label{th:clunky}
The proof of Lemma \ref{th:quasi-isomorphic-to-directed} may strike the reader as clunky. In fact, what is contrived is our use of $\scrC_q$, while $\scrA_q$ is the natural geometric object. In principle, one could remove $\scrC_q$ from our argument entirely, and take \eqref{eq:restriction} as the starting point (in which case, the specific choices made in defining $\scrA_q$ would be unnecessary, since their only purpose is to facilitate the relation with $\scrC_q$). However, that seemed unwise from an expository perspective, since the definition of $\scrC_q$ is much easier to grasp at first sight.
\end{remark}

\begin{lemma} \label{th:a-triangle}
There is an exact triangle of $\scrA_q$-bimodules
\begin{equation} \label{eq:delta-triangle-4}
\xymatrix{
 \scrA_q^\vee[-n] \ar[rr]_-{\epsilon_q} && \scrA_q \ar[rr]_-{\scrR_q} && \scrB_q
\ar@/_1pc/[llll]_-{[1]}
}
\end{equation}
where the last term is an $\scrA_q$-bimodule by pullback along $\scrR_q$.
\end{lemma}

A form of this result was proved in \cite{seidel20} for the entire relative Fukaya categories, and the functor \eqref{eq:restriction}. For $n>2$, one immediately obtains it in the form stated above, by considering only our basis of Lefschetz thimbles. For $n = 2$, one needs a minor workaround involving formal enlargements, exactly as in the proof of Lemma \ref{th:quasi-isomorphic-to-directed}. At this point, we can check off one of the tasks left open earlier:

\begin{proof}[Proof of Lemma \ref{th:delta-triangle-2}]
Let's pull back \eqref{eq:delta-triangle-4} via the filtered quasi-isomorphism from Lemma \ref{th:quasi-isomorphic-to-directed} (we do so without changing the notation). The outcome is a diagram of $\scrC_q$-bimodules
\begin{equation} \label{eq:transfer-exactness}
\xymatrix{
 \scrA_q^\vee[-n] \ar[rr]_-{\epsilon_q} && \scrA_q \ar[rr]_-{\scrR_q} && \scrB_q
 \\
\ar@{-->}[rr]
\scrC_q^\vee[-n] \ar[u]^-{\htp} && \scrC_q \ar[u]^-{\htp} \ar@{^{(}->}[rr] && \scrB_q \ar@{=}[u]
}
\end{equation}
The dotted arrow, which is by definition the map $\epsilon_q$ in \eqref{eq:delta-triangle-2}, is obtained by filling in the diagram. Since the top row admits a third morphism making it into an exact triangle, so does the bottom one.
\end{proof}

\subsection{Hamiltonian Floer cohomology\label{subsec:hamiltonian-0}}
We will use the formalism from \cite{seidel17} for Hamiltonian Floer groups associated to $\bar{p}$, with slightly different notation. The Floer-theoretic invariants are finite-dimensional graded $\bQ((q))$-vector spaces
\begin{equation} \label{eq:all-floer}
\mathit{HF}^*_{q,q^{-1}}(\bar{p},r),\;\; r \in \bR.
\end{equation}
For $r \notin \bZ$, they are constructed from Hamiltonians that rotate the base at infinity by an angle $2\pi r$. The groups with $r = 1,2,\dots$ use different Hamiltonians, and have a fundamentally closer relationship with the Fukaya category. The groups for different values of $r$ are related by continuation maps
\begin{equation} \label{eq:continuation}
C_{q,q^{-1}}^{r_1,r_2}: \mathit{HF}^*_{q,q^{-1}}(\bar{p},r_1) \longrightarrow \mathit{HF}^*_{q,q^{-1}}(\bar{p},r_2), \quad r_1 < r_2.
\end{equation}
There are also PSS (Piunikhin-Salamon-Schwarz \cite{piunikhin-salamon-schwarz94}) maps
\begin{equation} \label{eq:pss-r}
\mathit{PSS}_{q,q^{-1}}^r: H^*(\bar{E};\bQ)((q)) \longrightarrow \mathit{HF}^*_{q,q^{-1}}(\bar{p},r),\;\; r>0.
\end{equation}
The PSS maps for different $r$ are related by composition with continuation maps. The following two statements are basic Floer theory results:
\begin{align} \label{eq:less-than-1}
& \mathit{PSS}_{q,q^{-1}}^r \text{ is an isomorphism for } r \in (0,1), \\
& C_{q,q^{-1}}^{r_1,r_2} \text{ is an isomorphism if $r_1,r_2 \in (k,k+1)$ for some $k \in \bZ$.} \label{eq:small-increase-of-c}
\end{align}
One can partially describe \eqref{eq:all-floer} in topological terms, along the lines of \cite{mclean12}. Rather than encapsulating that into a spectral sequence, we break up the statement into pieces (the proofs of which will be described in Section \ref{subsec:hyperbolic}; we emphasize that this follows known ideas):

\begin{lemma} \label{th:les-2}
For $r_2 \in \bZ$ and $r_1 \in (r_2 - 1 ,r_2)$, the continuation map fits into a long exact sequence
\begin{equation} \label{eq:hf-exact-2}
\cdots \rightarrow \mathit{HF}^*_{q,q^{-1}}(\bar{p},r_1) \xrightarrow{C_{q,q^{-1}}^{r_1,r_2}} \mathit{HF}^*_{q,q^{-1}}(\bar{p},r_2) \longrightarrow H^{*-1}(\bar{M};\bQ)((q)) \rightarrow \cdots
\end{equation}
\end{lemma}

\begin{lemma} \label{th:les-1}
For $r_2 \in \bZ$ and $r_3 \in (r_2,r_2 + 1)$, the continuation map fits into a long exact sequence
\begin{equation} \label{eq:hf-exact-1}
\cdots \rightarrow \mathit{HF}^*_{q,q^{-1}}(\bar{p},r_2) \xrightarrow{C_{q,q^{-1}}^{r_2,r_3}} \mathit{HF}^*_{q,q^{-1}}(\bar{p},r_3) 
\longrightarrow H^*(\bar{M};\bQ)((q)) \rightarrow \cdots
\end{equation}
\end{lemma}

Write
\begin{equation} \label{eq:k-class}
k^r_{q,q^{-1}} = \mathit{PSS}_{q,q^{-1}}^r(q^{-1}[\delta E]) \in \mathit{HF}^2_{q,q^{-1}}(\bar{p},r), \;\; r>0.
\end{equation}
Suppose that $k^1_{q,q^{-1}}$ is zero, and choose a bounding cochain for the underlying Floer cocycle. This bounding cochain matters only up to coboundaries, so the space of choices is an affine space over $\mathit{HF}^1_{q,q^{-1}}(\bar{p},1)$. Such a choice gives rise to connections, constructed in \cite{seidel17},
\begin{equation} \label{eq:hamiltonian-connections}
\begin{aligned} 
& \nabla_{q,q^{-1}}^r: \mathit{HF}^*_{q,q^{-1}}(\bar{p},r) \longrightarrow \mathit{HF}^*_{q,q^{-1}}(\bar{p},r),\;\; r = 1,2,\dots, \\
& \nabla_{q,q^{-1}}^r( f(q) x) = f(q)\, \nabla_{q,q^{-1}}^r x + f'(q) x \;\; \text{ for $f \in \bQ((q))$.}
\end{aligned}
\end{equation}
The corresponding theory for non-integral $r$, which is the subject of \cite{seidel16}, is more complicated. Suppose that $k_{q,q^{-1}}^{r_0} = 0$  for some $r_0 \in (0,\infty) \setminus \bZ$. As usual, we need to pick a bounding cochain, with the space of effective choices being an affine space over $\mathit{HF}^1_{q,q^{-1}}(\bar{p},r_0)$. This gives rise to operations
\begin{equation} \label{eq:hamiltonian-connections-2}
\begin{aligned}
& \nabla^{w,r_1,r_2}_{q,q^{-1}}: \mathit{HF}^*_{q,q^{-1}}(\bar{p},r_1) \longrightarrow \mathit{HF}^*_{q,q^{-1}}(\bar{p},r_2), \\
& \nabla_{q,q^{-1}}^{w,r_1,r_2}(f(q)x) = f(q) \nabla^{w,r_1,r_2}_{q,q^{-1}}(x) + 
f'(q) C_{q,q^{-1}}^{r_1,r_2}(x), \\
& \qquad \qquad \text{for any $w \in \bZ$, and $r_1, r_2 \in (0,\infty) \setminus \bZ$ with $r_0 + r_1 \leq r_2$.}
\end{aligned}
\end{equation}
Note that here $w$ is independent of the $r_k$. A partial compatibility statement between the integral and non-integral cases is given by \cite[Proposition 7.13]{seidel17}. We will only need a special case:

\begin{proposition} \label{th:compatible-connections}
Suppose that $k_{q,q^{-1}}^1 = 0$. Choose a bounding cochain, and use that to define the connection $\nabla^2_{q,q^{-1}}$ from \eqref{eq:hamiltonian-connections}. Fix
\begin{equation} \label{eq:r012}
r_0>1, \; r_1 \in (1,2), \; r_2 \in (2,3),\;  \text{ such that $r_0+r_1 \leq r_2$.}
\end{equation}
Under the continuation map, our bounding cochain gives rise to one for $k^{r_0}_{q,q^{-1}}$.
If we use that to define the connection $\nabla^{2,r_1,r_2}_{q,q^{-1}}$ from \eqref{eq:hamiltonian-connections-2}, the following diagram commutes:
\begin{equation} \label{eq:compatible-connections}
\xymatrix{
\ar@/_1pc/[rrr]_-{\nabla^{2,r_1,r_2}_{q,q^{-1}}}
\mathit{HF}^*_{q,q^{-1}}(\bar{p},r_1) \ar[r]^-{C_{q,q^{-1}}^{r_1,2}} & 
\mathit{HF}^*_{q,q^{-1}}(\bar{p},2) \ar[r]^-{\nabla^2_{q,q^{-1}}} & 
\mathit{HF}^*_{q,q^{-1}}(\bar{p},2) \ar[r]^-{C_{q,q^{-1}}^{2,r_2}} &
\mathit{HF}^*_{q,q^{-1}}(\bar{p},r_2).
}
\end{equation}
\end{proposition}

\begin{remark} \label{th:c-vs-r}
We need to clarify some notational discrepancies. In \cite{seidel16} the connections \eqref{eq:hamiltonian-connections-2} are denoted by $\nabla^c$, where $c = r-1$; and in \cite{seidel17} the connections \eqref{eq:hamiltonian-connections} are written just as $\nabla$. 
\end{remark}

\begin{lemma} \label{th:kernel-of-pss}
Take the Gromov-Witten invariant \eqref{eq:z1} and restrict it to $\bar{E}$. Then, the image of that element under $\mathit{PSS}_{q,q^{-1}}^r$, $r > 1$, is zero.
\end{lemma}

This is \cite[Lemma 8.4]{seidel16}. As an immediate consequence, if \eqref{eq:as} holds, then $k_{q,q^{-1}}^r = 0$ for $r>1$. Our next result involves two specific elements of Floer cohomology. The first is
\begin{equation}
e^r_{q,q^{-1}} = \mathit{PSS}_{q,q^{-1}}^r(1) \in \mathit{HF}^0_{q,q^{-1}}(\bar{p},r), \;\; r>0; \label{eq:diagonal-class}
\end{equation}
the second is the Borman-Sheridan class, defined in \cite[Section 8]{seidel16}, and denoted here by
\begin{equation}
s^r_{q,q^{-1}} \in \mathit{HF}^0_{q,q^{-1}}(\bar{p},r), \;\; r \in (1,\infty) \setminus \bZ.
\end{equation}
In both cases, the classes for different $r$ are related by continuation maps.

\begin{theorem} \label{th:1st-order}
Take \eqref{eq:r012}, and suppose that \eqref{eq:as} holds, making $\nabla_{q,q^{-1}}^{2,r_1,r_2}$ defined (by Lemma \ref{th:kernel-of-pss}). For a suitable choice of bounding cochain in the construction of that connection, we have the following equalities in $\mathit{HF}^0_{q,q^{-1}}(\bar{p},r_2)$:
\begin{align} \label{eq:nabla-equality-1}
& \nabla^{2,r_1,r_2}_{q,q^{-1}} e^{r_1}_{q,q^{-1}} = -\psi\, s^{r_2}_{q,q^{-1}}, \\
\label{eq:nabla-equality-2}
& \nabla^{2,r_1,r_2}_{q,q^{-1}} s^{r_1}_{q,q^{-1}} = -\eta\, s^{r_2}_{q,q^{-1}} - 4z^{(2)} \psi\, e^{r_2}_{q,q^{-1}}.
\end{align}
\end{theorem}

This is a form of the main result of \cite{seidel16}. More specifically, \eqref{eq:nabla-equality-1} is essentially the $c = 1$ case of \cite[Equation (2.32)]{seidel16}. The statement there is given in terms of symplectic cohomology, which is the $r \rightarrow \infty$ limit of the groups $\mathit{HF}^*_{q,q^{-1}}(\bar{p},r)$. However, the proof via \cite[Proposition 9.5]{seidel16} actually proves the result stated here; and the distinction is in any case irrelevant for our application, since (as a consequence of Lemmas \ref{th:les-2} and \ref{th:les-1}) all continuation maps are injective in degree zero. The other part \eqref{eq:nabla-equality-2} is \cite[Equation (2.35)]{seidel16}, with the same remarks applying.

Finally, we need to explain the relation between these considerations and the Fukaya category of the Lefschetz fibration. Unsurprisingly, these go via closed-open string maps.
Since we have worked over $\bQ((q))$ on the Hamiltonian Floer cohomology side, we will do the same for the Fukaya category, passing to
\begin{equation} \label{eq:generic-fukaya}
\scrA_{q,q^{-1}} = \scrA_q \otimes_{\bQ[[q]]} \bQ((q)).
\end{equation}
Then, the closed-open string maps which are relevant to us have the form
\begin{align}
\label{eq:oc-1}
& \mathit{CO}_{q,q^{-1}}^1 : \mathit{HF}^*_{q,q^{-1}}(\bar{p}, 1) \longrightarrow \mathit{HH}^*(\scrA_{q,q^{-1}},\scrA_{q,q^{-1}}), \\
\label{eq:oc-2}
& \mathit{CO}_{q,q^{-1}}^2:  \mathit{HF}^*_{q,q^{-1}}(\bar{p}, 2) \longrightarrow H^{*+n}(\mathit{hom}_{(\scrA_{q,q^{-1}},\scrA_{q,q^{-1}})}(\scrA_{q,q^{-1}}^\vee,\scrA_{q,q^{-1}})).
\end{align}
The elementary theory of connections from Scholium \ref{th:kaledin} has a straightforward analogue for \eqref{eq:generic-fukaya}. Namely, if the Kaledin class $\kappa_{q,q^{-1}} \in \mathit{HH}^2(\scrA_{q,q^{-1}},\scrA_{q,q^{-1}})$ vanishes, one can equip $\scrA_{q,q^{-1}}$ with an $A_\infty$-connection, and that in turn gives rise to an analogue $\nabla_{q,q^{-1}}^2$ of \eqref{eq:nabla-2}. The following is proved in \cite{seidel21b}:

\begin{proposition} \label{th:first-co}
The Kaledin class $\kappa_{q,q^{-1}}$ is the image of \eqref{eq:k-class} under \eqref{eq:oc-1}. Suppose that $k_{q,q^{-1}}^1 = 0$; choose a bounding cochain; and use the chain level map underlying \eqref{eq:oc-1} to associate to that a bounding cochain for $\kappa_{q,q^{-1}}$. Then, for the connections defined by those bounding cochains, we have a commutative diagram
\begin{equation}
\xymatrix{
\ar[d]_-{\mathit{CO}_{q,q^{-1}}^2}
\mathit{HF}^*_{q,q^{-1}}(\bar{p},r) \ar[r]^-{\nabla_{q,q^{-1}}^2} 
& \mathit{HF}^*_{q,q^{-1}}(\bar{p},r) \ar[d]^-{\mathit{CO}_{q,q^{-1}}^2}
\\
H^{*+n}(\mathit{hom}_{(\scrA_{q,q^{-1}},\scrA_{q,q^{-1}})}(\scrA_{q,q^{-1}}^\vee,\scrA_{q,q^{-1}}))
\ar[r]^-{\nabla_{q,q^{-1}}^2} &
H^{*+n}(\mathit{hom}_{(\scrA_{q,q^{-1}},\scrA_{q,q^{-1}})}(\scrA_{q,q^{-1}}^\vee,\scrA_{q,q^{-1}})).
}
\end{equation}
\end{proposition}

The bimodule homomorphism $\epsilon_q$ from \eqref{eq:delta-triangle-4} induces one in the target space of \eqref{eq:oc-2}, which we denote by $\epsilon_{q,q^{-1}}$. Another result from \cite{seidel21b} says that:

\begin{proposition} \label{th:second-co}
$\epsilon_{q,q^{-1}}$ is the image of \eqref{eq:diagonal-class} under \eqref{eq:oc-2}.
\end{proposition}

Essentially all we have done so far was to collect results from the literature. Let's see how these results are brought to bear.

\begin{corollary} \label{th:unique-connection}
Suppose that \eqref{eq:simply-connected} holds. Then $\mathit{HF}^1_{q,q^{-1}}(\bar{p},r) = 0$ for $1 \leq r<2$; and the continuation map $\mathit{HF}^2_{q,q^{-1}}(\bar{p},1) \rightarrow \mathit{HF}^2_{q,q^{-1}}(\bar{p},r)$ is injective.
\end{corollary}

\begin{proof}
Let's look first at Lemma \ref{th:les-1}, with $r_1 = 1$ and $r_2  = r \in (1,2)$:
\begin{equation} \label{eq:its-injective}
\begin{aligned}
& \cdots \rightarrow
\mathit{HF}^1_{q,q^{-1}}(\bar{p},1) \longrightarrow \mathit{HF}^1_{q,q^{-1}}(\bar{p},r)
\longrightarrow 0 = H^1(\bar{M};\bQ)((q))
\\ & \qquad \qquad \qquad 
\xrightarrow{\text{connecting map}} \mathit{HF}^2_{q,q^{-1}}(\bar{p},1) \xrightarrow{\text{continuation map}} \mathit{HF}^2_{q,q^{-1}}(\bar{p},r) \rightarrow \cdots
\end{aligned}
\end{equation}
This shows the desired injectivity statement. Next, we combine \eqref{eq:less-than-1} and Lemma \ref{th:les-2}, with $r_1 \in (0,1)$ and $r_2 = 1$. Since $\bar{E}$ is obtained from $\bar{M}$ by attaching $n$-cells, \eqref{eq:simply-connected} implies that $H^1(\bar{E}) = 0$. The outcome is 
\begin{equation} \label{eq:tricky-sequence}
\begin{aligned}
& 0 = H^1(\bar{E};\bQ)((q)) \longrightarrow \mathit{HF}^1_{q,q^{-1}}(\bar{p},1) \longrightarrow H^0(\bar{M};\bQ)((q)) = \bQ((q))
\\ & \qquad \qquad \qquad \qquad
\xrightarrow{\text{connecting map}} H^2(\bar{E};\bQ)((q)) \xrightarrow{\text{PSS map}} \mathit{HF}^2_{q,q^{-1}}(\bar{p},1) \rightarrow \cdots
\end{aligned}
\end{equation}
The Gromov-Witten invariant from Lemma \ref{th:kernel-of-pss} is nonzero, $z^{(1)}|\bar{E} = q^{-1}[\delta E] + \cdots$, but the PSS map takes it to the zero element in $\mathit{HF}^2_{q,q^{-1}}(\bar{p},r)$, and therefore by \eqref{eq:its-injective} also in $\mathit{HF}^2_{q,q^{-1}}(\bar{p},1)$. This implies that the PSS map in \eqref{eq:tricky-sequence} is not injective, hence the connecting map is nonzero, and therefore $\mathit{HF}^1_{q,q^{-1}}(\bar{p},1) = 0$. The corresponding statement for $1<r<2$ follows from \eqref{eq:its-injective}.
\end{proof}

\begin{corollary} \label{th:k1}
Suppose that \eqref{eq:simply-connected} and \eqref{eq:as} hold. Then $k^1_{q,q^{-1}} = 0$.
\end{corollary}

\begin{proof} 
This follows from Lemma \ref{th:kernel-of-pss} and Corollary \ref{th:unique-connection}.
\end{proof}

\begin{corollary} \label{th:1-variable-4.5}
Suppose that assumptions \eqref{eq:simply-connected} and \eqref{eq:as} hold. Then, for any choice of bounding cochain underlying the construction of $\nabla_{q,q^{-1}}^2$, we have the following equality in $\mathit{HF}^0_{q,q^{-1}}(\bar{E},2)$:
\begin{equation} \label{eq:the-equation-0}
\nabla_{q,q^{-1}}^2 \nabla_{q,q^{-1}}^2 e_{q,q^{-1}} + \left( \eta - \frac{\psi'}{\psi} \right) \nabla_{q,q^{-1}}^2 e_{q,q^{-1}} - 4z^{(2)} \psi^2 e_{q,q^{-1}} =0.
\end{equation}
\end{corollary}

\begin{proof}
In this context, the connections appearing in Proposition \ref{th:first-co} exist (Corollary \ref{th:k1}) and are unique (Corollary \ref{th:unique-connection}). Hence, we can apply both Proposition \ref{th:compatible-connections} and Theorem \ref{th:1st-order}, without worrying about choices of bounding cochains. With that in mind, the formulae \eqref{eq:nabla-equality-1}, \eqref{eq:nabla-equality-2} can be written as
\begin{align}
& C^{2,r_2}_{q,q^{-1}} \nabla^2_{q,q^{-1}} C^{r_1,2} e^{r_1}_{q,q^{-1}} = C^{2,r_2}_{q,q^{-1}}(-\psi \,C^{r_1,2}_{q,q^{-1}} s^{r_1}_{q,q^{-1}}), \\
& C^{2,r_2}_{q,q^{-1}} \nabla^2_{q,q^{-1}} C^{r_1,2} s^{r_1}_{q,q^{-1}} = C^{2,r_2}_{q,q^{-1}}(-\eta \, C^{r_1,2}_{q,q^{-1}} s^{r_1}_{q,q^{-1}} - 4z^{(2)} \psi\, C^{r_1,2}_{q,q^{-1}} e^{r_1}_{q,q^{-1}}).
\end{align}
(We apologize for the proliferation of subscripts and superscripts; it seemed worth while keeping track of the distinction between related but different objects.) Applying $C_{q,q^{-1}}^{r_1,2}$ to the class $e_{q,q^{-1}}^{r_1}$ yields its counterpart $e_{q,q^{-1}}^2$. As for $s_{q,q^{-1}}^{r_1}$, we denote its image in $\mathit{HF}^*_{q,q^{-1}}(\bar{p},2)$ by $s^2_{q,q^{-1}}$. The map $C^{2,r_2}_{q,q^{-1}}$ is injective in degree zero, because of Lemma \ref{th:les-1}. It follows that
\begin{align}
& \nabla^2_{q,q^{-1}} e^2_{q,q^{-1}} = -\psi \, s^2_{q,q^{-1}}, \\
& \nabla^2_{q,q^{-1}} s^2_{q,q^{-1}} = -\eta \, s^2_{q,q^{-1}} - 4z^{(2)} \psi\,  e^2_{q,q^{-1}}.
\end{align}
Combining those two equations yields \eqref{eq:the-equation-0}.
\end{proof}

\begin{corollary} \label{th:preliminary}
Suppose that assumptions \eqref{eq:simply-connected} and \eqref{eq:as} hold. Then the Kaledin class of $\scrA_{q,q^{-1}}$ is zero, and for a suitable choice of bounding cochain, the induced connection on $H^*(\mathit{hom}_{(\scrA_{q,q^{-1}},\scrA_{q,q^{-1}})}(\scrA_{q,q^{-1}}^\vee,\scrA_{q,q^{-1}}))$ satisfies
\begin{equation} \label{eq:the-equation-1}
\nabla_{q,q^{-1}}^2 \nabla_{q,q^{-1}}^2 \epsilon_{q,q^{-1}} + \left( \eta - \frac{\psi'}{\psi} \right) \nabla_{q,q^{-1}}^2 \epsilon_{q,q^{-1}} - 4z^{(2)} \psi^2 \epsilon_{q,q^{-1}} =0.
\end{equation}
\end{corollary}

\begin{proof}
This follows from Corollary \ref{th:1-variable-4.5} by using Propositions \ref{th:first-co} and \ref{th:second-co}. 
\end{proof}

\subsection{Relative Floer cohomology}
The reader may have noticed a discrepancy between the closed and open string theories: the first one was defined over $\bQ((q))$, whereas our relative Fukaya categories live over $\bQ[[q]]$, and we had to forget that fact by passing to \eqref{eq:generic-fukaya}. We will now remedy that, by introducing suitable relative Hamiltonian Floer cohomology groups which are defined over $\bQ[[q]]$. This is a rather finicky construction (carried out in Sections \ref{sec:relative-floer}--\ref{sec:apply-to-lefschetz}), leading us to state and use only a minimal amount of its properties. The notation for the relative groups will be
\begin{equation} \label{eq:relative-p-floer}
\mathit{HF}^*_q(\bar{p},r), \;\; 1 \leq r < 2,
\end{equation}
and the relation with the previously defined ones is that 
\begin{equation}
\mathit{HF}^*_{q,q^{-1}}(\bar{p},r) \iso \mathit{HF}^*_q(\bar{p},r) \otimes_{\bQ[[q]]} \bQ((q))
\end{equation}
canonically. The continuation maps \eqref{eq:continuation} have relative lifts, denoted by $C_q^{r_1,r_2}$ (and the analogue of \eqref{eq:small-increase-of-c} holds, so effectively only two different groups appear). We will also need the relative version of a specific case of Lemma \ref{th:les-1}:

\begin{lemma} \label{th:les-3}
For $r \in (1,2)$, the continuation map fits into a long exact sequence
\begin{equation} \label{eq:hf-exact}
\cdots \rightarrow \mathit{HF}^*_q(\bar{p},1) \xrightarrow{C_q^{1,r}} \mathit{HF}^*_q(\bar{p},r) 
\longrightarrow H^*(\bar{M};\bQ)[[q]] \rightarrow \cdots
\end{equation}
\end{lemma}

The relative lift of the PSS map is of the form
\begin{equation} \label{eq:pss-r-2}
\mathit{PSS}_q^r: q^{-1} H^*(\bar{E}, E;\bQ) \oplus
H^*(\bar{E};\bQ)[[q]] \longrightarrow \mathit{HF}^*_q(\bar{p},r),\;\; 1 \leq r < 2,
\end{equation}
where multiplication by $q$ maps the first summand to the second. The class $q^{-1}[\delta E]$ lies in the domain of \eqref{eq:pss-r-2}, and that allows us to lift \eqref{eq:k-class} to relative Floer cohomology. We denote the outcome by $k_q^r$. Similarly, if we take $z^{(1)}$ as in \eqref{eq:z1}, then $z^{(1)}|E$ lies in the domain of \eqref{eq:pss-r-2}. A relative version of Lemma \ref{th:kernel-of-pss} holds:

\begin{lemma} \label{th:kernel-of-pss-2}
The image of $z^{(1)}|E$ under $\mathit{PSS}_q^r$, $r > 1$, is zero.
\end{lemma}

\begin{corollary} \label{th:k2}
Suppose that \eqref{eq:simply-connected} and \eqref{eq:as} hold. Then $k^1_q = 0$.
\end{corollary}

\begin{proof}
This is the relative version of Corollary \ref{th:k1}, proved in the same way but with the ingredients replaced with their relative lifts. Lemma \ref{th:kernel-of-pss-2} and \eqref{eq:as} imply that $k^r_q = 0$ for $r>1$. To get the desired result from that, one uses \eqref{eq:simply-connected} and Lemma \ref{th:les-3} to show injectivity of the continuation map $\mathit{HF}^2_q(\bar{p},1) \rightarrow \mathit{HF}^2_q(\bar{p},r)$.
 \end{proof}

There is a relative lift of the closed-open string map \eqref{eq:oc-1},
\begin{equation} \label{eq:oc-1b}
\mathit{CO}_q^1 : \mathit{HF}^*_q(\bar{p}, 1) \longrightarrow \mathit{HH}^*(\scrA_q,\scrA_q), \\
\end{equation}
and part of Proposition \ref{th:first-co} carries over, meaning:

\begin{proposition} \label{th:first-co-2}
The Kaledin class $\kappa_q$ is the image of $k_q^1$ under \eqref{eq:oc-1b}. 
\end{proposition}

The proof follows exactly the non-relative version from \cite{seidel21b}, and we won't discuss it further. The other properties of relative Hamiltonian Floer cohomology used above (Lemmas \ref{th:les-3} and \ref{th:kernel-of-pss-2}) will be proved in Section \ref{sec:apply-to-lefschetz}. 

\begin{corollary} \label{th:1-variable-5}
Suppose that assumptions \eqref{eq:simply-connected} and \eqref{eq:as} hold. Then the Kaledin class of $\scrA_q$ is zero, and for a suitable choice of bounding cochain, the induced connection \eqref{eq:nabla-2} satisfies
\begin{equation} \label{eq:the-equation-2}
\nabla_{q}^2 \nabla_{q}^2 \epsilon_{q} + \left( \eta - \frac{\psi'}{\psi} \right) \nabla_{q}^2 \epsilon_{q} - 4z^{(2)} \psi^2 \epsilon_{q} =0.
\end{equation}
\end{corollary}

\begin{proof}
Vanishing of the Kaledin class follows from Corollary \ref{th:k2} and Proposition \ref{th:first-co-2}. As a consequence, $\scrA_q$ is a trivial $A_\infty$-deformation (see Scholium \ref{th:kaledin}). Therefore, both $\mathit{HH}^*(\scrA_q,\scrA_q)$ and $H^*(\mathit{hom}_{(\scrA_q,\scrA_q)}(\scrA_q^\vee,\scrA_q))$ are torsion-free $\bQ[[q]]$-modules, so one has
\begin{align} \label{eq:localize-1}
& \mathit{HH}^*(\scrA_q,\scrA_q) \hookrightarrow \mathit{HH}^*(\scrA_q,\scrA_q) \otimes_{\bQ[[q]]} \bQ((q)) \iso \mathit{HH}^*(\scrA_{q,q^{-1}},\scrA_{q,q^{-1}}), \\
\label{eq:localize-2}
& 
\begin{aligned} 
& H^*(\mathit{hom}_{(\scrA_q,\scrA_q)}(\scrA_q^\vee,\scrA_q)) \hookrightarrow
H^*(\mathit{hom}_{(\scrA_q,\scrA_q)}(\scrA_q^\vee,\scrA_q)) \otimes_{\bQ[[q]]} \bQ((q)) 
\\ & \qquad \qquad
\iso H^*(\mathit{hom}_{(\scrA_{q,q^{-1}},\scrA_{q,q^{-1}})}(\scrA_{q,q^{-1}}^\vee,\scrA_{q,q^{-1}})). \end{aligned}
\end{align}
The isomorphisms above require a bit of explanation. In general, Hochschild cohomology does not commute with base change, because of the nature of the underlying chain complex as an infinite product. However, $\scrA$ is cohomologically smooth, as one can see clearly from its quasi-isomorphism with the directed algebra $\tilde{\scrC}_q$ (see Lemmas \ref{th:directed-2} and \ref{th:quasi-isomorphic-to-directed}); and for such algebras, the base change formula holds. The same applies to \eqref{eq:localize-2}.

We need to return for a moment to the argument underlying Corollary \ref{th:preliminary}. The bounding cochain which appears there is one inherited from a bounding cochain for $k_{q,q^{-1}}^1$ via the chain level closed-open map. In our situation, one can start by first choosing a bounding cochain for $k_q^1$; then, the outcome is that the connection $\nabla_q^2$ on $H^*(\mathit{hom}_{(\scrA_q,\scrA_q)}(\scrA_q^\vee,\scrA_q))$ is compatible with that from \eqref{eq:the-equation-1}. Moreover, $\epsilon_q$ maps to $\epsilon_{q,q^{-1}}$ under \eqref{eq:localize-2}, by definition of the latter element.
Therefore, \eqref{eq:the-equation-1} just says that the desired equation holds after tensoring with $\bQ((q))$. Because of the injectivity of \eqref{eq:localize-2}, it must hold on the nose.
\end{proof}

This completes our account of the overall argument: all the steps leading up to Theorem \ref{th:1-variable} have now been introduced. It remains to make a note concerning the topological assumption \eqref{eq:simply-connected}.

\begin{remark} \label{th:elliptic-surface}
Suppose that we have a finite group $\Gamma$ acting on $\cornerbar{E}$, which preserves $\delta E|$. We also require that in a neighbourhood of $\infty$, that group should map each fibre of $\cornerbar{p}$ to itself. One then has induced actions on all the Hamiltonian Floer cohomology groups we have discussed. In particular, one can consider $\Gamma$-equivariant connections, which have the advantage that the space of choices is an affine space over the $\Gamma$-invariant part of $\mathit{HF}^1$ (in its various versions, depending on which connection we are talking about). The connection which appears in Theorem \ref{th:1st-order} is automatically $\Gamma$-invariant, because the entire argument is geometric, hence compatible with additional symmetries.
As a consequence, the argument above goes through under the weaker assumption that
\begin{equation} \label{eq:invariant-vanish}
H^1(\bar{M};\bQ)^\Gamma = 0.
\end{equation}
For instance, suppose that we start with the Hesse pencil of elliptic curves on $\bC P^2$ and blow up its base locus, which yields the rational elliptic surface
\begin{equation}
\cornerbar{E} = \{y_0 (x_0^3+x_1^3+x_2^3) + y_1 x_0x_1x_2 = 0\} \subset \bC P^1 \times \bC P^2.
\end{equation}
The pencil has an action of $\mathit{Sym}_3$ permuting the $x$ variables. If one looks at a smooth fibre $\bar{M}$, say that at $[y_0:y_1] = [1:1]$, the action has four orbits with isotropy subgroups $\bZ/2$, namely those of $[x_0:x_1:x_2] = [1:\lambda:\lambda]$ for $2\lambda^3 + \lambda^2 + 1 = (\lambda+1)(2\lambda^2-\lambda+1) = 0$, as well as $[x_0:x_1:x_2] = [0:1:-1]$. From Euler characteristic considerations alone, it follows that $\bar{M}/\Gamma$ is a rational curve, hence \eqref{eq:invariant-vanish} holds. (One could object that the projection $\cornerbar{E} \rightarrow \bC P^1$ is not exactly a Lefschetz fibration in the ordinary sense, since each of the four singular fibres has three nodal singularities; however, that is not a problem for the argument leading up to Proposition \ref{th:1-variable-5}, which is the only part of our overall approach where \eqref{eq:simply-connected} was used.) 
This allows Theorem \ref{th:1-variable} to be applied to the elliptic surface, justifying the discussion of that example in Section \ref{subsec:examples}.
\end{remark}

\section{A relative classification theorem for $A_\infty$-structures\label{sec:noncommutative-divisors}}

The theory developed in this section is purely algebraic, and in a sense elementary. For a given $A_\infty$-algebra $\scrA$, it seeks to classify pairs consisting of another such algebra $\scrB$ and a homomorphism $\scrF: \scrA \rightarrow \scrB$. We will also describe a modified version for $A_\infty$-deformations, meaning that it includes a formal parameter $q$, and possible curvature terms. Part of this material has appeared elsewhere \cite{seidel08, seidel14b, seidel15}, but the account here places it in a more organic framework (which can also be considered more natural from a geometric viewpoint, compare Remark \ref{th:clunky}).

\subsection{Introduction to the problem\label{subsec:algebra}}
We work with $A_\infty$-algebras over $\bQ$. All $A_\infty$-algebras and related structures will be graded and cohomologically unital. As already mentioned before, the $A_\infty$-category (actually dg category) of bimodules over an $A_\infty$-algebra $\scrA$ will be denoted by $(\scrA,\scrA)$. Sign conventions are as in \cite{seidel04} for $A_\infty$-algebras, and \cite{seidel14b} for bimodules; as in those references, we use $|\cdot|$ for the degree of an element, and $\|\cdot\| = |\cdot|-1$ for the reduced degree. (Those sign conventions do not agree with those in the classical world of dg algebras; when passing between the two, additional signs appear.)

\begin{remark} \label{th:semisimple}
Actually, our applications are to $A_\infty$-categories with finitely many objects, or equivalently $A_\infty$-algebras over the semisimple ground ring $R = \bQ^{\oplus m}$. This has been implicit in our discussion, starting back in Section \ref{subsec:directed}. The necessary adaptations are straighforward, so we prefer to stick to the simpler language of working over $\bQ$.
\end{remark}

Fix an $A_\infty$-algebra $\scrA$. We consider pairs consisting of an $A_\infty$-algebra $\scrB$ and an $A_\infty$-homomorphism $\scrF: \scrA \rightarrow \scrB$. Two such pairs $(\scrB,\scrF)$ and $(\tilde\scrB,\tilde\scrF)$ are called quasi-isomorphic if there is a homotopy commutative diagram (with $\scrG$ a quasi-isomorphism of $A_\infty$-algebras, denoted here by $\htp$)
\begin{equation} \label{eq:bf-quasi-isomorphism}
\xymatrix{
&& \scrB \ar[d]_-{\htp}^-{\scrG} \\
\scrA \ar[urr]^-{\scrF} \ar[rr]_-{\tilde\scrF}
&& \tilde\scrB.
}
\end{equation}
Any $A_\infty$-algebra is a bimodule over itself (the diagonal bimodule; the bimodule operations are the $A_\infty$-algebra operations, with certain added signs \cite[Equation 2.3]{seidel14b}). In our context, pulling back the diagonal $\scrB$-bimodule by $\scrF$ (on both sides) yields a $\scrA$-bimodule, which we denote by $\scrF^*\scrB$. This comes with a bimodule map $\scrA \rightarrow \scrF^*\scrB$, which we denote by $\scrF$ as well. We consider the mapping cone bimodule, and its projection map
\begin{equation} \label{eq:cone-bimodule}
\left\{
\begin{aligned}
& \scrP = \big\{ \scrA \xrightarrow{\scrF} \scrF^*\scrB \big\}[-1], \\
& \rho: \scrP \longrightarrow \scrA.
\end{aligned}
\right.
\end{equation}
(The notation $\{\cdots\}$ will be used for mapping cones throughout our discussion.) The datum $(\scrP,\rho)$ is an invariant of our situation. This means that if we have quasi-isomorphic pairs as in \eqref{eq:bf-quasi-isomorphism}, with associated bimodule data $(\scrP,\rho)$ and $(\tilde\scrP,\tilde\rho)$, there is a homotopy commutative diagram of $\scrA$-bimodules
\begin{equation} \label{eq:pdelta-quasi-isomorphism}
\xymatrix{
\scrP \ar[d]^-{\htp}_{\gamma} \ar[drr]^-{\rho} \\ 
\tilde\scrP \ar[rr]_-{\tilde\rho} && \scrA.
}
\end{equation}

\begin{lemma} \label{th:ambidextrous}
In the situation of \eqref{eq:cone-bimodule}, one always has
\begin{equation} \label{eq:delta-left-right}
\mathit{id} \otimes_{\scrA} \rho \htp \rho \otimes_{\scrA} \mathit{id}: 
\scrP \otimes_{\scrA} \scrP \longrightarrow \scrP.
\end{equation}
\end{lemma}

Let's decrypt the statement a little. Writing $T(\scrA[1]) = \bigoplus_{d \geq 0} \scrA[1]^{\otimes d}$, the bimodule tensor product is $\scrQ \otimes_{\scrA} \scrR = \scrQ \otimes T(\scrA[1]) \otimes \scrR$, with suitable operations \cite[Equation (2.12)]{seidel14b}. For any $\scrA$-bimodule $\scrQ$, one has canonical (up to homotopy) quasi-isomorphisms $\scrA \otimes_{\scrA} \scrQ \htp \scrQ$ and $\scrQ \otimes_{\scrA} \scrA \htp \scrQ$. Explicit representatives $\epsilon^{\mathit{left}}$ and $\epsilon^{\mathit{right}}$ are given in \cite[Equations (2.21) and (2.26)]{seidel08}. Using those quasi-isomorphisms, and given any bimodule map $\scrQ \rightarrow \scrA$, one can form two bimodule maps $\scrQ \otimes_{\scrA} \scrQ \rightarrow \scrQ$, as in the two sides of \eqref{eq:delta-left-right}. If the two are homotopic, the original bimodule map is called ambidextrous (in the terminology from \cite{seidel08}). Lemma \ref{th:ambidextrous} asserts that $\rho$ always has this property. 

\begin{proof}
The $A_\infty$-algebra structure on $\scrB$, and the functor $\scrF$, produce a bimodule homomorphism
$\phi: \scrF^*\scrB \otimes_{\scrA} \scrF^*\scrB \rightarrow \scrF^*\scrB$.  Concretely, the linear component is
\begin{equation}
\begin{aligned}
& \phi^{0,1,0}: \scrB \otimes T(\scrA[1]) \otimes \scrB \longrightarrow \scrB, \\
& \phi^{0,1,0}(b' \otimes a_t \otimes \cdots \otimes a_1 \otimes b) =
\sum_{\substack{ r \geq 0 \\ \!\!\!d_1+\cdots+d_r = t\!\!\!}} (-1)^{|b'|} \mu_{\scrB}^{r+2}(b',\scrF^{d_r}(a_t,\dots,a_{t-d_r+1}),\dots, \\[-1.5em] & \hspace{25em} \scrF^{d_1}(a_{d_1},\dots,a_1),b).
\end{aligned}
\end{equation}
and the higher order ones make use of additional elements of $\scrA$ on the left and right in the same way. This fits into a commutative diagram
\begin{equation} \label{eq:phi-diagram}
\xymatrix{
\scrA \otimes_{\scrA} \scrF^*\scrB 
\ar[drr]^-{\epsilon^{\mathit{left}}}
\ar[d]_-{\scrF \otimes_{\scrA} \mathit{id}}
\\
\scrF^*\scrB \otimes_{\scrA} \scrF^*\scrB \ar[rr]^-{\phi} && \scrF^*\scrB
\\
\scrF^*\scrB \otimes_{\scrA} \scrA
\ar[urr]_-{\epsilon^\mathit{right}}
\ar[u]^-{\mathit{id} \otimes_{\scrA} \scrF}
}
\end{equation}
The difference between the two sides in \eqref{eq:delta-left-right} is the following bimodule map (grading shifts have been omitted for simplicity):
\begin{equation} \label{eq:spell-out-left-right}
\xymatrix{
\scrP \otimes_{\scrA} \scrP = \big\{ \scrA \otimes_{\scrA} \scrA 
\ar[r]^-{\raisebox{.5em}{$\scriptstyle\left(\begin{smallmatrix} \mathit{id} \otimes_{\scrA} \scrF \\ \scrF \otimes_{\scrA} \mathit{id} \end{smallmatrix}\right)$}}
& (\scrA \otimes_{\scrA} \scrF^*\scrB) \oplus 
(\scrF^*\scrB \otimes_{\scrA} \scrA) \ar[rr]^-{\raisebox{.5em}{$\scriptstyle(\scrF \otimes_{\scrA} \mathit{id},\, -\mathit{id} \otimes_{\scrA} \scrF)$}}
 \ar[d]^-{(\epsilon^{\mathit{left}}, -\epsilon^{\mathit{right}})}
 && \scrF^*\scrB \otimes_{\scrA} \scrF^*\scrB \big\}
\\
\scrP = \{\scrA \ar[r]_-{\scrF} & \scrF^*\scrB\} 
}
\end{equation}
In view of \eqref{eq:phi-diagram}, one can use $\phi$ to give a nullhomotopy for this map.
\end{proof}

To summarize the discussion so far in a more down-to-earth way, let's fix a bimodule $\scrP$. Write $\mathit{Aut}(\scrP) = H^0(\mathit{hom}_{(\scrA,\scrA)}(\scrP,\scrP))^\times$ for its automorphism group; and $\mathit{Ambi}(\scrP) \subset H^0(\hom_{(\scrA,\scrA)}(\scrP,\scrA))$ for the subspace of ambidextrous bimodule maps, meaning that they satisfy the analogue of \eqref{eq:delta-left-right}. Then, we have a canonical map
\begin{equation} \label{eq:map-to-bimodule-data}
\left\{ 
\parbox{17.5em}{
pairs $(\scrB,\scrF)$ whose associated bimodule is quasi-isomorphic to $\scrP$, up to \eqref{eq:bf-quasi-isomorphism}}\right\} \longrightarrow \mathit{Ambi}(\scrP)/\mathit{Aut}(\scrP).
\end{equation}

\begin{theorem} \label{th:relative-classification}
Take some $(\scrB,\scrF)$, with associated bimodule $\scrP$. Assume that the tensor powers of $\scrP$ satisfy
\begin{equation} \label{eq:classification-groups}
\left.
\begin{aligned} 
& H^{2-2r}(\mathit{hom}_{(\scrA,\scrA)}(\scrP^{\otimes_{\scrA} r}, \scrA)) = 0 \\
& H^{3-2r}(\mathit{hom}_{(\scrA,\scrA)}(\scrP^{\otimes_{\scrA} r}, \scrP)) = 0
\end{aligned}
\right\}\; \text{for all $r \geq 2$.}
\end{equation}
In this situation, the bimodule data determine $(\scrB,\scrF)$ in the following sense. Suppose we have another pair $(\tilde{\scrB},\tilde{\scrF})$, such that associated bimodule data satisfy \eqref{eq:pdelta-quasi-isomorphism} for some $\gamma$. Then, there is a diagram \eqref{eq:bf-quasi-isomorphism} such that the bimodule map induced by $\scrG$ is homotopic to $\gamma$.
\end{theorem}

One can rephrase the main implication of Theorem \ref{th:relative-classification} as follows:

\begin{corollary} \label{th:concrete-classification}
If \eqref{eq:classification-groups} holds, \eqref{eq:map-to-bimodule-data} is injective.
\end{corollary}

It is easy to impose additional conditions that make \eqref{eq:map-to-bimodule-data} a bijection (namely, those in \eqref{eq:one-degree-higher} below).

\begin{example} (A silly example, but one that provides a check on why the degrees in \eqref{eq:classification-groups} are the relevant ones.)  Suppose that $\scrB$ is a strictly unital $A_\infty$-algebra, and not zero, so that the unit map $\bQ \rightarrow \scrB$ is injective. Then, setting $\scrA = \bQ$, the unit map is an $A_\infty$-homomorphism, with $\scrP = (\scrB/\bQ)[-1]$. The cohomology groups that appear in \eqref{eq:classification-groups} are 
\begin{equation} \label{eq:classification-groups-x}
\left.
\begin{aligned}
& \mathit{Hom}^{2-r}( (\scrB/\bQ)^{\otimes r}, \bQ), \\
& \mathit{Hom}^{2-r}( (\scrB/\bQ)^{\otimes r}, \scrB/\bQ)
\end{aligned} \right\}
\; \text{for $r \geq 2$,}
\end{equation}
which together add up to $\mathit{Hom}^{2-r}( (\scrB/\bQ)^{\otimes r}, \scrB)$. If one thinks of $\scrB$ just as a chain complex containing a distinguished unit element, then these groups parametrize the choices of potential $A_\infty$-operations on $\scrB$ for which our distinguished element is the strict unit (``potential'' because that does not take the associativity constraints into account). As a more concrete (and trivial) example, take $\scrB = \bQ \oplus \bQ[1]$, a rare case in which \eqref{eq:classification-groups-x} vanish, and where the unitality condition determines $\mu^d_{\scrB}$, $d \geq 2$, uniquely. The remaining operation $\mu_{\scrB}^1$ corresponds exactly to the choice of  $\rho$. 
\end{example}
%
%
%

\subsection{Simplifying the homomorphism}
The proof of Theorem \ref{th:relative-classification} will be broken down into several steps. The first stage is turning $A_\infty$-homomorphisms into inclusions.

\begin{lemma} \label{th:turn-into-inclusion}
Let $\tilde\scrF: \scrA \rightarrow \tilde\scrB$ be an $A_\infty$-homomorphism. Suppose that we have a chain complex $\scrB$, containing $\scrA$ as a subcomplex, together with a quasi-isomorphism $\scrG^1: \scrB \rightarrow \tilde\scrB$ such that $\scrG^1|\scrA = \tilde{\scrF}^1$. Then there is an $A_\infty$-algebra structure on $\scrB$ (with the differential as before) and an $A_\infty$-homomorphism $\scrG: \scrB \rightarrow \tilde\scrB$ (with the linear part $\scrG^1$ as before), such that $\scrA \subset \scrB$ is an $A_\infty$-subalgebra, and $\scrG|\scrA = \tilde{\scrF}$. In other words, there is a diagram \eqref{eq:bf-quasi-isomorphism} where $\scrF$ is an inclusion, and which is strictly commutative.
\end{lemma}

\begin{proof}
Since this is the first of several similar arguments, we'll give ourselves some expository leeway. The space $\mathit{Hom}(T(\scrB[1]),\scrB[1])$ comes with a graded bracket $[\cdot,\cdot]$, which is natural if one thinks of it as the space of coderivations of the tensor coalgebra. (Our sign conventions are such that a graded Lie algebra is a special case of an $L_\infty$-structure.) Write $T^{> r}(\scrB[1]) = T(\scrB[1])/\bigoplus_{d \leq r} T(\scrB[1])^{\otimes d}$. The $A_\infty$-associativity equations can be written as 
\begin{equation} \label{eq:classical-mc}
\half [\mu_{\scrB},\mu_{\scrB}] = 0 \;\; \text{ for }\mu_{\scrB} = (\mu_{\scrB}^d)_{d \geq 1} 
\in \mathit{Hom}^1(T^{>0}\scrB[1],\scrB[1]) \subset \mathit{Hom}^1(T(\scrB[1]),\scrB[1]).
\end{equation}
To adapt this picture to our situation, where we have an existing differential on $\scrB$ and an existing $A_\infty$-structure on $\scrA \subset \scrB$, we proceed as follows. For $d \geq 2$, choose $\mu_{\scrB}^{\mathit{given},d}$ to be some extension of $\mu_{\scrA}^d$ to a map $\scrB^{\otimes d} \rightarrow \scrB[2-d]$. We are looking for an $A_\infty$-algebra with the given $\mu_{\scrB}^1$, and with
\begin{equation} \label{eq:modified-mu-b}
\mu_{\scrB}^d = \mu_{\scrB}^{\mathit{given},d} + \mu_{\scrB}^{\mathit{new},d} \;\; \text{for $d >1$,}
\end{equation}
where the new term vanishes on $\scrA^{\otimes d}$. Thus, $\mu_{\scrB}^{\mathit{new}}$ is an element of
\begin{equation} \label{eq:b-lie}
\frakb = \mathit{Hom}(T^{>1} \scrB[1]/ T^{>1} \scrA[1], \scrB[1]).
\end{equation}
This space comes with the structure of a curved dg Lie algebra (which we think of as a curved $L_\infty$-structure with operations $\lambda_{\frakb}^d$ that vanish for $d > 2$). To define that structure, we start with the previously used bracket and set
\begin{equation} \left\{
\begin{aligned}
& \lambda^0_{\frakb} = \half [\mu_{\scrB}^{\mathit{given}}, \mu_{\scrB}^{\mathit{given}}], \\
& \lambda^1_{\frakb} = [\mu_{\scrB}^{\mathit{given}},\cdot] = [\cdot,\mu_{\scrB}^{\mathit{given}}], \\
& \lambda^2_{\frakb} = [\cdot,\cdot].
\end{aligned} \right.
\end{equation}
The $A_\infty$-associativity equation for \eqref{eq:modified-mu-b} can be written as a curved Maurer-Cartan equation
\begin{equation} \label{eq:012-mc}
\lambda^0_{\frakb} + \lambda^1_{\frakb}(\mu_{\scrB}^{\mathit{new}}) + \half \lambda^2_\frakb(\mu_{\scrB}^{\mathit{new}},\mu_{\scrB}^{\mathit{new}}) = 0.
\end{equation}
We have a complete decreasing filtration of $\frakb$, whose pieces $F^r \frakb$, $r \geq 1$, consist of those multilinear maps which vanish on $\bigoplus_{d \leq r} \scrB^{\otimes d}$. Our curved dg Lie structure is compatible with that filtration.
That makes it possible to analyze \eqref{eq:012-mc} through order-by-order obstruction theory, by which we mean looking at the cohomology of the associated graded spaces. Ignoring shifts in the grading, for the sake of brevity, these are
\begin{equation} \label{eq:b-filtered}
H^*(F^{r-1} \frakb/F^r \frakb) = \mathit{Hom}\big(H^*(\scrB)^{\otimes r}/H^*(\scrA)^{\otimes r}, H^*(\scrB)\big).
\end{equation}

Let's take a look at the other piece of our problem, $A_\infty$-homomorphisms, by itself. Namely, suppose that we are given $A_\infty$-algebras $\scrB$ and $\tilde{\scrB}$, and want to construct an $A_\infty$-homomorphism $\scrG: \scrB \rightarrow \tilde{\scrB}$. The $A_\infty$-homomorphism property can be written as an extended Maurer-Cartan equation,
\begin{equation} \label{eq:l-infinity-mc}
\lambda^1(\scrG) + \half \lambda^2(\scrG,\scrG) + \textstyle\frac16 \lambda^3(\scrG,\scrG,\scrG) + \cdots = 0,
\end{equation}
where the $\lambda^d$ are $L_\infty$-operations on $\mathit{Hom}(T^{>0}(\scrB[1]),\tilde{\scrB})$. As before, we can introduce a relative version, in which we start with a given $\tilde\scrF: \scrA \rightarrow \tilde\scrB$ defined on an $A_\infty$-subalgebra $\scrA \subset \scrB$, as well as an extension $\scrG^1: \scrB \rightarrow \tilde{\scrB}$ of $\tilde\scrF^1$ as a chain map. One proceeds as before: first, pick arbitrary extensions $\scrG^{\mathit{given},d}$ of the $\tilde{\scrF}^d$, for $d \geq 2$; then, look for a solution $\scrG^d = \scrG^{\mathit{given},d} + \scrG^{\mathit{new},d}$, $d \geq 2$, where $\scrG^{\mathit{new},d}$ vanishes
on $\scrA^{\otimes d}$. This means that $\scrG^{\mathit{new}}$ lies in
\begin{equation} \label{eq:g-lie}
\frakg = \mathit{Hom}(T^{>1}\scrB[1]/T^{>1}\scrA[1], \tilde{\scrB}).
\end{equation}
That space has the structure of a curved $L_\infty$-algebra, defined in terms of the previous one as
\begin{equation}
\begin{aligned}
& \lambda^0_\frakg = \lambda^1(\scrG^\mathit{given}) + \half \lambda^2(\scrG^{\mathit{given}},\scrG^{\mathit{given}}) + \textstyle\frac16 \lambda^3(\scrG^{\mathit{given}},\scrG^{\mathit{given}},\scrG^{\mathit{given}}) + \cdots, \\
& \lambda^1_\frakg = \lambda^1(\cdot) + \lambda^2(\scrG^{\mathit{given}},\cdot) + \half \lambda^3(\scrG^{\mathit{given}}, \scrG^{\mathit{given}}, \cdot) + 
\textstyle\frac16 \lambda^4(\scrG^{\mathit{given}},\scrG^{\mathit{given}},\scrG^{\mathit{given}},\cdot) + 
\cdots, \\
& \dots
\end{aligned}
\end{equation}
Then, the equation for $\scrG^{\mathit{new}}$ is a version of \eqref{eq:l-infinity-mc} with curvature term:
\begin{equation} \label{eq:l0-infinity}
\lambda^0_\frakg + \lambda^1_\frakg(\scrG^{\mathit{new}}) + \half \lambda^2_\frakg(\scrG^{\mathit{new}},\scrG^{\mathit{new}})+ \textstyle\frac16 \lambda^3_\frakg(\scrG^{\mathit{new}},\scrG^{\mathit{new}},\scrG^{\mathit{new}}) + \cdots = 0.
\end{equation}
The analogue of \eqref{eq:b-filtered}, for the same kind of filtration by tensor length of inputs as before, is 
\begin{equation} \label{eq:g-filtered}
H^*(F^{r-1} \frakg/F^r\frakg) = \mathit{Hom}\big(H^*(\scrB)^{\otimes r}/H^*(\scrA)^{\otimes r}, H^*(\tilde\scrB)\big).
\end{equation}

Finally, one can combine the two theories into one which covers the construction of the pair $(\scrB,\scrG)$. We will directly proceed to the version which works relatively to $\scrA$. One can write this as an equation \eqref{eq:l0-infinity} for the pair  $(\mu_{\scrB}^\mathit{new}, \scrG^{\mathit{new}})$, in a curved $L_\infty$-algebra
\begin{equation} \label{eq:c-lie-algebra}
\frakc = \mathit{Hom}(T^{>1}\scrB[1]/T^{>1}\scrA[1], \scrB[1] \oplus \tilde{\scrB}).
\end{equation}
As a graded vector space, this is the direct sum of \eqref{eq:b-lie} and \eqref{eq:g-lie}, but the curved $L_\infty$-operations have cross-terms, which is natural since the property of $\scrG$ being an $A_\infty$-homomorphism depends on our choice of $\mu_{\scrB}$. What's relevant for us is the analogue of \eqref{eq:b-filtered} and \eqref{eq:g-filtered}, namely
\begin{equation} \label{eq:p-filtered}
H^*(F^{r-1} \frakc/F^{r}\frakc) = \mathit{Hom}\big(H^{*}(\scrB)^{\otimes r}/H^{*}(\scrA)^{\otimes r}, H^*( \{\scrB \stackrel{\scrG^1}\rightarrow \tilde{\scrB} \}) \big).
\end{equation}
In the situation we are considering, the cone in \eqref{eq:p-filtered} is acyclic, so the cohomology is zero. It then follows by a standard obstruction theory argument that a solution $(\scrB,\scrG)$ always exists.
\end{proof}

\begin{remark}
Since we have ultimately just worked order-by-order with respect to the filtration, $L_\infty$-structures are not really necessary, and one could rewrite the argument without appealing to that formalism (that removes all denominators from the argument, hence makes it available in arbitrary characteristic). However, one would then need to check by hand that the obstruction terms which occur are cocycles in $F^{r-1} \frakc/F^r \frakc$. Fitting our problem into an $L_\infty$-framework provides a more conceptual explanation.
\end{remark}

\subsection{$A_\infty$-subalgebras}
With Lemma \ref{th:turn-into-inclusion} in mind, we simplify the problem under discussion, considering only inclusions of $A_\infty$-subalgebras $\scrC \subset \scrB$. In this case, it is convenient to replace \eqref{eq:cone-bimodule} by the equivalent
\begin{equation}
\left\{
\begin{aligned}
& \scrP = (\scrB/\scrC)[-1], \text{ as an $A_\infty$-bimodule over $\scrC$,} \\
& \rho = \text{ the boundary map of the short exact sequence $0 \rightarrow \scrC \rightarrow \scrB \rightarrow \scrP[1] \rightarrow 0$.}
\end{aligned}
\right.
\end{equation}
We consider the weight filtration on $\scrB$ as in \eqref{eq:weights}. Suppose that $\scrC$ and $\scrP$ (or equivalently, $GrW\scrB$) are given. The construction of the rest of the $A_\infty$-structure can be thought of as solving a Maurer-Cartan equation in a dg Lie algebra
\begin{equation} \label{eq:frakw}
\frakw = W^1\mathit{Hom}(T(\scrB[1]), \scrB[1]),
\end{equation}
whose complete decreasing filtration by weight satisfies
\begin{equation} \label{eq:weight-cone}
W^r \frakw / W^{r+1} \frakw \iso \big\{
\mathit{hom}_{(\scrC,\scrC)}(\scrP[2]^{\otimes_{\scrC} r}, \scrC) \longrightarrow
\mathit{hom}_{(\scrC,\scrC)}(\scrP[2]^{\otimes_{\scrC} r+1}, \scrP[2]) \big\}.
\end{equation}
This means that there is a long exact sequence
\begin{equation} \label{eq:les-weight}
\begin{aligned}
& \cdots \rightarrow H^{*-2r}(\mathit{hom}_{(\scrC,\scrC)}(\scrP^{\otimes_{\scrC} r+1}, \scrP)) \longrightarrow
H^*(W^r \frakw / W^{r+1} \frakw)
\\ & \qquad \qquad \longrightarrow H^{*+1-2r}(\mathit{hom}_{(\scrC,\scrC)}(\scrP^{\otimes_{\scrC} r}, \scrC)) \rightarrow \cdots,
\end{aligned}
\end{equation}
see \cite[Equation (2.61)]{seidel14b}. As discussed there, the connecting homomorphism is $\rho \mapsto \mathit{id} \otimes_{\scrC} \rho - \rho \otimes_{\scrC} \mathit{id}$. 
%

\begin{lemma} \label{th:first-order}
(A minor generalization of \cite[Lemma 2.12]{seidel14b}) Fix $\scrC$ and a bimodule $\scrP$, such that \eqref{eq:classification-groups} holds. Consider $A_\infty$-structures on $\scrB = \scrC \oplus \scrP[1]$ which extend the given structure on $\scrA$, and induce the given bimodule structure on $\scrP$. If two such structures give rise to the same class $[\rho] \in H^0(\hom_{(\scrC,\scrC)}(\scrP,\scrC))$, they must be related by an $A_\infty$-homomorphism which restricts to the identity on $\scrC$, and which induces the identity on $\scrP$.
\end{lemma}

\begin{proof}
In view of \eqref{eq:les-weight}, the assumption \eqref{eq:classification-groups} means that $H^1(W^r \frakw/ W^{r+1} \frakw) = 0$ for all $r \geq 2$, and that 
\begin{equation} \label{eq:weight-1}
H^1(\frakw/W^2 \frakw) \longrightarrow H^0(\mathit{hom}_{(\scrC,\scrC)}(\scrP,\scrC))
\end{equation}
is injective. A solution of the Maurer-Cartan equation yields, at first order, a cocycle in $\frakw/W^2 \frakw$. It is not hard to see that the image of that cohomology class under \eqref{eq:weight-1} is exactly $[\rho]$. Given that, it is a standard fact from obstruction theory that this class determines the Maurer-Cartan solution up to equivalence. In our context, equivalence in $\frakw$ is precisely given by the relation described in the statement of this Lemma.
\end{proof}

There is one final part to our construction, which is a transfer result allowing one to replace the bimodule by a quasi-isomorphic one:

\begin{lemma} \label{th:bimodule-transfer}
Let $\scrC \subset \tilde\scrB$ be an $A_\infty$-subalgebra, with associated bimodule $\tilde\scrP$. Suppose that we are given a bimodule $\scrP$ and quasi-isomorphism $\gamma: \scrP \rightarrow \tilde{\scrP}$. Then there is another $A_\infty$-algebra $\scrB$ containing $\scrC$, with associated bimodule $\scrP$, and an $A_\infty$-homomorphism $\scrG: \scrB \rightarrow \tilde{\scrB}$ such that $\scrG|\scrC = \mathit{id}$, and where the induced map of $A_\infty$-bimodules equals $\gamma$.
\end{lemma}

\begin{proof}
Let's set $\scrB = \scrC \oplus \scrP[1]$. The task of constructing the $A_\infty$-algebra $\scrB$ and $A_\infty$-homomorphism $\scrG$, with the prescribed properties, can be formalized as a solution of an extended Maurer-Cartan equation in a curved $L_\infty$-algebra
\begin{equation}
\frakt = W^1\mathit{Hom}(T(\scrB[1]), \scrB[1] \oplus \tilde{\scrB}).
\end{equation}
The situation is as in \eqref{eq:c-lie-algebra}, but we use the weight filtration rather than the length filtration. 
The associated graded space $W^r \frakt / W^{r+1} \frakt$ can be written as in \eqref{eq:weight-cone}, except that this time the pieces are 
\begin{equation}
\begin{aligned}
& \mathit{hom}_{(\scrC,\scrC)}(\scrP[2]^{\otimes_{\scrC} r}, \{\scrC \xrightarrow{\mathit{id}} \scrC\}[-1]) 
\\ \text{and } \;\; &
\mathit{hom}_{(\scrC,\scrC)}(\scrP[2]^{\otimes_{\scrC} r+1}, \{\scrP \xrightarrow{\gamma} \tilde\scrP\}[1] \big).
\end{aligned}
\end{equation}
By assumption, these are both acyclic.
\end{proof}

We can now bring all the pieces together:

\begin{proof}[Proof of Theorem \ref{th:relative-classification}] Take some $(\tilde\scrB,\tilde\scrF)$. As a chain complex, form
\begin{equation} \label{eq:tilde-b-cone}
\scrB = \mathit{Cone}\big(\scrA \xrightarrow{(-\mathit{id},\mathit{inclusion})} \scrA \oplus \tilde\scrB\big).
\end{equation}
We have an inclusion $\scrA \hookrightarrow \scrB$, which takes $a$ to $(a,0)$ in the second term of \eqref{eq:tilde-b-cone}. We also have a chain map $\scrG^1: \scrB \rightarrow \tilde{\scrB}$, which is zero on the first term in \eqref{eq:tilde-b-cone}, and $(a,\tilde{b}) \mapsto a+\tilde{b}$ on the second term. This is clearly a quasi-isomorphism, and $\scrG^1|\scrA$ is the inclusion $\scrA \hookrightarrow \tilde{\scrB}$. Using Lemma \ref{th:turn-into-inclusion}, we can equip $\scrB$ with an $A_\infty$-algebra structure compatible with $\scrA \hookrightarrow \scrB$, and so that $(\scrB,\mathit{inclusion})$ is quasi-isomorphic to $(\tilde{\scrB},\tilde{\scrF})$.

With that in mind, it is sufficient to focus on the case of $A_\infty$-algebras containing $\scrA$ as a subalgebra. Suppose we have two such algebras, such that the associated bimodule data are quasi-isomorphic, in the sense of \eqref{eq:pdelta-quasi-isomorphism}. Then, by applying Lemma \ref{th:bimodule-transfer}, we can replace these by quasi-isomorphic ones where the two bimodules that occur are actually the same, and their maps to $\scrA$ are homotopic. After that, Lemma \ref{th:first-order} completes the argument.
\end{proof}

\subsection{Curvature and $A_\infty$-deformations}
Let $\scrA$ be an $A_\infty$-algebra. A (curved) $A_\infty$-deformation consists of operations
\begin{equation}
\left\{
\begin{aligned}
& \mu^0_{\scrA_q} \in (\scrA[[q]])^2, \\
& \mu^1_{\scrA_q}: \scrA \longrightarrow (\scrA[1])[[q]], \\
& \mu^2_{\scrA_q}: \scrA^{\otimes 2} \longrightarrow \scrA[[q]], \\
& \dots
\end{aligned}
\right.
\end{equation}
which reduce to those of $\scrA$ if we set $q = 0$, meaning in particular that the curvature $\mu^0_{\scrA_q}$ has no $q$-constant term; and which satisfy the $A_\infty$-associativity equations extended by including that curvature term. An equivalent point of view is to take $\scrA_q = \scrA[[q]]$, and to consider the operations as multilinear ones living on that space. When we adopt that viewpoint, we refer to $\scrA_q$ as a curved $A_\infty$-algebra. There are corresponding $q$-notions of $A_\infty$-homomorphisms, homotopies between such homomorphisms, and bimodules (see e.g.\ \cite[Sections 3.2, 3.7, 4.2]{fooo}). Recall that the structure of an $A_\infty$-bimodule does not have its own curvature term; instead, the bimodule equations involve the curvature of the underlying $A_\infty$-algebra. For instance, given a bimodule $\scrP_q$ over $\scrA_q$, the failure of the differential on $\scrP_q$ to square to zero is measured by the formula
\begin{equation}
\mu^{0,1,0}_{\scrP_q}(\mu^{0,1,0}_{\scrP_q}(\cdot)) + (-1)^{|\cdot|} \mu^{1,1,0}_{\scrP_q}(\mu_{\scrA_q}^0,\,\cdot\,)
+ \mu^{0,1,1}_{\scrP_q}(\,\cdot\,\mu_{\scrA_q}^0) = 0.
\end{equation}
 Similarly, when we have a homomorphism $\scrF_q: \scrA_q \rightarrow \scrB_q$, its curvature term $\scrF_q^0 \in \scrB_q^1$ enters into the pullback of bimodules. Specifically, for $\scrP_q = \scrF_q^*\scrQ_q$, the differential is given by
\begin{equation}
\mu^{0,1,0}_{\scrP_q}(\cdot) = \sum_{s,r \geq 0} \mu^{s,1,r}_{\scrQ_q}(
\overbrace{\scrF_q^0,\dots,\scrF_q^0}^s,\,\cdot\,,\overbrace{\scrF_q^0,\dots,\scrF_q^0}^r).
\end{equation}

As before, we fix $\scrA_q$, and consider pairs consisting of some $\scrB_q$ and a curved $A_\infty$-homomorphism $\scrF_q: \scrA_q \rightarrow \scrB_q$. More precisely, we consider them up to an equivalence relation as in \eqref{eq:bf-quasi-isomorphism}, where $\scrG_q: \scrB_q \rightarrow \tilde{\scrB}_q$ is a filtered quasi-isomorphism (meaning, it is a curved $A_\infty$-homomorphism which reduces to a quasi-isomorphism if we set $q = 0$). To such a pair one associates an $\scrA_q$-bimodule and bimodule map, as in \eqref{eq:cone-bimodule}:
\begin{equation} \label{eq:cone-bimodule-2}
\left\{
\begin{aligned}
& \scrP_q = \big\{ \scrA_q \xrightarrow{\scrF_q} \scrF_q^*\scrB_q \big\}[-1], \\
& \rho_q: \scrP_q \longrightarrow \scrA_q.
\end{aligned}
\right.
\end{equation}
These are again considered up to filtered quasi-isomorphism. The map $\rho_q$ will always be ambidextrous, as in Lemma \ref{th:ambidextrous}. Hence, as in \eqref{eq:map-to-bimodule-data}, we have a map
\begin{equation} \label{eq:map-to-bimodule-data-2}
\left\{ 
\parbox{18em}{
pairs $(\scrB_q,\scrF_q)$ whose associated bimodule is filtered quasi-isomorphic to $\scrP_q$, up to \eqref{eq:bf-quasi-isomorphism}}\right\} \longrightarrow \mathit{Ambi}(\scrP_q)/\mathit{Aut}(\scrP_q).
\end{equation}

There is an analogue of Theorem \ref{th:relative-classification}, which we skip and proceed directly to the more concrete analogue of Corollary \ref{th:concrete-classification}:

\begin{corollary} \label{th:concrete-classification-2}
If $\scrP_q$ satisfies
\begin{equation} \label{eq:classification-groups-2}
\left.
\begin{aligned} 
& H^{2-2r}(\mathit{hom}_{(\scrA_q,\scrA_q)}(\scrP_q^{\otimes_{\scrA_q} r}, \scrA_q)) = 0 \\
& H^{3-2r}(\mathit{hom}_{(\scrA_q,\scrA_q)}(\scrP_q^{\otimes_{\scrA_q} r}, \scrP_q)) = 0
\end{aligned}
\right\}\; \text{for all $r \geq 2$.}
\end{equation}
then \eqref{eq:map-to-bimodule-data-2} is injective.
\end{corollary}

Note that, if we take $(\scrA,\scrP)$ to be the $q = 0$ specializations of $(\scrA_q,\scrP_q)$, then our original vanishing assumption \eqref{eq:classification-groups} implies \eqref{eq:classification-groups-2}, by a $q$-filtration argument.

\begin{example}
Suppose that 
\begin{equation} \label{eq:rank-one}
\mathit{Ambi}(\scrP) \iso \bQ[[q]] .
\end{equation}
Since $\bQ[[q]]^\times \subset \mathit{Aut}(\scrP)$, and since the order of vanishing of an element of \eqref{eq:rank-one} is invariant under the $\mathit{Aut}(\scrP)$, then we have $\mathit{Ambi}(\scrP)/\mathit{Aut}(\scrP) = \{0,1,2,\cdots,\infty\}$.
\end{example}

The proof of Corollary \ref{th:concrete-classification-2} follows the spirit of our earlier arguments, with certain departures in the details. We lean on our previous Lemma \ref{th:turn-into-inclusion}, and extend that to higher orders in $q$ as follows:

\begin{lemma} \label{th:turn-into-inclusion-2}
Take $A_\infty$-algebras $\scrC \subset \scrB$, and a quasi-isomorphism $\scrG: \scrB \rightarrow \tilde{\scrB}$. Set $\tilde\scrF = \scrG|\scrC$. Suppose that $(\scrC,\tilde\scrB,\tilde\scrF)$ come with curved $A_\infty$-deformations $\tilde\scrF_q: \scrC_q \rightarrow \tilde{\scrB}_q$. Then there are curved $A_\infty$-deformations $\scrB_q$ and $\scrG_q: \scrB_q \rightarrow \tilde{\scrB}_q$, such that $\scrC_q \subset \scrB_q$ is a curved $A_\infty$-subalgebra, and $\scrG_q|\scrC_q = \tilde{\scrF}_q$.
\end{lemma}

\begin{proof}
Consider
\begin{equation}
\frakc_q = \mathit{Hom}(T(\scrB[1])/T(\scrC[1]), \scrB[1] \oplus \tilde{\scrB})[[q]],
\end{equation}
which carries a curved $L_\infty$-algebra structure. The definition follows the same idea as in \eqref{eq:c-lie-algebra}. More explicitly, fix operations $\mu_{\scrB_q}^{\mathit{given},d}$ (including $d = 0$) which restrict to those of $\scrC_q$, and which are $q$-deformations of those on $\scrB$. Similarly, fix some $\scrG_q^{\mathit{given},d}$ which restrict to $\tilde{\scrF}_q$, and which are $q$-deformations of $\scrG$. These determine the curved $L_\infty$-structure: for instance, the curvature term $\lambda^0_{\frakc_q}$ is the failure of the ``given'' data to satisfy the requirements ($A_\infty$-associativity and $A_\infty$-homomorphism equations). 
Repairing this failure amounts to solving an extended Maurer-Cartan equation for an element 
\begin{equation}
(\mu_{\scrB_q}^{\mathit{new}}, \scrG_q^{\mathit{new}}) \in q\frakc_q^1.
\end{equation}
If we set $q = 0$, the resulting $\frakc_q/q\frakc_q$ is a version of \eqref{eq:c-lie-algebra} without the $>1$ restriction. As a chain complex, one can filter it by tensor length, and the quotients are acyclic, as in \eqref{eq:p-filtered}. It follows that $\frakc_q/q\frakc_q$ is acyclic, or in a different terminology, that $\frakc_q$ is filtered acyclic. The rest is once more standard obstruction theory: filtered acyclicity ensures that we can always find a solution to the extended Maurer-Cartan equation.
\end{proof}

\begin{lemma} \label{th:first-order-2}
(A generalization of \cite[Lemma 2.6]{seidel15}) Fix $\scrC_q$ and a bimodule $\scrP_q$, such that \eqref{eq:classification-groups-2} holds. Consider curved $A_\infty$-structures on $\scrB_q = \scrC_q \oplus \scrP_q[1]$ which extend the given structure on $\scrC_q$, and induce the given bimodule structure on $\scrP_q$. If two such structures give rise to the same class $[\rho_q] \in H^0(\hom_{(\scrC_q,\scrC_q)}(\scrP_q,\scrC_q))$, they must be related by an $A_\infty$-homomorphism which restricts to the identity on $\scrC_q$ (hence has no curvature term), and which induces the identity on $\scrP_q$.
\end{lemma}

This is the analogue of Lemma \ref{th:first-order}, and uses the same method of proof: one uses $\scrC_q$ and $\scrP_q$ to define a dg Lie algebra
\begin{equation} \label{eq:frakw-2}
\frakw_q = W^1\mathit{Hom}(T(\scrB_q[1]), \scrB_q[1]),
\end{equation}
which comes with its complete decreasing filtration $W^r \frakw_q$. The assumption \eqref{eq:classification-groups-2} provides the necessary information about the groups $H^1(W^r \frakw_q/W^{r+1} \frakw_q)$. Similar considerations appear in the analogue of Lemma \ref{th:bimodule-transfer}: 

\begin{lemma} \label{th:bimodule-transfer-2}
Let $\scrC_q \subset \tilde\scrB_q$ be a curved $A_\infty$-subalgebra, with associated bimodule $\tilde\scrP_q$. Suppose that we are given a bimodule $\scrP_q$ and quasi-isomorphism $\gamma_q: \scrP_q \rightarrow \tilde{\scrP}_q$. Then there is another curved $A_\infty$-algebra $\scrB_q$ containing $\scrC_q$, with associated bimodule $\scrP_q$, and an $A_\infty$-homomorphism $\scrG_q: \scrB_q \rightarrow \tilde{\scrB}_q$ such that $\scrG_q|\scrC_q = \mathit{id}$, and where the induced map of $A_\infty$-bimodules equals $\gamma_q$.
\end{lemma}

Finally, these ingredients combine to yield Corollary \ref{th:concrete-classification-2} in the same way as before; we omit the details. 

We also want to consider a more specific situation, following \cite{seidel14,seidel15}. Namely, suppose that $\scrC_q = \scrC[[q]]$ is the trivial $q$-linear extension of some $A_\infty$-structure on $\scrC$, and similarly for $\scrP_q = \scrP[[q]]$. Given some curved $A_\infty$-structure on $\scrB_q = \scrC_q \oplus \scrP_q[1]$ which extends those data (in the sense of Lemma \ref{th:first-order-2}), one can write 
\begin{equation}
[\rho_q] \in H^0(\mathit{hom}_{(\scrC,\scrC)}(\scrP,\scrC))[[q]].
\end{equation}

\begin{lemma} \label{th:field-of-definition}
(A mild generalization of \cite[Lemma 2.7]{seidel15}; see also \cite[Lemma A.12]{seidel15} for a discussion in more general terms.) Assume that
\begin{equation} \label{eq:one-degree-higher}\left\{
\begin{aligned} 
& H^{3-2r}(\mathit{hom}_{(\scrC,\scrC)}(\scrP^{\otimes_{\scrC} r}, \scrC)) = 0 && \text{for $r \geq 2$,} \\
& H^{4-2r}(\mathit{hom}_{(\scrC,\scrC)}(\scrP^{\otimes_{\scrC} r}, \scrP)) = 0 && \text{for $r \geq 3$.}
\end{aligned}
\right.
\end{equation}
Take a $\bQ$-linear subspace $V \subset \bQ[[q]]$, and an element $[\rho_q] \in \mathit{Ambi}(\scrP) \otimes V$. Then, there is some $\scrB_q$ whose associated bimodule map is $[\rho_q]$, and such that for each $r \geq 1$, the part of the $A_\infty$-structure on $\scrB_q$ which has weight $r$ has coefficients which lie in $\mathit{Sym}^r(V) \subset \bQ[[q]]$.
\end{lemma}

\begin{proof} 
Because the weight filtration on $\scrB = \scrC \oplus \scrP[1]$ comes with a preferred splitting, the dg Lie algebra \eqref{eq:frakw} is bigraded. More precisely, one can write $\frakw = \prod_{r \geq 1} W^{=r}\frakw$, where $W^{=r}$ stands for the part that increases weights exactly by $r$. We use that to define another dg Lie algebra,
\begin{equation} \label{eq:v-algebra}
\frakw_V = \prod_r W^{=r} \frakw \otimes \mathit{Sym}^r V \subset \frakw[[q]].
\end{equation}
By \eqref{eq:les-weight}, we can lift our class $[\rho_q]$ to $H^0(W^{=1}\frakw) \otimes V = H^0(W^1\frakw/W^2 \frakw) \otimes V$, and then represent it by an element of $W^{=1} \frakw \otimes V \subset \frakw_V$. We then extend that to a Maurer-Cartan element in $\frakw_V$, order by order in $r$. At each step, the obstruction lies in $H^2(W^{=r} \frakw_V) \iso H^2(W^r \frakw/W^{r+1} \frakw) \otimes \mathit{Sym}^r(V)$. From \eqref{eq:les-weight}, one sees that \eqref{eq:one-degree-higher} implies the vanishing of those cohomology groups. 
\end{proof}

It will occasionally be useful to be able to transfer this result across quasi-isomorphisms, in the following sense.

\begin{lemma} \label{th:v-transfer}
Take a pair $(\scrC,\scrP)$ consisting of an $A_\infty$-algebra a bimodule over it, as well as another such pair $(\tilde{\scrC},\tilde{\scrP})$, together with a quasi-isomorphism of pairs between them. Suppose that we have an $A_\infty$-structure on $(\scrC \oplus \scrP[1])[[q]]$ which is nondecreasing with respect to weights, and such that the weight $r$ part has coefficients in $\mathit{Sym}^r(V) \subset \bQ[[q]]$. Then there is an $A_\infty$-structure on 
\begin{equation} \label{eq:tilde-double}
(\tilde{\scrC} \oplus \tilde{\scrP}[1])[[q]]
\end{equation}
with the same property, as well as a filtered quasi-isomorphism between the two; and the filtered quasi-isomorphism also has the same properties with respect to weights.
\end{lemma}

\begin{proof}
This is fairly straightforward. There is a curved $L_\infty$-algebra similar to \eqref{eq:v-algebra} which governs, simultaneously, the construction of the $A_\infty$-structure on \eqref{eq:tilde-double} and that of the quasi-isomorphism. With respect to the filtration by weights, that $L_\infty$-algebra is filtered acyclic, allowing us to solve the construction problem order by order.
\end{proof}

\section{Differentiation operations in the $A_\infty$-context\label{sec:connections}}
The aim of this section is to provide the context on $A_\infty$-connections needed in Section \ref{subsec:directed-2}. This is entirely elementary material, and we will explain the basic features while keeping the discussion relatively narrow. For some related aspects, notably the process by which an $A_\infty$-connection gives rise to connections such as \eqref{eq:nabla-2} on Hochschild-type groups, we refer to \cite{seidel21b}.

\subsection{Definitions}
Let $\scrA$ be an $A_\infty$-algebra over $\bQ$, and $\scrA_q$ a (curved) deformation. To define Hochschild cohomology $\mathit{HH}^*(\scrA_q,\scrA_q)$, we use the standard chain complex 
\begin{equation} \label{eq:hochschild-complex}
\begin{aligned}
& \mathit{CC}^*(\scrA_q,\scrA_q) = \mathit{Hom}(T(\scrA_q[1]),\scrA_q), \\
& (\delta \gamma_q)^d(a_d,\dots,a_1) = \\ & \qquad\qquad 
\textstyle \sum_{ij} (-1)^{\|\gamma_q\|+\|a_1\|+\cdots+\|a_i\|} \gamma_q^{d-j+1}(a_d,\dots,\mu_{\scrA_q}^j(a_{i+j},\dots,a_{i+1}),\dots,a_1) \\
& \qquad \qquad -\textstyle \sum_{ij}
(-1) ^{\|\gamma_q\| (\|a_1\| + \cdots +\|a_i \|)} \mu^{d-j+1}_{\scrA_q}(a_d,\dots,\gamma_q^j(a_{i+j},\dots,a_{i+1}),\dots,a_1).
\end{aligned}
 \end{equation}

An $A_\infty$-pre-connection $\calnablaq$ on $\scrA_q$ consists of
\begin{equation} \label{eq:calnabla-def}
\left\{
\begin{aligned} 
& \calnablaq^0 \in \scrA_q^1, \\
& \calnablaq^1: \scrA_q \longrightarrow \scrA_q, \\
& \calnablaq^2: \scrA_q^{\otimes 2} \longrightarrow \scrA_q[-1], \\
& \cdots
\end{aligned}
\right.
\end{equation}
such that each $\calnablaq^d$, $k \neq 1$, is $q$-linear, while 
\begin{equation} \label{eq:q-derivation}
\calnablaq^1( g a) = g \calnablaq^1 a + g' \, a \;\; \text{ for } g \in \bQ[[q]], \, a \in \scrA_q.
\end{equation}
Equivalently, a pre-connection is of the form $\calnablaq = \partial_q + \alpha_q$, where the first term is ordinary differentiation as a linear endomorphism, and $\alpha_q \in \mathit{CC}^1(\scrA_q,\scrA_q)$. Even though the pre-connection itself is not a Hochschild cochain, one can apply the Hochschild differential to it, which yields an element $\delta(\calnablaq) \in \mathit{CC}^2(\scrA_q,\scrA_q)$. Clearly, its cohomology class
\begin{equation}
\kappa_q = [\delta \calnablaq] \in \mathit{HH}^2(\scrA_q,\scrA_q)
\end{equation}
is independent of all choices. This is what we call the Kaledin class. If one takes $\calnablaq = \partial_q$, the corresponding representative of the Kaledin class is just the ordinary $q$-derivative $\partial_q \mu_{\scrA_q}$ of the $A_\infty$-structure. An $A_\infty$-connection is one such that $\partial(\calnablaq) = 0$. Two such connections are homotopic if their difference is a Hochschild coboundary. Clearly, an $A_\infty$-connection exists if and only if the Kaledin class is zero; and in that case, the homotopy classes of connections form an affine space over $\mathit{HH}^1(\scrA_q,\scrA_q)$.

\begin{lemma} \label{th:functoriality-of-connections}
Consider an isomorphism $\scrF_q: \scrA_q \rightarrow \tilde{\scrA}_q$ of curved $A_\infty$-algebras (this means that if we specialize to $q = 0$, $\scrF^1: \scrA \rightarrow \tilde{\scrA}$ is an isomorphism). Then, there is an induced bijection between pre-connections on both sides; and the associated representatives for the Kaledin class also correspond to each other, under the induced isomorphism of Hochschild complexes. In particular, we get a bijection between actual connections on both sides, and that is compatible with the notion of homotopy.
\end{lemma}

\begin{proof}
Alongside the standard Hochschild complexes of the two algebras, there is a mixed version (the Hochschild complex of $\scrA$ with coefficients in the pullback via $\scrF$ of the diagonal bimodule of $\tilde{\scrA}$):
\begin{equation}
\begin{aligned}
& \mathit{CC}^*(\scrA_q,\scrF_q^*\tilde\scrA_q) = \mathit{Hom}(T(\scrA_q[1]),\tilde{\scrA}), \\
& (\delta\phi_q)^d(a_d,\dots,a_1) = 
\\ & \qquad
\textstyle \sum_{ij} (-1)^{\|\phi_q\|+\|a_1\|+\cdots+\|a_i\|} \phi_q^{d-j+1}(a_d,\dots,\mu_{\scrA_q}^j(a_{i+j},\dots,a_{i+1}),\dots,a_1) \\
& \qquad -\textstyle \sum_{k,\, i_1+\cdots+i_k = d,\, l}
(-1) ^{\|\phi_q\| (\|a_1\| + \cdots +\|a_{i_1+\cdots+i_{l-1}}\|)} \mu^{d-i_1-\cdots-i_k+k}_{\tilde\scrA_q}(
\scrF_q^{i_k}(a_d,\dots,a_{d-i_k+1}),\\
& \qquad \qquad \qquad \qquad \dots,
\phi_q^{i_l}(a_{i_1+\cdots+i_l},\dots,a_{i_1+\cdots+i_{l-1}+1}),\dots,\scrF_q^{i_1}(a_{i_1},\dots,a_1)).
\end{aligned}
\end{equation}
It comes with isomorphisms
\begin{equation} \label{eq:back-and-forth}
\mathit{CC}^*(\scrA_q,\scrA_q) \longrightarrow \mathit{CC}^*(\scrA_q,\scrF_q^*\tilde{\scrA}_q) \longleftarrow \mathit{CC}^*(\tilde{\scrA}_q,\tilde{\scrA}_q).
\end{equation}
Explicitly, the image of $\gamma \in \mathit{CC}^*(\scrA_q,\scrA_q)$ under the map on the left is
\begin{equation} \label{eq:back}
(a_d,\dots,a_1) \longmapsto \textstyle \sum_{ij} (-1)^{\|\phi_q\|(\|a_1\|+\cdots+\|a_i\|)} \scrF_q^{d-j+1}(a_d,\dots,\gamma_q(a_{i+j},\dots,a_{i+1}),\dots,a_1);
\end{equation}
and the map on the right in \eqref{eq:back-and-forth} sends $\tilde{\gamma}$ to
\begin{equation} \label{eq:forth}
\textstyle \sum_{k,\, i_1+\cdots+i_k = d} \tilde{\gamma}_q^k(\scrF_q^{i_k}(a_d,\dots,a_{d-i_k+1}),\dots, \scrF_q^{i_1}(a_{i_1},\dots,a_1)).
\end{equation}
Arguing order-by-order in $q$, and then with respect to the number of inputs, shows that both maps are indeed isomorphisms.

One can introduce a notion of connection over $\scrF_q$, which is a sequence of maps
\begin{equation}
\begin{aligned}
& \relativenablaq^d: \scrA_q^{\otimes d} \longrightarrow \tilde{\scrA}_q[1-d], \\
& \relativenablaq^d(a_d,\dots, g a_j,\dots,a_1) = g \relativenablaq^d(a_d,\dots,a_1) + g'\, \scrF_q^d(a_d,\dots,a_1).
\end{aligned}
\end{equation}
The formulae from \eqref{eq:back} and \eqref{eq:forth} produce, starting with a pre-connection on $\scrA_q$ or $\tilde{\scrA}_q$, such a relative connection. Again, these are bijective correspondences. With that at hand, the desired properties are easy to see.
\end{proof} 

\begin{lemma} \label{th:functoriality-of-connections-2}
Suppose that $\scrF_q: \scrA_q \rightarrow \tilde{\scrA}_q$ is a filtered quasi-isomorphism. Then, the associated isomorphism $\mathit{HH}^*(\scrA_q,\scrA_q) \iso \mathit{HH}^*(\tilde{\scrA}_q,\tilde{\scrA}_q)$ takes one Kaledin class to the other. If those classes vanish, there is a preferred bijection between homotopy classes of $A_\infty$-connections on both sides.
\end{lemma}

\begin{proof}
We can use exactly the same strategy as in Lemma \ref{th:functoriality-of-connections}. The only difference is that the maps in \eqref{eq:back-and-forth} are now only filtered quasi-isomorphisms, which explains the weaker outcome.
\end{proof}

\subsection{Gauge equivalence classes}
Everything we have said previously would work over an arbitrary coefficient field. At this point, we start using the fact that we are working over $\bQ$. The starting point is \cite[Proposition 7.3(b)]{lunts10}:

\begin{lemma} \label{th:kaledin-2}
The Kaledin class is zero if and only if $\scrA_q$ is a trivial $A_\infty$-deformation.
\end{lemma}

This, together with Lemmas \ref{th:functoriality-of-connections}, means that we can always restrict to $\scrA_q = \scrA[[q]]$. In that case, a connection has the form $\calnablaq = \partial_q + \alpha_q$, where $\alpha_q \in \mathit{CC}^1(\scrA,\scrA)[[q]]$ is a cocycle.

\begin{lemma} \label{th:unique-gauge}
Any two $A_\infty$-connections on $\scrA_q$ are related by an automorphism of $\scrA_q$ which is the identity modulo $q$. Moreover, that automorphism is unique.
\end{lemma}

\begin{proof}
Without loss of generality, take $\scrA_q =\scrA[[q]]$. Suppose that $\alpha_q$ is of order $q^m$ (meaning, all terms of lower order are zero). Let $\scrF_q$ be an automorphism of $\scrA[[q]]$ which is equal to the identity up to (and including) order $q^m$. By going through the formulae in Lemma \ref{th:functoriality-of-connections}, one sees that this transforms our connection into a new one with
\begin{equation} \label{eq:ftransf}
\tilde{\alpha}_q = \alpha_q - (\partial_q \scrF_q) + O(q^{m+1}).
\end{equation}
Recall that Hochschild cocycles can be regarded as infinitesimal automorphisms. If such a cocycle has order $q^{m+1}$ for some $m \geq 0$, one can exponentiate it, which yields an automorphism that is equal the identity, plus our cocycle, plus terms of order $q^{2m+2}$. By applying this to $q\alpha_q/(m+1)$ and inserting that into \eqref{eq:ftransf}, one obtains a new connection with $\tilde{\alpha}_q$ of order $q^{m+1}$. One can iterate that process, and the composition of the resulting infinite sequence of automorphisms makes sense, yielding an overall automorphism that transforms any given connection to the trivial one $\partial_q$.

The last step is to show that there are no nontrivial automorphisms of the kind appearing in the Lemma, which preserve the trivial connection. But that's obvious, since such an automorphism must be $q$-constant.
\end{proof}

\section{Hamiltonian Floer cohomology for Lefschetz fibrations\label{sec:mclean}}
This section reviews the groups $\mathit{HF}^*_{q,q^{-1}}(\bar{p},r)$ from \eqref{eq:all-floer}, following \cite{seidel17}, and adds more concrete computational material along the lines of \cite{mclean12}. We only deal with fibrations which have trivial monodromy at infinity, which makes both the technical arguments and the outcome of the computations particularly simple. None of this material is substantially new, but it serves as a reminder, and sets up our later discussion of the relative version of those Floer groups.

\subsection{A topological barrier trick}
We start with elementary technical ingredients. Fix some $r \in \bR$ and $\rho>0$. Consider maps
\begin{equation} \label{eq:toy-floer}
\left\{
\begin{aligned}
& 
u: \bR \times S^1 \longrightarrow \bC,
\\ &
\textstyle \lim_{s \rightarrow \pm\infty} u(s,t) = x_{\pm}(t), \text{ where $x_\pm$ are loops disjoint from $\{|w|=\rho\} \subset \bC$.}
\\ &
\partial_s u + i(\partial_t - 2\pi i r) u = 0 \text{ in a neighbourhood of $u^{-1}(\{|w|=\rho\})$.}
\end{aligned}
\right.
\end{equation}

\begin{lemma} \label{th:winding-number-ineq}
Suppose $u$ intersects $\{|w| = \rho\}$ transversally, and set $\Omega = u^{-1}(\{|w| \geq \rho\})$. Then we have the winding number inequality
\begin{equation} \label{eq:winding-number-ineq}
\mathrm{wind}(u|\partial \Omega) \leq r\,\mathrm{wind}(\partial \Omega),
\text{ with equality only if $\partial \Omega = \emptyset$.}
\end{equation}
The notation requires some explanation. By assumption, $\partial \Omega = u^{-1}(\{|w| = \rho\})$ is the union of finitely many oriented loops $C_1,\dots,C_k \subset \bR \times S^1$. On the left hand side of \eqref{eq:winding-number-ineq}, we consider the winding number (or degree) of $u|C_j$, and add up all those numbers. On the right hand side, we similarly add up the winding numbers (degrees) of the projections $C_j \rightarrow S^1$.
\end{lemma}

\begin{proof} Take
\begin{equation} \label{eq:rotate-u}
\tilde{u}(s,t) = \exp(-2\pi i r t) u(s,t): \bR^2 = \bC \longrightarrow \bC.
\end{equation}
This is holomorphic in a neighbourhood of $\widetilde{\partial\Omega} = \tilde{u}^{-1}(\{|w| = \rho\})$. Therefore, $\tilde{u}|\widetilde{\partial\Omega}$ moves along the circle $\{|w| = \rho\}$ clockwise with nonzero speed (here, the choice of orientation is important). Take any component $C \subset \partial\Omega$, parametrize it by $c = c(\theta): S^1 \rightarrow C$, and lift that to a map $\tilde{c}: [0,1] \rightarrow \bR^2$. Since $\tilde{u}(\tilde{c}(\theta))$ moves clockwise, we see from \eqref{eq:rotate-u} that
\begin{equation}
u(\tilde{c}(\theta)) = \rho\exp(2\pi i \gamma(\theta)) \;\; \text{ with }
\frac{d\gamma}{d\theta} < r \frac{d\mathrm{im}(\tilde{c})}{d\theta}.
\end{equation}
As a consequence,
\begin{equation}
\mathrm{wind}(u|C) = \gamma(1) - \gamma(0) < r\big(\mathrm{im}(\tilde{c}(1))-\mathrm{im}(\tilde{c}(0))\big) = r\, \mathrm{wind}(C).
\end{equation}
\end{proof}

\begin{lemma} \label{th:barrier}
(i) If both $x_{\pm}$ lie inside $\{|w| < \rho\}$, then all of $u$ lies in that subset.

(ii) If both $x_{\pm}$ lie outside $\{|w| \leq \rho\}$, then $\mathrm{wind}(x_-) \leq \mathrm{wind}(x_+)$; moreover, if those two winding numbers are equal, then all of $u$ lies outside $\{|w| \leq \rho\}$.

(iii) If $x_+$ lies in $\{|w| \leq \rho\}$ and $x_-$ lies outside it, then $\mathrm{wind}(x_-) < r$.
\end{lemma}

\begin{proof}
(i) We can perturb $\rho$, making it slightly smaller, and thereby achieve the transversality condition required to apply Lemma \ref{th:winding-number-ineq}. Because of the limit assumptions, $\Omega$ is compact, and as a consequence, $\mathrm{wind}(u|\partial \Omega)$ and $\mathrm{wind}(\partial \Omega)$ are both zero. This implies $\partial \Omega = \emptyset$ by \eqref{eq:winding-number-ineq}, and therefore $\Omega = \emptyset$.

(ii) In this case, we make $\rho$ slightly larger to achieve transversality. By assumption, $\Omega$ contains all points $(s,t)$ where $\pm s \gg 0$, and can therefore be compactificed by adding circles $\{\pm \infty\} \times S^1$. As a consequence,
$\mathrm{wind}(u|\partial \Omega) = \mathrm{wind}(x_-) - \mathrm{wind}(x_+)$, while $\mathrm{wind}(\partial \Omega) = 0$ as before. Hence, \eqref{eq:winding-number-ineq} implies that $\mathrm{wind}(x_-) \leq \mathrm{wind}(x_+)$, with equality only if $\partial \Omega = \emptyset$, in which case we must have $\Omega = \bR \times S^1$.

(iii) Here, we can compactify $\Omega$ just by adding $\{-\infty\} \times S^1$. Therefore, $\mathrm{wind}(u|\partial \Omega) = \mathrm{wind}(x_-)$, whereas $\mathrm{wind}(\partial \Omega) = 1$. One again ends by applying \eqref{eq:winding-number-ineq}.
\end{proof}

\subsection{Symplectic geometry of the hyperbolic disc\label{subsec:hyperbolic}}
Let's assemble some basic terminology concerning the hyperbolic disc (see \cite[Sections 3,4, and 8]{seidel17} or \cite[Section 3]{seidel18} for closely related expositions). Take the Lie group
\begin{equation}
G = \mathit{PU}(1,1) = \left\{ \textstyle g = \left(\begin{smallmatrix} a & b \\ \bar{b} & \bar{a} \end{smallmatrix}\right)\, : \, a,b \in \bC, \;
|a|^2 - |b|^2 = 1 \right\}/\pm\!\mathit{id},
\end{equation}
which is isomorphic to $\mathit{PSL}_2(\bR)$. We denote its universal cover by $\tilde{G}$. Any $\tilde{g} \in \tilde{G}$ has a rotation number $\mathrm{rot}(\tilde{g}) \in \bR$, which satisfies
\begin{equation}
\mathrm{rot}(\tilde{g}) \in \begin{cases} \bR \setminus \bZ & \text{if } g \text{ (the image of $\tilde{g}$ in $G$) is elliptic,} \\ \bZ & \text{otherwise.} \end{cases}
\end{equation}
The Lie algebra of $G$ is
\begin{equation}
\frakg = \mathfrak{u}(1,1) = \big\{ \gamma = \left(\begin{smallmatrix} i\alpha & \beta \\ \bar{\beta} & -i\alpha \end{smallmatrix}\right) \;:\;
\alpha \in \bR, \;\beta \in \bC \big\}.
\end{equation}
The nonnegative cone $\frakg_{\geq 0} \subset \frakg$ consists of those $\gamma$ such that $\alpha \geq |\beta|$ (the ``square root'' of the condition $\mathrm{det}(\gamma) \geq 0$). A path $(g_t)$ in $G$ is called nonnegative if $(dg_t/dt) g_t^{-1} \in \frakg_{\geq 0}$ for all $t$. 

Suppose that we are given a loop $(\gamma_t)_{t \in S^1}$ in $\frakg$. By integrating the differential equation 
\begin{equation} \label{eq:define-monodromy}
(dg_t/dt)g_t^{-1} = \gamma_t, \;\; g_0 = \mathit{id},
\end{equation}
one gets an element $g = g_1 \in G$, called the monodromy of $(\gamma_t)$. Any two loops with conjugate monodromies are gauge equivalent, meaning that they are related by a transformation
\begin{equation} \label{eq:gauge}
\gamma_t \longmapsto h_t^{-1} \gamma_t h_t - h_t^{-1} (\partial_t h_t),
\;\; \text{for $(h_t)$ a loop in $G$.}
\end{equation}
The monodromy element also comes with a preferred lift $\tilde{g}$ to $\tilde{G}$. Two loops $(\gamma_t)$ with conjugate $\tilde{g}$ are $\tilde{G}$-gauge equivalent, which means they are related by a transformation \eqref{eq:gauge} in which $(h_t)$ is a contractible loop.
%

Let $W$ be the open unit disc in $\bC$. We denote its closure by $W\|$. For future use, we also want to introduce the notation
\begin{equation}
\begin{aligned} 
& W_{\leq \rho} = \{w \in W\;:\; |w| \leq \rho\}, \\
& W_{\geq \rho} = \{w \in W\;:\; |w| \geq \rho\}.
\end{aligned}
\end{equation}
The group $G$ acts on $W$ by
\begin{equation}
g(w) = \frac{aw+b}{\bar{b}w + \bar{a}}.
\end{equation}
The action preserves the complex structure and the symplectic form
\begin{equation} \label{eq:hyperbolic}
\omega_W = \frac{d\mathrm{re}(w) \wedge d\mathrm{im}(w)}{(1-|w|^2)^2}.
\end{equation}
The action extends smoothly to $W\|$. On the infinitesimal level, one gets a homomorphism from $\frakg$ to the Lie algebra $\smooth(W,\bR)$ with the Poisson bracket, and from there to Hamiltonian vector fields: 
\begin{align} \label{eq:h-gamma}
& H_\gamma = \frac{1}{1-|w|^2} \left( \half(1+|w|^2)\alpha - \mathrm{im}(\beta w) \right), \\ 
& X_\gamma = (\bar{\beta} + 2i\alpha w - \beta w^2) \, \partial_w.
\end{align}
The nonnegative elements are precisely those for which $X_\gamma$ is a nonnegative vector field on the boundary of $W\|$ (meaning, $X_\gamma$ is either zero or anticlockwise-pointing at each boundary point). Equivalently, they are those which satisfy $H_\gamma \geq 0$. 

We pick two sample elements of $\frakg$ with different behaviour (elliptic and hyperbolic):
\begin{align} \label{eq:sample-gamma-1}
& \gamma^{\mathit{rot}} = \mathrm{diag}(\pi,-\pi),\;\; \textstyle H^{\mathit{rot}} = H_{\gamma^{\mathit{rot}}} = \frac{\pi}{2} \frac{1+|w|^2}{1-|w|^2}, \\
& \label{eq:sample-gamma-2}
\gamma^{\mathit{hyp}} = \textstyle\left(\begin{smallmatrix} 0 & i\pi \\ -i\pi & 0 \end{smallmatrix}\right),
\;\;
H^{\mathit{hyp}} = H_{\gamma^{\mathit{hyp}}} = \textstyle \frac{\pi\mathrm{re}(w)}{1-|w|^2}.
\end{align}
Here, \eqref{eq:sample-gamma-1} has been normalized so that its Hamiltonian vector field is the rotational vector field with one-periodic flow, meaning $2\pi i w\partial_w$; and we have matched the same constant in \eqref{eq:sample-gamma-2}, in order to achieve that
\begin{equation} \label{eq:h-functions}
-\! H^{\mathit{rot}} < H^{\mathit{hyp}} < H^{\mathit{rot}}.
\end{equation}

\begin{lemma} \label{th:morse-function}
Fix some $\rho_1 \in (0,1)$. Then there are $\rho_2 \in (\rho_1,1)$, $H_2 \in \smooth(W,\bR)$, such that:
\begin{align}
& H_2 \geq -H^{\mathit{rot}}\text{ everywhere, with equality on $W_{\leq \rho_1}$;} \\
& H_2 = H^{\mathit{hyp}} \text{ on $W_{\geq \rho_2}$;} \\
& \mybox{$H_2$ is a Morse function with exactly two critical points: a local maximum (at the origin, inherited from $-H^{\mathit{rot}}$), and a saddle point $w_2$ (in the annulus $\rho_1 < |w_2| < \rho_2$, where neither of the previous equalities applies).}
\end{align}
\end{lemma}

\begin{proof}
This is done by a simple interpolation:
\begin{equation} \label{eq:interpolation}
\begin{aligned}
& H_2 = -H^{\mathit{rot}} + \psi_2(|w|) (H^{\mathit{hyp}} + H^{\mathit{rot}}), \\
& \psi_2(\rho) = 0 \text{ for $\rho \leq \rho_1$,} \;\; 0<\psi_2(\rho)< 1 \text{ for $\rho_1<\rho<\rho_2$}, \;\;
\psi_2(\rho) = 1 \text{ for $\rho \geq \rho_2$.}
\end{aligned}
\end{equation}
If $w$ is a critical point of $H_2$ and $|w| > \rho_1$, then $(\nabla H^{\mathit{hyp}})_{w}$ must point in radial direction, so as to be in linear dependence from the other terms in $(\nabla H_2)_w$; and this means that $w \in (-1,1)$. It is elementary to find a $\psi_2$ (for $\rho_2$ close to $1$) such that $H_2|(-1,1)$ has exactly one critical point $w_2$ other than the origin; that point satisfies $\rho_1<w_2<\rho_2$ and is a nondegenerate local minimum of $H_2|(-1,1)$ (see Figure \ref{fig:interpolate}). By differentiating \eqref{eq:interpolation} along the circle of radius $|w_2|$, one sees that the Hessian at $w_2$ in imaginary direction is negative. As a consequence, $w_2$ is a nondegenerate saddle point of $H_2$.
\end{proof}
\begin{figure}
\begin{centering}
\includegraphics{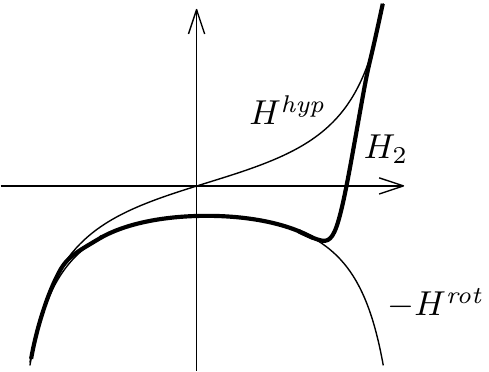}
\caption{\label{fig:interpolate}The function from \eqref{eq:interpolation}, restricted to real values of $w$.}
\end{centering}
\end{figure}

\begin{lemma} \label{th:morse-function-2}
(The notation continues from the previous Lemma.)
There are $\rho_3 \in (\rho_2,1)$, $H_3 \in 
\smooth(W,\bR)$, such that:
\begin{align}
& \label{eq:everywhere}
H_3 \geq H_2 \text{ everywhere, with equality on $W_{\leq \rho_2}$;} 
\\ &
H_3 = H^{\mathit{rot}} \text{ on $W_{\geq \rho_3}$;}
\\ &
\mybox{$H_3$ is a Morse function with exactly three critical points. Two are inherited from $H_2$. The third one, called $w_3$, is a local minimum (and satisfies $\rho_2 < |w_3| < \rho_3$).}
\end{align}
\end{lemma}

\begin{proof}
We use the same approach as in \eqref{eq:interpolation},
\begin{equation}
\begin{aligned} \label{eq:interpolation-2}
& H_3 = H_2 + \psi_3(|w|) (H^{\mathit{rot}}-H_2), \\
& \psi_3(\rho) = 0 \text{ for $\rho \leq \rho_2$,} \;\; 0<\psi_3(\rho)< 1 \text{ for $\rho_2<\rho<\rho_3$}, \;\;
\psi_3(\rho) = 1 \text{ for $\rho \geq \rho_3$.}
\end{aligned}
\end{equation}
This time, one can choose $\psi_3$ so that $H_3|(-1,1)$ has a nondegenerate local minimum at some point $w_3 \in (-\rho_3,-\rho_2)$. Because of the difference in the sign of $w_3$, the Hessian in imaginary direction is now positive, so what we get is a local minimum of $H_3$.
\end{proof}

\subsection{Hamiltonian Floer groups\label{subsec:hamiltonian-review}}
With those preliminaries at hand, we summarize the definition of the Floer cohomology groups, and prove Lemmas \ref{th:les-2} and \ref{th:les-1}.

\begin{convention} \label{th:nullhomologous}
Throughout the discussion of Hamiltonian Floer cohomology, we will consider only those one-periodic orbits $x$ which are nullhomologous. This means that there is an oriented compact surface $S$ with $\partial S \iso S^1$, and a map $v$ from that surface to our symplectic manifold, with $v|\partial S = x$; which allows us to define many quantities, such as the Conley-Zehnder index of $x$, in terms of such $v$. (This assumption is not absolutely necessary, but simplifies some arguments.)
\end{convention}

As in \cite{seidel17}, we consider symplectic manifolds that map to the hyperbolic disc, rather than to the complex plane. From a purely symplectic geometry viewpoint this distinction is irrelevant, since the two spaces are symplectically isomorphic; but it affects the class of almost complex structures that will be used.

\begin{setup} \label{th:fibration}
Take a symplectic manifold $\bar{E}$ with integral symplectic class $[\omega_{\bar{E}}]$ and vanishing first Chern class, together with a proper map
\begin{equation}
\bar{p}: \bar{E} \longrightarrow W,
\end{equation}
which is a trivial symplectic fibration over $W_{\geq 1/2}$, with respect to \eqref{eq:hyperbolic}. This means that if we set $\bar{M} = \bar{E}_{1/2} = \bar{p}^{-1}(1/2)$, there is a (unique) symplectic isomorphism extending the inclusion $\bar{M} \subset \bar{E}$, which fits into a commutative diagram
\begin{equation} \label{eq:outside-trivialization}
\xymatrix{
W_{\geq 1/2} \times \bar{M} \ar[rr]^-{\iso} \ar[dr] && 
\bar{E}|W_{\geq 1/2} \ar[dl] \\
& W_{\geq 1/2}. &
}
\end{equation}
We will usually treat this isomorphism as an inclusion, writing $W_{\geq 1/2} \times \bar{M} \subset \bar{E}$. As one consequence of this structure, there is an obvious compactification $\cornerdoublebar{E} \rightarrow W\|$. We make another topological assumption:
\begin{equation} \label{eq:constant-section}
\mybox{
The constant lift of a loop in $W_{\geq 1/2}$ to a loop in $\bar{E}|W_{\geq 1/2}$ (constant being understood with respect to our trivialization) must be nullhomologous in $\bar{E}$. Note that if $x$ is such a lift, then $x^*T\bar{E}$ has a preferred trivialization (coming from the triviality of $TW$ and the constancy in fibre direction). Then, for a bounding surface $v$ as in Convention \ref{th:nullhomologous}, we require that the relative degree (integral of the relative first Chern class) of $v^*T\bar{E}$ should be $-1$ times the winding number of the loop in $W_{\geq 1/2}$.
}
\end{equation}
We consider compatible almost complex structures on $\bar{E}$ which, outside a compact subset, are the product of the complex structure on $W$ and an almost complex structure on $\bar{M}$. Similarly, we consider Hamiltonians $H$ on $\bar{E}$ such that outside a compact subset,
\begin{equation} \label{eq:product-h}
H = H_\gamma + H_{\bar{M}}.
\end{equation}
Here, $\gamma \in \frakg$ (and we say that $H$ is of type $\gamma$), $H_{\bar{M}}$ is a function on the fibre, and we have used \eqref{eq:outside-trivialization}. 
\end{setup}

\begin{remark}
We would like to mention an equivalent formulation of \eqref{eq:outside-trivialization}. Let's introduce this terminology:
\begin{equation} \label{eq:locally-trivial}
\mybox{
$\bar{p}$ is symplectically locally trivial at $x \in \bar{E}$ if: $x$ is a regular point; $T\bar{E}^v_x = \mathit{ker}(D\bar{p}_x)$ is a symplectic subspace, hence has a symplectic orthogonal complement $T\bar{E}^h_x$ which maps isomorphically to $TW_{\bar{p}(x)}$; and finally, $(\omega_{\bar{E}} - \bar{p}^*\omega_W)\,|\,T\bar{E}^h_x = 0$.
}
\end{equation}
Then, \eqref{eq:outside-trivialization} is equivalent to saying that, first, $\bar{p}$ is locally trivial at any point of $\bar{E}|W_{\geq 1/2}$ (which makes it a trivial symplectic fibration over any simply-connected open subset of $W_{\geq 1/2}$); and secondly, that the monodromy around the ``hole'' in $W_{\geq 1/2}$ is the identity.
\end{remark}

The time-dependent Hamiltonians that enter the definition of Floer cohomology are as follows.
\begin{equation} \label{eq:time}
\mybox{
Given a rotation number $r$, choose a loop $(\gamma_t)$ in $\frakg$ such that: the monodromy has rotation number $r$, and is hyperbolic if $r \in \bZ$ (in the remaining cases, it is automatically elliptic). Next, choose $(H_t)$ so that each $H_t$ is of type $\gamma_t$, and which has nondegenerate one-periodic orbits.
}
\end{equation}
Note that a hyperbolic element of $G$ acts fixed-point freely on $W$, and an elliptic element of $G$ has a unique fixed point. Hence, the one-periodic orbits of $(H_t)$ must be contained in a compact subset of $\bar{E}$. Take a codimension two submanifold (not necessarily symplectic) $\Omega \subset \bar{E}$ which represents $[\omega_{\bar{E}}]$, and which is disjoint from the one-periodic orbits of $(H_t)$. Additionally, make a generic choice of almost complex structure $(J_t)$, obviously within the class from Setup \ref{th:fibration}. One then defines the Floer complex as usual:
\begin{equation} \label{eq:floer-complex}
\mathit{CF}^*_{q,q^{-1}}(H_t,J_t) = \bigoplus_x \bQ((q)) x\,.
\end{equation}
The degree of each $x$ is its Conley-Zehnder index $i(x) \in \bZ$, and we have also fixed a trivialization of its orientation space (the determinant line of a suitable Cauchy-Riemann operator on a thimble). The differential $d_{q,q^{-1}}$ counts solutions $u$ of Floer's equation which are isolated up to translation, and with limits $x_{\pm}$, like this:
\begin{equation} \label{eq:floer-differential}
\text{$(x_-)$-coefficient of } d_{q,q^{-1}}x_+ = \sum_u \pm q^{u \cdot \Omega} \, .
\end{equation}
Here, the sign involves the previously considered trivialization of the orientation spaces.

\begin{remark} \label{th:unique-hf}
As part of the construction, one shows that the Floer complex is independent of $\Omega$ up to canonical isomorphism. Given two choices $\Omega$ and $\tilde{\Omega}$, and $(x,v)$ as in Convention \ref{th:nullhomologous}, the intersection number $v \cdot (\tilde{\Omega}-\Omega)$ is independent of $v$. Multilying each generator $x$ by $q^{v \cdot (\tilde{\Omega} -\Omega)}$ relates the differentials defined using those two submanifolds.
\end{remark}

The cohomology $\mathit{HF}^*_{q,q^{-1}}(H_t,J_t)$ of the Floer complex only depends on $r$, up to canonical isomorphism, and that justifies our notation $\mathit{HF}^*_{q,q^{-1}}(\bar{p},r)$. We refer to \cite{seidel17} for details. The main tool in this and other arguments are continuation maps, whose geometric basics we will now recall quickly. Continuation map equations are of the form
\begin{equation} \label{eq:continuation-map-equation}
(\partial_s u - W_{s,t}) + J_{s,t}(u) (\partial_t u - X_{s,t}) = 0.
\end{equation}
Here, $(J_{s,t})$ is a family of almost complex structures, and $(W_{s,t}, X_{s,t})$ families of Hamiltonian vector fields associated to functions $(F_{s,t}, H_{s,t})$. They must satisfy
\begin{equation} \label{eq:plusminus-limit}
(F_{s,t},H_{s,t},J_{s,t}) \longrightarrow (0, H_{\pm,t},J_{\pm,t}) \;\; \text{as $\pm s \rightarrow \infty$, exponentially fast.}
\end{equation}

\begin{remark}
Classically, one often considers only the special case $F_{s,t} = 0$; there is no substantial loss involved in imposing that restriction, since any equation \eqref{eq:continuation-map-equation} can be reduced to that case. Still, the general framework fits more naturally into a discussion of $G$-connections \cite{seidel17}, which is suitable for addressing issues such as the uniqueness (well-definedness) of continuation and PSS maps. Hence, we have adopted it in our exposition, even though the examples arising in concrete computations will still have $F_{s,t} = 0$.
\end{remark}

The analysis of \eqref{eq:continuation-map-equation} runs parallel to Floer's equation, except for slight complications (whose study goes back at least to \cite{oh97b}) involving energy. The topological version of the energy is
\begin{equation} \label{eq:topological-energy}
\begin{aligned} 
E^{\mathit{top}}(u) & = \int_{\bR \times S^1} u^*\omega_{\bar{E}} - d(u^*F_{s,t} \mathit{ds} + u^*H_{s,t} \mathit{dt}) \\
& = \int_{\bR \times S^1} \omega_{\bar{E}}(\partial_s u - W_{s,t}, \partial_t u - X_{s,t}) \mathit{ds}\, \mathit{dt}
+ \int_{\bR \times S^1} u^*R_{s,t} \,\mathit{ds}\, \mathit{dt},
\end{aligned}
\end{equation}
where 
\begin{equation} \label{eq:curvature}
R_{s,t} = -\partial_s H_{s,t} + \partial_t F_{s,t} - \{F_{s,t}, H_{s,t}\}.
\end{equation}
In the last line of \eqref{eq:topological-energy}, the first term is the actual geometric energy for a solution of \eqref{eq:continuation-map-equation}, and so the important point is the sign of the curvature term \eqref{eq:curvature}. In our situation, the continuation map will be well-defined if:
\begin{equation} \label{eq:nonnegative-curvature}
\mybox{$R_{s,t}$ is of type $\gamma_{s,t}$, with $\gamma_{s,t} \in \frakg_{\geq 0}$ for all $(s,t)$.}
\end{equation}

\subsection{Computational aspects\label{subsec:compute-hf}}
We now turn to the proofs of Lemma \ref{th:les-2} and \ref{th:les-1}, which depend on specifically chosen Hamiltonians. Fix $r_2 \in \bZ$. The ingredients are: 
\begin{align}
\label{eq:fibre-hamiltonian}
&
\mybox{
On $\bar{M}$, take a time-dependent Hamiltonian $(H_{\bar{M},t})$, and a corresponding family $(J_{\bar{M},t})$, with the transversality properties necessary to define Floer cohomology. 
}
\\ &
\label{eq:only-constants}
\mybox{
Take $H_2$, $H_3$ provided by Lemmas \ref{th:morse-function} and \ref{th:morse-function-2}, and a small $\epsilon>0$ such that: every $\epsilon$-periodic orbit of $H_2$ or $H_3$ is constant, and the linearized Hamiltonian flow at such points is nondegenerate for all times in $(0,\epsilon]$.
}
\\ &
\label{eq:1-hamiltonian}
\mybox{
Set $r_1 = r_2 - \epsilon$. Let $(H_{r_1,t})$ be a time-dependent Hamiltonian on $\bar{E}$ such that, on $W_{\geq 1/2} \times \bar{M}$, $H_{r_1,t} = r_1 H^{\mathit{rot}} + H_{\bar{M},t}$. We also assume that all one-periodic orbits of $H_{r_1,t}$ are nondegenerate. Correspondingly, we choose $(J_t)$ whose restriction to $W_{\geq 1/2} \times \bar{M}$ is $i \times J_{\bar{M},t}$, and which can be used to define $\mathit{HF}^*_{q,q^{-1}}(H_{r_1,t},J_t)$.
}
\end{align}
All one-periodic orbits of $(H_{r_1,t})$ are disjoint from $W_{\geq 1/2} \times \bar{M}$; and all Floer trajectories are contained in $\bar{E}|W_{\leq 1/2}$. For the orbits, that follows from the fact that the Hamiltonian vector field projects to $2\pi i r_1 w\partial_w$ on $W_{\geq 1/2}$ (and $r_1$ is not an integer); for Floer trajectories $u$, one uses Lemma \ref{th:barrier}(i) applied to the projection $\bar{p} \circ u$, with $\rho$ slightly larger than $1/2$, and then takes the limit $\rho \rightarrow 1/2$ (a maximum principle argument also works, if you prefer). These observations also explain why, in spite of the strong restrictions we have placed on the behaviour of $(H_{r_1,t})$ and $(J_t)$ on $W_{\geq 1/2} \times \bar{M}$, nondegeneracy of one-periodic orbits, and regularity of Floer trajectories, can be easily achieved. 

We now define another Floer cohomology group $\mathit{HF}^*_{q,q^{-1}}(H_{r_2,t},J_t)$. The almost complex structures are the same as before, and the Hamiltonian is
\begin{equation} \label{eq:r2-hamiltonian}
H_{r_2,t} = \begin{cases} H_{r_1,t} & \text{on } \bar{E}|W_{\leq 1/2}, \\
 \epsilon (H_2 \circ \phi^{\mathit{rot}}_{-r_2t}) + r_2 H^{\mathit{rot}} + H_{\bar{M},t} & \text{on } W_{\geq 1/2} \times \bar{M}.
\end{cases}
\end{equation}
Here, $(\phi^{\mathit{rot}}_t)$ is the one-periodic rotational circle action on $W$, generated by $H^{\mathit{rot}}$. The two parts of \eqref{eq:r2-hamiltonian} match up since
\begin{equation}
\epsilon (H_2 \circ \phi^{\mathit{rot}}_{-r_2t}) + r_2 H^{\mathit{rot}} = -\epsilon H^{\mathit{rot}} + r_2 H^{\mathit{rot}} = r_1 H^{\mathit{rot}} 
\text{ on $W_{\leq \rho_1}$.}
\end{equation}
By looking at the behaviour on $\bar{E}|W_{\geq \rho_2}$, one sees that $H_{r_2,t}$ is of type
\begin{equation} \label{eq:2-type}
\exp(r_2t\gamma^{\mathit{rot}}) (\epsilon\gamma^{\mathit{hyp}} + r_2\gamma^{\mathit{rot}}) \exp(-r_2t\gamma^{\mathit{rot}}) \in \frakg.
\end{equation}
The loop \eqref{eq:2-type} in $\frakg$ generates the path $\exp(r_2 t \gamma^{\mathit{rot}}) \exp(\epsilon t \gamma^{\mathit{hyp}})$ in $G$. The endpoint of that path is $\exp(\epsilon\gamma^{\mathit{hyp}})$, but its canonical lift to $\tilde{G}$ has rotation number $r_2$. Hence, the Hamiltonian \eqref{eq:r2-hamiltonian} has the behaviour at infinity required in the definition of $\mathit{HF}^*_{q,q^{-1}}(\bar{p},r_2)$. 

\begin{lemma} \label{th:h1-h2}
The one-periodic orbits of $(H_{r_2,t})$ are of two kinds. 
\begin{align} \label{eq:first-orbits}
&
\mybox{
The first kind are disjoint from $W_{\geq 1/2} \times \bar{M}$, and agree with the orbits of $(H_{r_1,t})$. They form a subcomplex of $\mathit{CF}^*_{q,q^{-1}}(H_{r_2,t},J_{t})$ which can be identified with $\mathit{CF}_{q,q^{-1}}^*(H_{r_1,t},J_{t})$.
} 
\\
& \label{eq:second-orbits}
\mybox{
The second kind of orbits are of the form $(\phi^{\mathit{rot}}_{r_2t}(w_2), x(t))$, where $w_2$ is the critical point from Lemma \ref{th:morse-function}, and $x$ is a one-periodic orbit of $(H_{\bar{M},t})$. They form a quotient complex, whose cohomology is $H^{*-1}(\bar{M};\bQ)((q))$.
}
\end{align}
\end{lemma}

\begin{proof}
Since the Hamiltonian isotopy generated by $(H_{r_2,t})$ preserves $\bar{E}|W_{\leq 1/2}$, any one-periodic orbit is either contained or disjoint from that subset. The orbits in $\bar{E}|W_{\leq 1/2}$ obviously agree with those of $(H_{r_1,t})$. Next, observe that we have the following relation between the Hamiltonian isotopies associated to $(H_{r_2,t})$ (on $\bar{E}$), $H_2$ (on $W$), and $H_{\bar{M},t}$ (on $\bar{M}$):
\begin{equation} \label{eq:twisted-flow}
\phi_{H_{r_2},t}(w,x) = (\phi^{\mathit{rot}}_{r_2t}(\phi_{H_2,\epsilon t}(w)), \phi_{H_{\bar{M},t}}(x)) \;\; \text{ for $(w,x) \in W_{\geq 1/2} \times \bar{M}$;}
\end{equation}
in view of \eqref{eq:only-constants}, this implies that the one-periodic orbits lying in
$W_{\geq 1/2} \times \bar{M}$ are precisely those described in \eqref{eq:second-orbits}.
A topological clarification is necessary: as usual, we are considering only orbits which are nullhomologous. For an orbit as in \eqref{eq:second-orbits}, we have a (free) homotopy 
\begin{equation} \label{eq:decompose-loop}
(\phi^{\mathit{rot}}_{r_2t}(w_2),x(t)) \htp (\phi_{r_2t}^{\mathit{rot}}(w_2),x(0)) \text{ composed with } (w_2,x(t)).
\end{equation}
On the right hand side, the first loop is nullhomologous in $\bar{E}$ by \eqref{eq:fibre-hamiltonian}. If the second loop is nullhomologous in $\bar{M}$, the left hand side is nullhomologous in $\bar{E}$. On the other hand, there could be $x(t)$ which are not nullhomologous in $\bar{M}$, but such that $(\phi^{\mathit{rot}}_{rt}(w_2),x(t))$ is nullhomologous in $\bar{E}$; these would contribute to the Floer complex of $(H_{r_2,t})$, but turn out to be ultimately irrelevant, since orbits which are not nullhomologous (or nullhomotopic) give trivial summands of the Floer cohomology of $(H_{\bar{M},t})$.

Lemma \ref{th:barrier}(i), again applied to the projections of Floer trajectories to $W$, shows that a trajectory whose limits are both in the class \eqref{eq:first-orbits} remains inside $\bar{p}^{-1}(W_{\leq 1/2})$. The orbits of type \eqref{eq:second-orbits} project to loops with winding number $r_2$ in $W_{\geq 1/2}$, while the projection of the vector field of $(H_{r_2,t})$ rotates with speed $r_1<r_2$ in $\{1/2 \leq |w| \leq \rho_1\}$. Hence, Lemma \ref{th:barrier}(iii) shows that there can't be Floer trajectories whose limit $x_+$ is of type \eqref{eq:first-orbits}, and $x_-$ of type \eqref{eq:second-orbits}. This completes the verification of the properties stated in \eqref{eq:first-orbits}.

Let's look at orbits of type \eqref{eq:second-orbits}. Lemma \ref{th:barrier}(ii) shows that a Floer trajectory connecting two such orbits must remain inside $\bar{p}^{-1}(W_{\geq 1/2})$. One can therefore decouple the Floer equation into its $W$ and $\bar{M}$ components. The $W$-component has the same limits at both ends, and since the symplectic form on $W$ is exact, it must necessarily be stationary at the one-periodic orbits. The $\bar{M}$ component just satisfies the Floer equation for $(H_{\bar{M},t}, J_{\bar{M},t})$, and from that, one easily shows that the quotient complex computes the Floer cohomology of $(H_{\bar{M},t})$, meaning the ordinary cohomology of $\bar{M}$. The only piece of \eqref{eq:second-orbits} that needs a bit of extra attention concerns the grading. From a choice of bounding surface in $\bar{E}$, we get a preferred homotopy class of trivializations of the $T\bar{E}$ along our orbit, meaning the vector bundle
\begin{equation}
T\bar{E}_{(\phi^{\mathit{rot}}_{r_2t}(w_2),x(t))} = TW_{\phi^{\mathit{rot}}_{r_2t}(w_2)} \oplus T\bar{M}_{x(t)}
\end{equation}
over the circle. The Conley-Zehnder index is defined by using that trivialization to write the linearization of our Hamiltonian isotopy as a path in the linear symplectic group. Assume that $x$ is nullhomologous already in $\bar{M}$ (which is unproblematic since, as mentioned before, the other homology classes of orbits contribute trivially to Floer cohomology); choose a bounding surface in $\bar{M}$, and the resulting trivialization of $x^*T\bar{M}$. For $TW_{\phi^{\mathit{rot}}_{r_2t}(w_2)} = \bC$, choose a trivialization which rotates $r_2$ times as $t$ goes around the circle. This agrees with the trivialization coming from a bounding surface in $T\bar{E}$: to check that, one goes through the homotopy \eqref{eq:decompose-loop}, and then uses assumption \eqref{eq:fibre-hamiltonian} for the first loop on the right hand side. The consequence is that the Conley-Zehnder of our periodic orbit is $2r_2$ higher than it would be if we considered the orbit naively as lying in the product $W \times \bar{M}$. Now, the last-mentioned naive index is $1-2r_2$ ($W$ component) plus the Conley-Zehnder index in $\bar{M}$. Altogether, this yields the shift in the grading which appears in \eqref{eq:second-orbits}.
\end{proof}

Lemma \ref{th:h1-h2} provides an exact sequence with the same terms as \eqref{eq:hf-exact-2}, but where the map $\mathit{HF}^*_{q,q^{-1}}(\bar{p},r_1) \rightarrow \mathit{HF}^*_{q,q^{-1}}(\bar{p},r_2)$ is induced by the inclusion of the subcomplex from \eqref{eq:first-orbits}. To complete the proof of Lemma \ref{th:les-2}, we have to show that this coincides with the continuation map:

\begin{lemma} \label{th:continuation-inclusion}
For suitable choices made in its construction, the chain level continuation map $C_{q,q^{-1}}^{r_1,r_2}$ realizes the inclusion of complexes from Lemma \ref{th:h1-h2}.
\end{lemma}

\begin{proof}
When defining \eqref{eq:continuation-map-equation}, let's use the Hamiltonians
\begin{equation} \label{eq:chi-hamiltonian}
\begin{aligned}
& H_{r_1,r_2,s,t} = H_{r_1,t} + \chi(s) (H_{r_2,t}-H_{r_1,t}), \\
& \chi(s) = 1 \text{ for $s \ll 0$, } \chi(s) = 0 \text{ for $s \gg 0$, } \chi'(s) \leq 0 \text{ everywhere,}
\end{aligned}
\end{equation}
(with trivial $F_{s,t}$ term), and the same $(J_t)$ as before. The functions \eqref{eq:chi-hamiltonian} are of class 
\begin{equation}
\gamma_{r_1,r_2,s,t} = -\gamma^{\mathit{rot}} + \epsilon\chi(s) \exp(r_2t\gamma^{\mathit{rot}}) (\gamma^{\mathit{hyp}} + \gamma^{\mathit{rot}}) \exp(-r_2t\gamma^{\mathit{rot}}),
\end{equation}
Since $\gamma^{\mathit{hyp}}+\gamma^{\mathit{rot}} \in \frakg_{\geq 0}$ by \eqref{eq:h-functions}, the curvature condition \eqref{eq:nonnegative-curvature} is satified. 

Over the annulus $\{1/2 \leq |w| \leq \rho_1\}$ we have $H_{r_1,r_2,s,t} = H_{r_1,t} = H_{r_2,t}$, hence Lemma \ref{th:barrier} applies to projections to $W$ of solutions of the continuation map equation, with the following outcome. Since all solutions necessarily have a limit $x_+$ lying in $\bar{p}^{-1}(W_{\leq 1/2})$, they must be entirely contained in that subset (parts (i) and (iii) of the Lemma). Hence, they are really solutions of Floer's equation for $H_{r_1,t}$, and the only isolated ones are stationary at one-periodic orbits. 
\end{proof}

Next we tackle Lemma \ref{th:les-1}, which is similar if a little more complicated. Set $r_3 = r_2 + \epsilon$, and
\begin{equation} \label{eq:r3-hamiltonian}
H_{r_3,t} = \begin{cases} H_{r_1,t} & \text{on $\bar{E}|W_{\leq 1/2}$,} \\
\epsilon(H_3 \circ \phi_{-r_2t}^{\mathit{rot}}) + r_2 H^{\mathit{rot}} + H_{\bar{M},t} & \text{on $W_{\geq 1/2} \times \bar{M}$.}
\end{cases}
\end{equation}
This agrees with $H_{r_2,t}$ on $\bar{E}|W_{\leq \rho_2}$. It is of type $r_3\gamma^{\mathit{rot}}$, hence has monodromy $\exp(\epsilon \gamma^{\mathit{rot}})$ and rotation number $r_3$. The analogue of Lemma \ref{th:h1-h2} reads as follows:

\begin{lemma} \label{th:h2-h3}
The one-periodic orbits of $(H_{r_3,t})$ are of three kinds. 
\begin{align} \label{eq:first-orbits-2}
&
\mybox{
The first kind are disjoint from $W_{\geq 1/2} \times \bar{M}$, and agree with the orbits of $(H_{r_1,t})$. They form a Floer subcomplex, which can be identified with $\mathit{CF}^*_{q,q^{-1}}(H_{r_1,t},J_t)$.
} 
\\
& \label{eq:second-orbits-2}
\mybox{
The second kind of orbits are as in \eqref{eq:second-orbits}. Together with \eqref{eq:first-orbits-2}, they form a larger subcomplex, which can be identified with $\mathit{CF}_{q,q^{-1}}^*(H_{r_2,t},J_t)$.
}
 \\
& \label{eq:third-orbits}
\mybox{
Finally, we have orbits $(\phi^{\mathit{rot}}_{r_2t}(w_3),x(t))$, which are as in \eqref{eq:second-orbits-2} but instead use the critical point $w_3$ of $H_3$ from Lemma \ref{th:morse-function-2}. The resulting quotient complex has cohomology $H^*(\bar{M};\bQ)((q))$. (The grading shift between this and the corresponding statement in \eqref{eq:second-orbits} comes from the difference in Morse indices of the critical points.)
}
\end{align}
\end{lemma}

\begin{proof}
Most of the argument repeats the previous material, and we'll omit it. The only new issue concerns Floer trajectories both of whose limits $x_{\pm}$ lie in $W_{\geq 1/2} \times \bar{M}$. Lemma \ref{th:barrier}(ii) (the winding numbers being always equal to $r_2$) implies that each such trajectory remains in $W_{\geq 1/2} \times \bar{M}$, hence projects to a Floer trajectory in $W_{\geq 1/2}$. The actions of the one-periodic orbits in $W$ satisfy
\begin{equation} \label{eq:difference-of-actions}
A(t \mapsto \phi^{\mathit{rot}}_{r_2t}(w_2)) - A(t \mapsto \phi^{\mathit{rot}}_{r_2t}(w_3)) = H_3(w_2) - H_3(w_3).
\end{equation}
Consider the Morse complex of $H_3$, but restricted to $W_{\geq 1/2}$. This computes the cohomology of $W_{\geq 1/2}$ (because $\nabla H_3$ points outwards along $\partial W_{\geq 1/2}$, and the function grows to $+\infty$ as $|w| \rightarrow 1$), and the cap product action of cycles on that Morse complex computes the ring structure. From the latter observation, one can see that there must be gradient trajectories flowing upwards from $w_3$ to $w_2$, so \eqref{eq:difference-of-actions} must be positive (or course, one could get the same result by close inspection of the definition of those functions). Hence, a Floer trajectory in $W_{\geq 1/2}$ is either stationary at one of the two orbits $(\phi^{\mathit{rot}}_{r_2t}(w_2))$, $(\phi^{\mathit{rot}}_{r_2t}(w_3))$; or otherwise, it must have $(\phi^{\mathit{rot}}_{r_2t}(w_2))$ as its negative limit, and $(\phi^{\mathit{rot}}_{r_2t}(w_3))$ as its positive limit. This implies the subcomplex property from \eqref{eq:second-orbits-2} and, by the same argument as for \eqref{eq:second-orbits}, the comohology computation in \eqref{eq:third-orbits}.
\end{proof}

As before, to complete the proof of Lemma \ref{th:les-1}, we have to show:

\begin{lemma} \label{th:continuation-inclusion-2}
For suitable choices, the continuation map $C_{q,q^{-1}}^{r_2,r_3}$ agrees with the inclusion of the subcomplex from \eqref{eq:second-orbits-2}.
\end{lemma}

\begin{proof}
This is parallel to Lemma \ref{th:continuation-inclusion}, again with an added wrinkle. From \eqref{eq:everywhere} and the definitions \eqref{eq:r2-hamiltonian}, \eqref{eq:r3-hamiltonian}, we see that $H_{r_2,t} \leq H_{r_3,t}$ everywhere. Hence, if we interpolate between these functions, as in \eqref{eq:chi-hamiltonian}, then that produce a family of functions which is everywhere non-increasing in the parameter $s$. This allows one to restrict the behaviour of solutions of the continuation map equation by action arguments, exactly as in the proof of Lemma \ref{th:h2-h3}.
\end{proof}

\section{Relative Floer cohomology\label{sec:relative-floer}}

This section concerns Hamiltonian Floer cohomology in a relative sense, meaning for symplectic manifolds with a fixed symplectic divisor (a codimension two symplectic submanifold) which is preserved by the Hamiltonian vector fields under consideration. The issues one encounters have to do with one-periodic orbits, and Floer trajectories, contained in the divisor. These would be serious problems in general; for us, they will be made tractable by additional constraints imposed on the symplectic manifold and (maybe more surprisingly) on the Hamiltonians.

\subsection{Floer trajectories}
Take a closed symplectic manifold. For consistency with the notation elsewhere in the paper, we write this as $\bar{M}$, and its dimension as $2n-2$. We assume that $c_1(\bar{M}) = 0$, and that $[\omega_{\bar{M}}]$ is integral. Let $\delta M \subset \bar{M}$ be a symplectic divisor which represents $[\omega_{\bar{M}}]$; we write $M = \bar{M} \setminus \delta M$. 

\begin{setup} \label{th:tangent-hamiltonians}
Take the normal bundle $N = N(\delta M) \rightarrow \delta M$, as defined by the symplectic orthogonal splitting
\begin{equation} \label{eq:tangent-splitting}
T\bar{M}\,|\,\delta M = T(\delta M) \operp N.
\end{equation}
Fix a complex structure $J_N$ on $N$ compatible with its orientation (or equivalently, with the symplectic structure).  For almost complex structures $J$ on $\bar{M}$ compatible with $\omega_{\bar{M}}$, we require that $\delta M$ is an almost complex submanifold, and that $J|N = J_N$.
For a Hamiltonian $H \in \smooth(\bar{M},\bR)$, we require its vector field $X$ to be tangent to $\delta M$. Then, $L_X$ necessarily preserves the splitting \eqref{eq:tangent-splitting}. We additionally require that the $N$ part of that operation should preserve $J_N$, meaning $L_X(J_NY) = J_N(L_XY)$ for $Y \in \smooth(\delta M,N)$.
\end{setup}

Near a point of $\delta M$, the situation is isomorphic (by a form of the symplectic tubular neighbourhood theorem) to a neighbourhood of $(0,0)$ in the following local model:
\begin{equation}
\left\{
\begin{aligned}
& \bar{M}^{\mathit{local}} = \bR^{2n-4} \times \bC, \;\; \text{with its standard symplectic form}, \\ 
&\delta M^{\mathit{local}} = \bR^{2n-4} \times \{0\}, \\
& J_N^{\mathit{local}} = i.
\end{aligned}
\right.
\end{equation}
In such local coordinates $(y_1,y_2) \in \bR^{2n-4} \times \bC$, the Hamiltonians and almost complex structures from Setup \ref{th:tangent-hamiltonians} have the form
\begin{equation} \label{eq:local-hamiltonian}
\left\{
\begin{aligned} &
H^{\mathit{local}}(y_1,y_2) = f(y_1) + \half g(y_1) |y_2|^2 + O(|y_2|^3), \\
& J^{\mathit{local}}_{(y_1,y_2)} = \begin{pmatrix} \ast & O(|y_2|) \\ O(|y_2|) & i + O(|y_2|) \end{pmatrix}.
\end{aligned}
\right.
\end{equation}

Suppose that we have a time-dependent Hamiltonian $(H_t)_{t \in S^1 = \bR/\bZ}$ as in Setup \ref{th:tangent-hamiltonians}, its vector field $(X_t)$, and the resulting Hamiltonian isotopy $\phi = (\phi_t)_{t \in \bR}$. Let $x = x(t)$ be a one-periodic orbit, contained in $\delta M$. The linearization of the flow in normal direction gives maps
\begin{equation} \label{eq:normal-derivative}
D\phi_t|N_{x(0)}: N_{x(0)} \longrightarrow N_{x(t)},
\end{equation}
which are unitary with respect to $J_N$ and the symplectic form. In particular, the monodromy is a rotation of $N_{x(0)}$,
\begin{equation} \label{eq:alpha-n}
D\phi_1|N_{x(0)} = \exp( \alpha_N(x)J_N), \quad \alpha_N(x) \in [0,2\pi).
\end{equation}
This gives rise to a preferred class of nowhere zero sections of $x^*N$, namely 
\begin{equation} \label{eq:preferred-trivialization}
\nu_x(t) = \exp(-t\alpha_N(x)J_N) D\phi_t(\text{some nonzero vector in $N_{x(0)}$}).
\end{equation}
One can think of \eqref{eq:normal-derivative} as parallel transport maps for a unitary connection $\nabla_N$ on $x^*N$. Associated to that connection is the selfadjoint operator $Q_N(x) = J_N \nabla_N$, whose spectrum is $\alpha_N(x) + 2\pi \bZ$. The sections \eqref{eq:preferred-trivialization} form the $\alpha_N$-eigenspace.

Our first relevant notion is a kind of action, for one-periodic orbits $x$ contained in $\delta M$ (as usual, assumed to be nullhomologous).
\begin{equation} \label{eq:normal-action}
\mybox{
Write $x$ as the boundary value of $v: S \rightarrow \bar{M}$, as in Convention \ref{th:nullhomologous}. More precisely, we require that $v$ should intersect $\delta M$ transversally along $\partial S$, and that the resulting nowhere zero section of $x^*N$ should be in the same homotopy class as \eqref{eq:preferred-trivialization}. Then, we define
\[
A_N(x) =  - \int_S v^*\omega_{\bar{M}} + \int_{S^1} H_t(x(t)) \, \mathit{dt} +
(v|(S \setminus \partial S) \cdot \delta M).
\]
}
\end{equation}
The intersection number in the last term makes sense because the points of $v^{-1}(\delta M)$ cannot accumulate along $\partial S$, by the transversality condition on $v$. Because $\delta M$ represents $[\omega_{\bar{M}}]$, \eqref{eq:normal-action} is independent of the choice of bounding surface. 

\begin{example} \label{th:constant-orbit}
Suppose that $x(t) = x$ is a constant one-periodic orbit at some point of $\delta M$. As we go from $t = 0$ to $t = 1$, $D\phi_t|N_x$ rotates by $\alpha_N(x) + 2\pi w_N(x)$, for some $w_N(x) \in \bZ$. One can choose $v$ to be contained in a small neighbourhood of $x$, and a local computation leads to the formula
\begin{equation}
A_N(x) = \int_{S^1} H_t(x) \, \mathit{dt} + w_N(x).
\end{equation}
\end{example}
Let $u: \bR \times S^1 \longrightarrow \bar{M}$ be a map such that, for $s \rightarrow \pm \infty$, $u(s,t)$ is asymptotic to orbits $x_{\pm}(t)$. Recall the classical definition of energy,
\begin{equation} \label{eq:e}
E(u) = \int_{\bR \times S^1} u^*\omega_{\bar{M}} - dH_t(\partial_ s u) \mathit{ds}\, \mathit{dt} = \int_{\bR \times S^1} \omega_{\bar{M}}(\partial_s u, \partial_t u - X_t) \mathit{ds}\, \mathit{dt}.
\end{equation}
Suppose now that $u$ lies in $\delta M$. 
\begin{equation} \label{eq:relative-energy}
\mybox{
Choose a section $\nu_u$ of $u^*N$ which, as $s \rightarrow \pm\infty$, is asymptotic to nonzero sections of $x_{\pm}^*N$ in the homotopy classes \eqref{eq:preferred-trivialization}. Then 
\[
E(u) = A_N(x_-) - A_N(x_+) + (\nu_u \cdot \text{zero-section}),
\]
where the last term is the algebraic count of zeros of $\nu_u$. 
}
\end{equation}

From this point onwards, we will assume that all one-periodic orbits are nondegenerate. For orbits in $\delta M$, this implies that $\alpha_N(x) \in (0,2\pi)$. To any map $u$ as before, one can associate a class of Cauchy-Riemann operators $D_u$ on $u^*TM$, which become Fredholm in suitable Sobolev completions (for simplicity, let's say from $W^{1,2}$ to $L^2$), and whose index we denote by $I(u)$. If $u$ is contained in $\delta M$, the operator $D_u$ can be chosen so that it fits into a short exact sequence
\begin{equation} \label{eq:operator-les}
\xymatrix{
0 \ar[r] & 
W^{1,2}(u^*T(\delta M)) \ar[d]_-{D_{\delta M,u}} \ar[r] & 
\ar[d]_-{D_u} W^{1,2}(u^*T\bar{M}) \ar[r] & 
\ar[d]_-{D_{N,u}} W^{1,2}(u^*N) \ar[r] & 0
\\
0 \ar[r] & 
L^2(u^*T(\delta M)) \ar[r] & 
L^2(u^*T\bar{M}) \ar[r] & 
L^2(u^*N) \ar[r] & 0.
}
\end{equation}
The indices $I_{\delta M}(u) = \mathrm{index}(D_{\delta M,u})$ and $I_N(u) = \mathrm{index}(D_{N,u})$ therefore satisfy 
\begin{equation}
I(u) = I_{\delta M}(u) + I_N(u).
\end{equation}
Take the operator $D_{N,u}$ which appears in \eqref{eq:operator-les}. Asymptotically as $s \rightarrow \pm\infty$, that operator becomes equal to $\partial_s + Q_N(x_{\pm})$, where the $Q_N$ are the previously mentioned selfadjoint operators. The index can therefore be computed as
\begin{equation} \label{eq:normal-index}
\half I_N(u) = \nu_u \cdot \text{zero-section},
\end{equation}
where the right hand side is as in \eqref{eq:relative-energy}. Combining the two equations yields
\begin{equation} \label{eq:action-energy-index}
E(u) = A_N(x_-) - A_N(x_+) + \half I_N(u).
\end{equation}

\begin{lemma} \label{th:automatic-regularity}
If $I_N(u) \geq 0$, then $D_{N,u}$ is surjective.
\end{lemma}

\begin{proof}[Sketch of proof]
This is a form of automatic regularity. Let $\xi$ be a nontrivial solution of $D_{N,u}(\xi) = 0$, which decays as $s \rightarrow \pm\infty$. More precisely, the decay will always be exponential, with rates $-\alpha_N(x_{\pm}) \mp 2\pi m_{\pm}(\xi)$ for some integers 
\begin{equation}
m_-(\xi) \geq 1, \;\; m_+(\xi) \geq 0.
\end{equation}
The set $\xi^{-1}(0)$ is finite. Each point in it has a multiplicity $m_z(\xi) > 0$. By comparing the behaviour of $\xi$ with that of a section $\nu_u$ as in \eqref{eq:relative-energy}, one sees that
\begin{equation} \label{eq:linearized-m}
0 < m(\xi) \stackrel{\text{def}}{=} m_+(\xi) + m_-(\xi) + \sum_z m_z(\xi) = \nu_u \cdot \text{zero-section}.
\end{equation}
By \eqref{eq:normal-index} it follows that if $I_N(u) \leq 0$, then $D_{N,u}$ is injective. The statement we have made is obtained by applying that same idea to the dual operator.
\end{proof}

\begin{definition} \label{th:admissible}
We say that $(H_t)$ is Floer-admissible if
\begin{equation} \label{eq:admissible}
|A_N(x_-) - A_N(x_+)|  < 1
\end{equation}
for all one-periodic orbits $x_{\pm}$ contained in $\delta M$.
\end{definition}
%
%

Take some $(H_t)$ with nondegenerate one-periodic orbits, as well as a family $(J_t)$ of almost complex structures as in Setup \ref{th:tangent-hamiltonians}, and consider the resulting Floer equation. Any Floer trajectory $u$ has an associated linearized operator, which is a specific $D_u$ within the class described previously. If $u$ is contained in $\delta M$, one also has distinguished operators $D_{\delta M,u}$ and $D_{N,u}$ as in \eqref{eq:operator-les}. 

\begin{lemma} \label{th:using-admissibility}
Suppose that $(H_t)$ is Floer-admissible. Then, any Floer trajectory $u$ contained in $\delta M$ has surjective operator $D_{N,u}$.
\end{lemma}

\begin{proof}
We know that $E(u) \geq 0$. By \eqref{eq:action-energy-index} the index satisfies $I_N(u) > -2$, hence is actually nonnnegative. Lemma \ref{th:automatic-regularity} does the rest.
\end{proof}

Here is the appropriate version of the standard transversality result, following \cite{floer-hofer-salamon94, hofer-salamon95}:

\begin{lemma} \label{th:relative-transversality}
Suppose that $(H_t)$ is Floer-admissible, and choose a generic $(J_t)$ within the class from Setup \ref{th:tangent-hamiltonians}. Then all solutions of Floer's equation will be regular ($D_u$ is onto). Moreover, for solutions $u$ which are contained in $\delta M$ and have $I(u) \leq 2$, $D_{N,u}$ will be invertible. Finally, if $v: \bC P^1 \rightarrow \bar{M}$ is a non-constant $J_t$-holomorphic sphere, then $x(t) \notin v(\bC P^1)$ for any one-periodic orbit, and $u(s,t) \notin v(\bC P^1)$ for any solution $u$ of Floer's equation with $I(u) \leq 2$.
\end{lemma}

\begin{proof}
The case where $u$ is not contained in $\delta M$ is standard. For $u$ contained in $\delta M$, Lemma \ref{th:automatic-regularity} says that $D_{N,u}$ is always surjective. On the other hand, standard transversality inside $\delta M$ says that $D_{\delta M,u}$ is onto for generic $(J_t)$, and by \eqref{eq:operator-les} that implies the surjectivity of $D_u$. The kernel of $D_{\delta M,u}$ is at least one-dimensional, and that of $D_{N,u}$ has even dimension, so if $I(u) \leq 2$, then $D_{N,u}$ must be invertible. The only other point that needs to be mentioned concerns pseudo-holomorphic spheres contained in $\delta M$. For each such (nonconstant) sphere $v$, 
\begin{equation} \label{eq:v-degree}
\textstyle \int v^*c_1(\delta M) = -\int v^*c_1(N) = -\int v^*\omega_{\bar{M}} < 0. 
\end{equation}
Hence, those spheres fill out a subspace of $\delta M$ of dimension $\leq (2n-4)-6$, which is more than sufficient for our purposes.
\end{proof}
%

We need to impose another, more technical, condition on the almost complex structures and Hamiltonians. For that, let's return to the local model in \eqref{eq:local-hamiltonian}. Consider the subclass of Hamiltonians and almost complex structure of the following kind:
\begin{equation} \label{eq:locally-split}
\left\{
\begin{aligned}
& H^{\mathit{local}}(y_1,y_2) = H_1^{\mathit{local}}(y_1) + H_2^{\mathit{local}}(y_2), \\ & \qquad \text{ with } \nabla H_1 = 0 \text{ at $y_1 = 0$, and }
H_2^{\mathit{local}}(y_2) = {\textstyle\frac{a}{2}} |y_2|^2 + O(|y_2|^3) \text{ for some } a \in \bR, \\
& J^{\mathit{local}} = J_1^{\mathit{local}} \oplus J^{\mathit{local}}_2, \\ &
\qquad \text{ with $J_k^{\mathit{local}}$ depending only on $y_k$, and $J_2^{\mathit{local}} = i$ at $y_2 = 0$.}
\end{aligned}
\right.
\end{equation}
Suppose that we have $(H_t^{\mathit{local}}, J_t^{\mathit{local}})$ in this form. By definition, $x^{\mathit{local}}(t) = (0,0)$ is a one-periodic orbit. As usual, let's assume that this is nondegenerate. The Floer equation separates into two components (in $y_1$ and $y_2$ directions). One can apply the analysis from \cite{robbin-salamon02} to the second component, and get a precise description of how Floer trajectories approach $x^{\mathit{local}}$ in normal direction. 

\begin{definition} \label{th:local-splitting}
We say that $(H_t,J_t)$ is locally split if, near each one-periodic $x \subset \delta M$, there are $t$-dependent local coordinates centered at $x(t)$, in which $(\phi_t)$ reduces to the Hamiltonian isotopy induced by a $t$-dependent function as in \eqref{eq:locally-split}; and similarly $J_t$ corresponds to a $t$-dependent almost complex structure as in \eqref{eq:locally-split}.
\end{definition}

This allows us to apply the previously mentioned analysis to the way in which Floer trajectories converge to one-periodic orbits in $\delta M$. We omit the details, and only state the outcome.

\begin{lemma} \label{th:decay}
Suppose $(H_t,J_t)$ is locally split. Let $x$ be a one-periodic orbit contained in $\delta M$, and $u$ a Floer trajectory not contained in $\delta M$, such that $\lim_{s \rightarrow \infty} u(s,\cdot) = x$. Then the component of $u$ in normal direction to $\delta M$ decays as $\exp(-\alpha_N(x) - 2\pi m_+(u))$ for some integer $m_+(u) \geq 0$. Moreover, it approaches zero from the direction of an eigenvector of $Q_N(x)$ for the eigenvalue $\alpha_N(x) + 2\pi m_+(u)$.

In the corresponding situation for $\lim_{s \rightarrow -\infty}$, the decay is as $\exp(-\alpha_N(x) + 2\pi m_-(u))$ for some $m_-(u) \geq 1$, and the relevant eigenvector is that for the eigenvalue $\alpha_N(x) - 2\pi m_-(u)$.
\end{lemma}

Let's extend $m_\pm(u)$ to the case where the limit is outside $\delta M$, by setting it to zero. Let $u$ be a Floer trajectory not contained in $\delta M$, and with limits $x_\pm$ which may or may not lie in $\delta M$. One of the consequences of the previous two Lemmas is that $u^{-1}(\delta M)$ is compact; by holomorphic curve theory, it must therefore be finite. Any point $z \in u^{-1}(\delta M)$ has a positive local intersection multiplicity $m_z(u) > 0$. Following what we've done in the linearized case \eqref{eq:linearized-m}, write
\begin{equation} \label{eq:minimal-m}
m(u) = m_+(u) + m_-(u) + \sum_z m_z(u) \geq \begin{cases} 1 & \text{$x_-$ lies in $\delta M$,} \\
0 & \text{otherwise.} \end{cases} 
\end{equation}

At this point we make a further geometric choice, which is not absolutely essential but makes it easier to formulate the rest of the argument.

\begin{setup} \label{th:normal-section}
Given $(H_t)$, take a section $\beta_N$ of $N$ which is transverse to the zero-section, and whose zeros are disjoint from the one-periodic orbits $x$ contained in $\delta M$. Morever, fix a tubular neighbourhood of $\delta M$. Using that neighbourhood, the multiples of $\beta_N$ give rise to a family $(\delta M)_{\zeta}$ of symplectic divisors in $\bar{M}$, parametrized by sufficiently small $\zeta \in \bC$. For $\zeta = 0$ this is just $\delta M$, while for $\zeta \neq 0$ it intersects $\delta M$ transversally in $B_N = \beta_{N}^{-1}(0)$.
\end{setup}

With respect to that trivialization, the preferred nonzero sections \eqref{eq:preferred-trivialization} at any $x \subset \delta M$ have a winding number $w_N(x) \in \bZ$. We spell out the definition for the sake of clarity:
\begin{equation} \label{eq:winding-number}
\mybox{
define $a(t)$ by $\exp(-t\alpha_N(x) J_N) D\phi_t(\beta_{N,x(0)}) \in \bR^{>0} \cdot e^{2\pi i a(t)}\beta_{N,x(t)}$. Then $w_N(x) = a(1) - a(0)$.
}
\end{equation}
Again, we extend that to orbits lying outside $\delta M$, by setting it to zero. Many of our previous formulae can be restated in terms of those winding numbers. For instance, \eqref{eq:normal-action} becomes
\begin{equation} \label{eq:normal-action-1b}
A_N(x) =  - \int_S v^*\omega_{\bar{M}} + \int_{S^1}\, H_t(x(t)) \mathit{dt} + v \cdot (\delta M)_{\zeta} + w_N(x) \quad \text{for any $v: S \rightarrow \bar{M}$, $v|\partial S = x$.}
\end{equation}
Here, unlike \eqref{eq:normal-action}, there is no other condition on how $v$ approaches the boundary; and the intersection is with $(\delta M)_{\zeta}$ for any sufficiently small $\zeta \neq 0$. 
The index formula \eqref{eq:normal-index} can similarly be written as
\begin{equation} \label{eq:normal-index-1b}
\half I_N(u) = w_N(x_+) - w_N(x_-) + u \cdot (\delta M)_{\zeta}.
\end{equation}
Finally, if $u$ is a Floer trajectory not contained in $\delta M$, and the local splitting condition is satisfied, then Lemma \ref{th:decay} implies that
\begin{equation} \label{eq:m-topological}
u \cdot (\delta M)_{\zeta} = m(u) + w_N(x_-) - w_N(x_+).
\end{equation}

We are now ready to assemble the elements introduced above into the definition of relative Floer cohomology. Take $(H_t,J_t)$ as in Lemma \ref{th:relative-transversality}, with nondegenerate one-periodic orbits, and which additionally satisfy Definition \ref{th:local-splitting}. Choose $\beta_N$ as in Setup \ref{th:normal-section}. The associated relative Floer complex is
\begin{equation} \label{eq:relative-floer}
\mathit{CF}^*_q(H_t,J_t) = \bigoplus_x q^{w_N(x)}\bQ[[q]] x\, ,
\end{equation}
with the usual grading. The differential is defined as in \eqref{eq:floer-differential}, but now specifically using the divisors $(\delta M)_{\zeta}$, for sufficiently small $\zeta \neq 0$.
\begin{equation} \label{eq:top-power}
\text{$(x_-)$-coefficient of } d_q x_+ = \sum_u \pm q^{u \cdot (\delta M)_{\zeta}} \, .
\end{equation}
The cohomology of \eqref{eq:relative-floer} is denoted by $\mathit{HF}^*_q(H_t,J_t)$. The advantage of writing the powers of $q$ as in \eqref{eq:top-power} is that after inverting $q$, one obviously gets the standard definition of Floer cohomology over $\bQ((q))$. The flipside is that it's not obvious that \eqref{eq:top-power} actually defines an endomorphism of \eqref{eq:relative-floer}. To see that, one rewrites the intersection number as
\begin{equation} \label{eq:express-intersection}
u \cdot (\delta M)_{\zeta} + w_N(x_+) = w_N(x_-) + \begin{cases} m(u) & \text{if $u$ is not contained in $\delta M$,} \\
\half I_N(u) & \text{if $u$ is contained in $\delta M$.}
\end{cases}
\end{equation}
The first case is \eqref{eq:m-topological}, and the second case is \eqref{eq:normal-index-1b}. One has $m(u) \geq 0$ by definition, and $I_N(u) \geq 0$ by Lemma \ref{th:using-admissibility}, which yields the desired property. The chain complex \eqref{eq:relative-floer} is independent of the choice of $\beta_N$ up to canonical isomorphism (the transitions between different choices are given by multiplying the generators by suitable powers of $q$, compare Remark \ref{th:unique-hf}).

\begin{remark} \label{th:lowest-power}
Consider a Floer trajectory $u$ which is contained in $\delta M$ and appears in the definition of the Floer differential. Then Lemma \ref{th:relative-transversality} applies, which in view of \eqref{eq:express-intersection} shows that
\begin{equation} \label{eq:lowest-constraint}
w_N(x_+) + u \cdot (\delta M)_{\zeta} = w_N(x_-).
\end{equation}
In words, this means that all such trajectories come with the lowest possible power of $q$.
One can put this observation into a more algebraic framework, as follows: take the $\bQ[[q]]$-submodule 
\begin{equation} \label{eq:floer-submodule}
\bigoplus_{x \subset \delta M} q^{w_N(x)+1} \bQ[[q]]x \oplus \bigoplus_{x \not\subset \delta M} q^{w_N(x)} \bQ[[q]] x \subset \mathit{CF}^*_q(H_t,J_t).
\end{equation}
Because of \eqref{eq:minimal-m} and \eqref{eq:express-intersection}, this is preserved by $d_q$. The quotient can be thought of as $\bigoplus_{x \subset \delta M} \bQ x$, and the differential on that quotient only counts Floer trajectories $u$ contained inside $\delta M$, and which additionally satisfy \eqref{eq:lowest-constraint} (in particular, it is not necessarily equal to the Floer complex for $\delta M$ in the standard sense).
\end{remark}


\begin{example} \label{th:transverse-max}
Take a Morse function $H$ and almost complex structure $J$ (as in Setup \ref{th:tangent-hamiltonians} and Definition \ref{th:local-splitting}). We ask that at any critical point $x \in \delta M$, the Hessian of $H$ in normal direction to $\delta M$ should be negative. Finally, we require that the gradient flow should be Morse-Smale. This kind of situation is familiar from the Morse-theoretic construction of relative homology (see e.g.\ \cite[Section 4.2]{schwarz}). Its salient features are:
\begin{align}
& \label{eq:x-plus}
\mybox{
If $x_+$ is a critical point of $H$ contained in $\delta M$, then any gradient trajectory whose low-$H$ limit is $x_+$ must be entirely contained in $\delta M$.}
\\ & 
\label{eq:relative-morse}
\mybox{
Take the Morse complex of $H$. From \eqref{eq:x-plus} it follows that the critical points contained in $\delta M$ form a subcomplex. On the Morse cohomology level, the inclusion of that subcomplex computes $H^*(\bar{M},M) \rightarrow H^*(\bar{M})$.
}
\end{align}
%
The only one-periodic orbits of $\epsilon H$, for small $\epsilon>0$, will be constant ones. Critical points contained in $\delta M$ have $w_N(x) = -1$ (see Example \ref{th:constant-orbit}), and therefore
\begin{equation} \label{eq:classical-action}
A_N(x) = \epsilon H(x) - 1.
\end{equation}
It particular, $\epsilon H$ is Floer-admissible if $\epsilon$ is small. The relative Floer cohomology is fairly easy to compute: there are arbitrarily small $\epsilon>0$ such that
\begin{equation} \label{eq:morse-floer}
\mathit{HF}^*_q(\epsilon H,J) \iso q^{-1} H^*(\bar{M}, M;\bQ) \oplus H^*(\bar{M};\bQ)[[q]].
\end{equation}
To see this, start with a Morse-theoretic model:
\begin{equation} \label{eq:q-morse}
\begin{aligned}
& \mathit{CM}^*_q(H) = \bigoplus_{x \in \delta M} q^{-1}\bQ[[q]]x \oplus \bigoplus_{x \notin \delta M} \bQ[[q]]x,
\end{aligned}
\end{equation}
with the usual Morse differential. In view of \eqref{eq:relative-morse}, it is clear that the cohomology of \eqref{eq:q-morse} is the right hand side of \eqref{eq:morse-floer}. These statements are obviously unchanged if we replace $H$ by $\epsilon H$. To derive \eqref{eq:morse-floer}, one follows the argument from \cite[Proposition 7.4]{hofer-salamon95} to show that isolated Floer trajectories correspond bijectively to Morse trajectories; for such trajectories, the intersection number with $(\delta M)_{\zeta}$ will be $0$, for dimension reasons.
\end{example}

\begin{example} \label{th:transverse-min}
Consider a situation similar to the previous one, but where the Hessian in normal direction at any critical point $x \in\delta M$ is positive definite. Hence, as one-periodic orbit of $\epsilon H$ this satisfies $w_N(x) = 0$ and 
\begin{equation}
A_N(x) = \epsilon H(x),
\end{equation}
which again guarantees admissibility for small $\epsilon>0$. Tracing through the same argument as before, one finds that this time, 
\begin{equation} \label{eq:morse-floer-2}
\mathit{HF}^*_q(\epsilon H,J) \iso H^*(\bar{M};\bQ)[[q]].
\end{equation}
\end{example}

\subsection{Continuation maps}
We need to reconsider \eqref{eq:continuation-map-equation} in our setup (meaning the target space is $\bar{M}$, and the Hamiltonian functions and almost complex structures must be chosen as in Setup \ref{th:tangent-hamiltonians}). Only the behaviour of the Hamiltonians on $\delta M$ is relevant here. Take the curvature \ref{eq:curvature} and set
\begin{align} \label{eq:kmin}
& 
K^{\mathit{min}}(F_{s,t}, H_{s,t}) = \int_{\bR \times S^1} \mathrm{min}\big\{R_{s,t}(x) \,:\, x \in \delta M\big\}, \\
& \label{eq:kmax}
K^{\mathit{max}}(F_{s,t}, H_{s,t}) = \int_{\bR \times S^1} \mathrm{max}\big\{R_{s,t}(x) \,:\, x\in \delta M\big\}, \\
&
K(F_{s,t}, H_{s,t}) = K^{\mathit{max}}(F_{s,t}, H_{s,t}) - K^{\mathit{min}}(F_{s,t}, H_{s,t}) \geq 0.
\end{align}
For solutions of the continuation map equation remaining inside $\delta M$, \eqref{eq:topological-energy} yields
\begin{equation} \label{eq:energy-bound}
E^{\mathit{top}}(u) \geq K^{\mathit{min}}(F_{s,t},H_{s,t}).
\end{equation}
Energy, action and the normal index satisfy the same relation \eqref{eq:action-energy-index} as before. In order to deduce surjectivity of $D_{N,u}$ from that, we need to know that $E^{\mathit{top}}(u) - A_N(x_-) + A_N(x_+) >-1$. In view of \eqref{eq:energy-bound}, this holds if the following condition is satisfied.

\begin{definition} \label{th:admissible-2}
We say that $(F_{s,t}, H_{s,t})$ is continuation-admissible if each of the limiting Hamiltonians $(H_{\pm, t})$ is Floer-admissible, and the following additional property holds. Whenever $x_-$ is a one-periodic orbit of $(H_{-,t})$ contained in $\delta M$, and $x_+$ correspondingly for $(H_{+,t})$, 
\begin{equation} \label{eq:admissible-2}
A_N(x_-) - A_N(x_+) < 1 +  K^{\mathit{min}}(F_{s,t}, H_{s,t}).
\end{equation}
\end{definition}

There is also an analogue of Definition \ref{th:local-splitting}:

\begin{definition} \label{th:local-splitting-2}
We say that $(F_{s,t}, H_{s,t}, J_{s,t})$ is locally split if the limiting Hamiltonians $(H_{\pm,t})$ are locally split, and the following additional property holds. Let $x$ be a one-periodic orbit contained in $\delta M$. Then, for $\pm s \gg 0$, we have $(F_{s,t}, H_{s,t}, J_{s,t}) = (0, H_{\pm,t}, J_{\pm,t})$ in a neighbourhood of $x(t)$.
\end{definition}

To define the continuation map, one starts with data that satisfy Definitions \ref{th:admissible-2} and \ref{th:local-splitting-2}. To achieve transversality, one may not only have to perturb the almost complex structure, but also the Hamiltonians, which is unproblematic since \eqref{eq:admissible-2} is an open condition. After that, one can use \eqref{eq:continuation-map-equation} in the same way as in \eqref{eq:top-power}, to define a chain map
\begin{equation} \label{eq:continuation-map}
\mathit{Cont}_q(F_{s,t}, H_{s,t},J_{s,t}): \mathit{CF}^*_q(H_{+,t},J_{+,t}) \longrightarrow \mathit{CF}^*_q(H_{-,t},J_{-,t}).
\end{equation}

\begin{lemma} \label{th:continuation-isomorphism}
Suppose that
\begin{align} \label{eq:c1}
&
\begin{aligned}
& -\!1-K^{\mathit{max}}(F_{s,t}, H_{s,t}) < A_N(x_-) - A_N(x_+) < 1+K^{\mathit{min}}(F_{s,t}, H_{s,t})
\;\; \text{ for $x_-$ a}
\\ & \qquad
\text{one-periodic orbit of $(H_{-,t})$, and $x_+$ a one-periodic orbit of $(H_{+,t})$, both in $\delta M$;} 
\end{aligned} \intertext{and}
&
\label{eq:c2}
\begin{aligned}
&
|A_N(x_+) - A_N(x_-)| < 1-K(F_{s,t}, H_{s,t}) \;\;
\text{where both $x_\pm$ are one-periodic orbits}
\\ & \qquad \text{ of $(H_{+,t})$ in $\delta M$; and the same for two orbits of $(H_{-,t})$.}
\end{aligned}
\end{align}
Then $\mathit{HF}^*(H_{+,t},J_{+,t})$ and $\mathit{HF}^*(H_{-,t},J_{-,t})$ are isomorphic via continuation maps.
\end{lemma}

\begin{proof}
For the reverse family $(F_{-s,t},H_{-s,t})$, we have $K^{\mathit{min}}(F_{-s,t}, H_{-s,t}) = -K^{\mathit{max}}(F_{s,t}, H_{s,t})$. Hence, if we want to define a map in reverse direction to \eqref{eq:continuation-map}, the necessary condition is $A_N(x_+) - A_N(x_-) < 1+K^{\mathit{max}}(F_{s,t}, H_{s,t})$. This explains \eqref{eq:c1}. To show that the continuation maps in both directions are homotopy inverses, one considers the composite continuation map equation (from $H_-$ to $H_+$ and then back, or vice versa), and deforms that back to Floer's equation. The conditions \eqref{eq:admissible-2} required by those equations yield \eqref{eq:c2}.
\end{proof}

A trivial special case is when $F_{s,t} = 0$ and $H_{s,t} = H_t$. This shows that our Floer cohomology groups are independent of the choice of almost complex structure, which we will therefore often omit from the notation.

Take $(H_t)$, with possibly degenerate one-periodic periodic orbits, but such that each orbit contained in $\delta M$ is nondegenerate in normal direction ($\alpha_N(x) \neq 0$), and which is Floer-admissible. Consider $\mathit{HF}^*(\tilde{H}_t)$, where $\tilde{H}_t$ is a small perturbation which makes the one-periodic orbits nondegenerate, and additionally is locally split; one can use Lemma \ref{th:continuation-isomorphism} to see that this is independent of the perturbation up to canonical isomorphism. (Strictly speaking, we should spell out ``small'' to show that there is a contractible space of such choices, but we will omit the details; the condition $\alpha_N(x) \neq 0$ is crucial here, since it prevents discontinuous jumps in $A_N(x)$ when perturbing the Hamiltonian). Those Floer cohomology groups form a locally trivial bundle over the space of all $(H_t)$. As a consequence, one has the following:

\begin{lemma} \label{th:deform-h}
Let $(H_{s,t})$, $(s,t) \in [0,1] \times S^1$, be a family of functions as in Setup \ref{th:tangent-hamiltonians}. For $s = 0,1$, these should have nondegenerate one-periodic orbits and be locally split; and for any value of the parameter $s$, $(H_{s,t})$ must be Floer-admissible, and its one-periodic orbits in $\delta M$ must be nondegenerate in normal direction. Then $\mathit{HF}^*_q(H_{0,t}) \iso \mathit{HF}^*_q(H_{1,t})$. More precisely, there is a distinguished such isomorphism, which is invariant under deforming $(H_{s,t})$ rel endpoints, as long as all the conditions remain true within that deformation.
\end{lemma}

\subsection{Piunikhin-Salamon-Schwarz in the relative setting\label{subsec:pss}}
The following discussion concerns Cauchy-Riemann equations on $\bC$. To stay close to the previous terminology, we think of that Riemann surface as a thimble
\begin{equation} \label{eq:thimble}
T = (\bR \times S^1) \cup \{+\infty\}. 
\end{equation}
The relevant equations \eqref{eq:continuation-map-equation} have domain $T$, and instead of \eqref{eq:plusminus-limit} we have
\begin{equation} \label{eq:plusminus-limit-2}
\left\{
\begin{aligned}
& (F_{s,t}, H_{s,t},J_{s,t}) \longrightarrow (0,H_t,J_t) \quad \text{for $s \rightarrow -\infty$,} 
\\ & 
F_{s,t} \mathit{ds} + H_{s,t} \mathit{dt} \text{ and } J_{s,t} \text{ extend smoothly over $+\infty$.}
\end{aligned}
\right.
\end{equation}
Write $x$ for the limit of solutions $u$ as $s \rightarrow -\infty$. 
If $u$ lies in $\delta M$, the linearized operator $D_u$ fits in a short exact sequence like \eqref{eq:operator-les}. The analogues of \eqref{eq:normal-index}, \eqref{eq:normal-index-1b} and \eqref{eq:action-energy-index} are
\begin{align} \label{eq:normal-index-2}
& 
\half I_N(u) = \nu_u \cdot \text{zero-section} = -w_N(x) + u \cdot (\delta M)_{\zeta}, 
\\ & \label{eq:action-energy-index-2}
E^{\mathit{top}}(u) = A_N(x) + \half I_N(u).
\end{align}
An automatic regularity argument as in Lemma \ref{th:automatic-regularity}, which we will not spell out here, yields:

\begin{lemma} \label{th:automatic-regularity-2}
If $I_N(u) \geq 0$, $D_{N,u}$ is onto; and if $I_N(u) > 0$, evaluation at $+\infty$ yields a surjective map $\mathit{ker}(D_{N,u}) \rightarrow N_{u(+\infty)}$.
\end{lemma}


\begin{definition} \label{th:admissible-3}
We say that an $(F_{s,t}, H_{s,t})$ as in \eqref{eq:plusminus-limit-2} is PSS-admissible if $(H_t)$ is Floer-admissible, and the following additional property holds. Whenever $x$ is a one-periodic orbit of $(H_t)$ contained in $\delta M$, 
\begin{equation} \label{eq:admissible-3}
A_N(x) < K^{\mathit{min}}(F_{s,t}, H_{s,t}).
\end{equation}
\end{definition}

There is also a corresponding notion of locally split $(F_{s,t}, H_{s,t}, J_{s,t})$, as in Definition \ref{th:local-splitting-2}. Fix a manifold $C$ and a map $\gamma: C \rightarrow \bar{M}$. Given $(F_{s,t}, H_{s,t}, J_{s,t})$ as in \eqref{eq:plusminus-limit-2} and a one-periodic orbit $x$ of $(H_t)$, we consider the space of pairs $(c,u)$, where
\begin{equation} \label{eq:pair-moduli-space}
\left\{
\begin{aligned}
& u: T \rightarrow \bar{M} \text{ is a solution of \eqref{eq:continuation-map-equation}, with limit $x$ as $s \rightarrow-\infty$}, \\
& c \in C, \;\; \gamma(c) = u(+\infty).
\end{aligned}
\right.
\end{equation}

\begin{lemma} \label{th:relative-transversality-2}
Suppose that $(F_{s,t}, H_{s,t})$ is PSS-admissible. Then, for a generic small perturbation supported in a compact subset of $\bR \times S^1$, and a generic $(J_{s,t})$ (as always, these choices remain within the classes from Setup \ref{th:tangent-hamiltonians}), each moduli space \eqref{eq:pair-moduli-space} is regular.
\end{lemma}

\begin{proof}
As usual, the only issue is with solutions entirely contained in $\delta M$. For such a solution $u$, the PSS-admissibility condition, \eqref{eq:energy-bound} and \eqref{eq:normal-index-2} tell us that
\begin{equation} \label{eq:positive-normal-index}
\half I_N(u) \geq K^{\mathit{min}}(F_{s,t}, H_{s,t}) - A_N(x) > 0.
\end{equation}
Hence, Lemma \ref{th:automatic-regularity-2} applies. Look at the universal linearized operator $D_u^{\mathit{univ}}$, which includes deformations of $(F_{s,t}, H_{s,t}, J_{s,t})$. Standard regularity theory inside $\delta M$ tells us that the piece $D_{\delta M,u}^{\mathit{univ}}$ responsible for deformations inside $\delta M$ is onto, and has surjective evaluation map $\mathit{ker}(D_{\delta M,u}^{\mathit{univ}}) \rightarrow T(\delta M)_{u(+\infty)}$. Together with the previously obtained information about the normal direction, this implies the corresponding result for $D_u^{\mathit{univ}}$: it is surjective, and $\mathit{ker}(D_u^{\mathit{univ}}) \rightarrow T\bar{M}_{u(+\infty)}$ is onto. Therefore the fibre product of the universal moduli space with $\gamma$ is smooth; a standard application of the Sard-Smale theorem to that leads to the desired result
\end{proof}

Note that Lemma \ref{th:relative-transversality-2} does not require any assumptions on how $\gamma$ intersects $\delta M$. To complete the transversality argument for compactified moduli spaces, one has to include bubbling, in a way which is similar to, but more complicated than, the situation for Lemma \ref{th:relative-transversality}. The issue here are broken solutions which include simple pseudo-holomorphic chains: this means a solution $u: T \rightarrow \bar{M}$ of \eqref{eq:continuation-map-equation} together with non-multiply-covered $J_{\infty}$-holomorphic maps $v_1,\dots,v_r: \bC P^1 \rightarrow \bar{M}$, which have pairwise different images,  and a $c \in C$ such that
\begin{equation} \label{eq:simple-chain}
u(+\infty) = v_1(0), \, v_1(\infty) = v_2(0), \dots, v_{r-1}(\infty) = v_r(0), \, v_r(\infty) = \gamma(c).
\end{equation}
If $v_k$ is contained in $\delta M$, we have $\int v_k^* c_1(N) > 0$, compare \eqref{eq:v-degree}. Therefore, the linearized operator in normal direction, which we denote by $D_{N,v_k}$, is onto, and the evaluation map at $\infty$ on $\mathit{ker}(D_{N,v_k})$ is also surjective. That is enough to establish transversality results for chains \eqref{eq:simple-chain}. (These are the same kinds of arguments that arise when defining Gromov-Witten invariants via symplectic divisors \cite{cieliebak-mohnke07}.)

Suppose that $\gamma$ is a pseudo-cycle. We assume $(F_{s,t}, H_{s,t}, J_{s,t})$ is PSS-admissible, locally split, and that the necessary regularity properties of the moduli spaces are satisfied. One can then use the zero-dimensional spaces \eqref{eq:pair-moduli-space} to define a Floer cocycle 
\begin{equation} \label{eq:pss-q-def-1}
\mathit{PSS}_q(F_{s,t}, H_{s,t}, J_{s,t}, \gamma) \in \mathit{CF}^{\mathrm{codim(\gamma)}}_q(H_t,J_t),
\end{equation}
where the notation is $\mathrm{codim}(\gamma) = \mathrm{dim}(\bar{M}) - \mathrm{dim}(C)$. This involves counting solutions with $\pm q^{u \cdot (\delta M)_{\zeta}}$ exactly as in \eqref{eq:top-power}. For a solution $u$ not contained in $\delta M$, a topological argument similar to that in the definition of the differential shows that $u \cdot (\delta M)_{\zeta} \geq w_N(x)$. For a solution $u$ contained in $\delta M$, we use \eqref{eq:normal-index-2} and \eqref{eq:positive-normal-index} to get
\begin{equation} \label{eq:w-plus-1}
u \cdot (\delta M)_{\zeta} = u \cdot B = w_N(x) + \half I_N(u) \geq w_N(x) + 1.
\end{equation}
This shows that \eqref{eq:pss-q-def-1} actually lies in the subcomplex \eqref{eq:floer-submodule}. 

\begin{remark}
For $u$ contained in $\delta M$, regularity of the moduli space \eqref{eq:pair-moduli-space} means that $D_u$ is onto, and that the evaluation map $\mathit{ker}(D_u) \rightarrow T\bar{M}_{u(+\infty)}$ must be transverse to the image of $D\gamma_c$. This implies that the index of the $\delta M$ component satisfies
\begin{equation} \label{eq:eval}
I_{\delta M}(u) \geq 
\mathrm{rank}\big(\mathit{ker}(D_{\delta M,u}) \rightarrow T(\delta M)_{u(+\infty)}\big) \geq \mathrm{dim}(\delta M) - \mathrm{dim}(C) = \mathrm{codim}(\gamma)-2.
\end{equation}
Since we are looking at moduli spaces of dimension $\leq 1$ in the definition of \eqref{eq:pss-q-def-1}, we have
\begin{equation}
I(u) = I_{\delta M}(u) + I_N(u) \leq \mathrm{codim}(\gamma)+1.
\end{equation}
Taken together with \eqref{eq:positive-normal-index}, these inequalities imply that $I_N(u) = 2$, so equality holds in \eqref{eq:w-plus-1}. This is parallel to what we observed in Remark \ref{th:lowest-power}.
\end{remark}

Suppose that one takes a pseudo-cycle contained in $\delta M$, which we will accordingly denote by $\gamma_{\delta M}$. In that case, every solution of \eqref{eq:pair-moduli-space} whose limit $x$ is not contained in $\delta M$ must intersect that submanifold with positive multiplicity at $\infty$, which means that in the analogue of \eqref{eq:express-intersection} one has $m(u) > 0$. This, together with our previous observation \eqref{eq:w-plus-1}, shows that 
\begin{equation} \label{eq:pss-q-def-2}
\mathit{PSS}_q(F_{s,t}, H_{s,t}, J_{s,t}, \gamma_{\delta M}) \in q\mathit{CF}^*_q(H_t,J_t).
\end{equation}
Given a formal series $\gamma_q = q^{-1} \gamma_{\delta M} + \gamma_0 + q\gamma_1 + \cdots$, where $\gamma_{\delta M}$ is a pseudo-cycle in $\delta M$ and the $\gamma_k$ are pseudo-cycles in $\bar{M}$, we get (in abbreviated notation)
\begin{equation}
\mathit{PSS}_q(\gamma_q) = q^{-1} \mathit{PSS}_q(\gamma_{\delta M}) + 
\mathit{PSS}_q(\gamma_0) + q\mathit{PSS}_q(\gamma_1) + \cdots \in 
\mathit{CF}^*_q(H_t,J_t).
\end{equation}
Via Poincar\'e duality, this leads to a map
\begin{equation} \label{eq:m-pss}
\mathit{PSS}_q: q^{-1} H^*(\bar{M}, M;\bQ) \oplus H^*(\bar{M};\bQ) \longrightarrow \mathit{HF}^*_q(H_t,J_t).
\end{equation}
Strictly speaking, to complete that construction, one has to show that homologous pseudo-cycles yield cohomologous Floer cocycles. However, that argument is entirely standard, and we'll therefore omit it.

\begin{example}
Take a time-independent function $(H_t = H)$ as in Example \ref{th:transverse-max}, and an arbitrary $(F_{s,t}, H_{s,t})$. Then, for sufficiently small $\epsilon>0$, $(\epsilon F_{s,t}, \epsilon H_{s,t})$ is PSS-admissible. This follows directly from \eqref{eq:classical-action} and the fact that
\begin{equation}
K^{\mathit{min}}(\epsilon F_{s,t}, \epsilon H_{s,t}) = \epsilon K^{\mathit{min}}(F_{s,t}, H_{s,t}).
\end{equation}
The computation of Floer cohomology in \eqref{eq:morse-floer} identifies that with the left hand side of \eqref{eq:m-pss}. (One can show that in these specific circumstances \eqref{eq:m-pss} is an isomorphism, but we do not have any use for that statement.)
\end{example}

\begin{example} 
Take $H$ as in Example \ref{th:transverse-min}. Then, for small $\epsilon>0$, there is no PSS map that lands in $\mathit{HF}^*_q(\epsilon H)$, meaning that the PSS-admissibility condition cannot be satisfied. To see that, suppose that we have $(F_{s,t}, H_{s,t})$ which define such a map. Integrating with respect to the volume form associated to $\omega_{\delta M}$, one has
\begin{equation}
\int_{\delta M} R_{s,t} = -\partial_s \int_{\delta M} H_{s,t} + \partial_t \int_{\delta M} F_{s,t},
\end{equation}
and therefore (using Stokes)
\begin{equation}
K^{\mathit{min}}(F_{s,t},H_{s,t}) \leq {\textstyle \frac{1}{\mathrm{vol}(\delta M)}} \int_{\bR \times S^1 \times \delta M} R_{s,t} \,\mathit{ds} \mathit{dt} = {\textstyle \frac{1}{\mathrm{vol}(\delta M)}} \int_{\delta M} \epsilon H \leq \epsilon H(x^{\mathit{max}}) = A_N(x^{\mathit{max}}),
\end{equation}
where $x^{\mathit{max}}$ is the maximum of $H|\delta M$ (which is a critical point of $H$, by assumption). That obviously contradicts \eqref{eq:admissible-3}. 
\end{example}

\section{Relative Floer theory and Lefschetz pencils\label{sec:apply-to-lefschetz}}
Combining the arguments from Sections \ref{sec:mclean} and \ref{sec:relative-floer}, we will now define the relative Floer groups $\mathit{HF}^*_q(\bar{p},r)$ from \eqref{eq:relative-p-floer}, and prove that they satisfy the properties needed in our main argument. This makes essential use of some specific topological properties of the maps $\bar{p}$ arising from anticanonical Lefschetz pencils.

\subsection{The geometric and Floer-theoretic framework}
We work in a relative version of the situation from Section \ref{sec:mclean}.

\begin{setup} \label{th:fibration-2}
Take $\bar{p}: \bar{E} \rightarrow W$ as in Setup \ref{th:fibration}, together with a symplectic divisor $\delta E \subset \bar{E}$ which represents $[\omega_{\bar{E}}]$. Every point of the divisor should satisfy \eqref{eq:locally-trivial}, and $\mathit{TE}^h$ must be tangent to $\delta E$. As a consequence, if we write $\delta M = \delta E \cap \bar{M}$, $\mathit{TE}^h$-parallel transport yields a symplectic isomorphism 
\begin{equation} \label{eq:outside-trivialization-2}
\xymatrix{
W \times \delta M \ar[rr]^-{\iso} \ar[dr] && 
\delta E \ar[dl] \\
& W,
}
\end{equation}
which is compatible with \eqref{eq:outside-trivialization}. Next, consider the normal bundles $N(\delta M)$ of $\delta M \subset \bar{M}$ and $N(\delta E)$ of $\delta E \subset \bar{E}$. The diffeomorphism \eqref{eq:outside-trivialization} induces an isomorphism
\begin{equation} \label{eq:partial-normal-bundle}
\xymatrix{ \ar[d]
W_{\geq 1/2} \times N(\delta M) \ar[rr]^-{\iso} &&  \ar[d]
N(\delta E)|W_{\geq 1/2} \\
W_{\geq 1/2} \times \delta M \ar[rr]^-{\iso} && \delta E|W_{\geq 1/2}
}
\end{equation}
Choose a compatible complex structure on $J_{N(\delta M)}$ on the normal bundle to $\delta M$; carry that over by the isomorphism \eqref{eq:partial-normal-bundle}; and then choose an extension to all of $N(\delta E)$, denoting it by $J_{N(\delta E)}$. We impose a final condition:
\begin{equation} \label{eq:can-be-extended}
\mybox{Take \eqref{eq:partial-normal-bundle}, and modify that by multiplying it with $w$ on the fibre over $w \in W_{\geq 1/2}$. Then, that modified map can be extended to an isomorphism
\[
\xymatrix{ \ar[d]
W \times N(\delta M) \ar[rr]^-{\iso} &&  \ar[d]
N(\delta E) \\
W \times \delta M \ar[rr]^-{\iso}_-{\eqref{eq:outside-trivialization-2}} && \delta E.
}
\]
}
\end{equation}
Note that this condition implies (but is stricter than) our previous topological assumption \eqref{eq:constant-section}.

We allow Hamiltonians $H$ on $\bar{E}$ of the following form. Take a function $H_{\bar{M}}$ on the fibre, satisfying the conditions from Setup \ref{th:tangent-hamiltonians}, and write $H_{\delta M}$ for its restriction to $\delta M$. Also, take a function $H_W$ on the base, which outside a compact subset agrees with $H_{\gamma}$ for some $\gamma \in \frakg$. Then, we want $H|W \times \delta M = H_W+ H_{\delta M}$, and also that outside a compact subset, $H = H_{\gamma} + H_{\bar{M}}$. Finally, we impose the same conditions as in Setup \ref{th:tangent-hamiltonians} on the behaviour of $H$ on the normal bundle to $\delta E$ (this is not a consequence of the previous requirements). To clarify, the functions $H_{\bar{M}}$ (hence $H_{\delta M}$) and $H_{\gamma}$ can vary depending on which $H$ we consider.

Similarly, for almost complex structures, we start with some $J_{\bar{M}}$ as in Setup \ref{th:tangent-hamiltonians}, and its restriction $J_{\delta M}$ to the divisor. We then consider $J$ whose restriction to $W \times \delta M$ is the product $i \times J_{\delta M}$, and which outside a compact subset agrees with $i \times J_{\bar{M}}$ (as before, $J_{\bar{M}}$ can be different for different almost complex structures). Moreover, on the normal bundle to $\delta E$, this should agree with the previously fixed $J_{N(\delta E)}$.
\end{setup}


We will work with a mild variation of the concepts from Section \ref{sec:relative-floer}. Let $(H_t)$ be a time-dependent Hamiltonian; by definition, this comes with an accompanying Hamiltonian $(H_{\bar{M},t})$ on the fibre. If $x$ is a one-periodic orbit (as usual, assumed to be nullhomologous in $\bar{E}$) lying on $\delta M$, we define 
\begin{equation} \label{eq:partial-action}
\begin{aligned}
& A_{N(\partial E)}^{\mathit{fibrewise}}(x) = - \int_S v^*(\omega_{\bar{E}} - \bar{p}^* \omega_{W}) + \int_{S^1} H_{\bar{M},t}(x(t))\, \mathit{dt} + (v|(S \setminus \partial S) \cdot \delta E),
\\ &
\text{where $v: S \rightarrow \bar{E}$ satisfies the analogue of \eqref{eq:normal-action} with respect to $\delta E \subset \bar{E}$.}
\end{aligned}
\end{equation}
Because $\omega_W$ is exact, $\omega_{\bar{E}}-\bar{p}^*\omega_{W}$ is still Poincar\'e dual to $\delta E$, which implies that \eqref{eq:partial-action} is independent of $v$. Take a map $u = (u_W,u_{\bar{M}}): \bR \times S^1 \rightarrow \delta E = W \times \delta M$ which is asymptotic to orbits $x_\pm$. The analogue of \eqref{eq:relative-energy}, using only the energy of the fibre component, says that 
\begin{equation} \label{eq:relative-energy-fibre}
E(u_{\bar{M}}) = A_{N(\delta E)}^{\mathit{fibrewise}}(x_-) - 
A_{N(\delta E)}^{\mathit{fibrewise}}(x_+) + (\nu_u \cdot \text{zero-section}).
\end{equation}
Assuming nondegeneracy of the one-periodic orbits, we have an operator $D_{N(\delta E),u}$ whose index $I_{N(\delta E)}(u)$ can be computed by our usual formula \eqref{eq:normal-index}. As a consequence, we have the counterpart of \eqref{eq:action-energy-index}:
\begin{equation} \label{eq:action-energy-index-3}
E(u_{\bar{M}}) = A_{N(\delta E)}^{\mathit{fibrewise}}(x_-) - A_{N(\delta E)}^{\mathit{fibrewise}}(x_+) + \half I_{N(\delta E)}(u).
\end{equation}

\begin{definition} \label{th:admissible-1b}
We say that $(H_t)$ is fibrewise Floer-admissible if
\begin{equation} \label{eq:fibrewise-admissible}
|A_{N(\delta E)}^{\mathit{fibrewise}}(x_-) - A_{N(\delta E)}^{\mathit{fibrewise}}(x_+)| < 1
\end{equation}
for all one-periodic orbits $x_{\pm}$ contained in $\delta E$.
\end{definition}

The key point here is that, if $u = (u_W,u_{\bar{M}})$ is a Floer trajectory in $\delta E$, then each component satisfies a separate Floer equation, and therefore $E(u_{\bar{M}}) \geq 0$. This explains why $A_{N(\delta E)}^{\mathit{fibrewise}}$ is the relevant quantity in \eqref{eq:fibrewise-admissible}. We now combine this with the other necessary conditions.
\begin{equation} \label{eq:time-2}
\mybox{
Choose a loop $(\gamma_t)$, and corresponding $(H_t)$, as in \eqref{eq:time}. Additionally, we require that each $H_t$ satisfies: the conditions from Setup \ref{th:fibration-2}, the analogue of the locally splitting property from Definition \ref{th:local-splitting} for $\delta E \subset \bar{E}$; and that it is fibrewise Floer-admissible.
}
\end{equation}
As before, an auxiliary choice of section of the normal bundle $\beta_{N(\delta E)}$, which is nonzero along one-periodic orbits of $(H_t)$, comes in handy. It gives rise to winding numbers $w_{N(\delta E)}(x) \in \bZ$, in the same as way as in \eqref{eq:winding-number}, and leading to a formula parallel to \eqref{eq:normal-action-1b}:
\begin{equation} \label{eq:partial-action-2}
\begin{aligned}
& A_{N(\delta E)}^{\mathit{fibrewise}}(x) =  - \int_S v^*(\omega_{\bar{E}} - \bar{p}^*\omega_W) + \int_{S^1}\, H_{\bar{M},t}(x(t)) \mathit{dt} + v \cdot (\delta E)_{\zeta} + w_N(x) \\ & \qquad \qquad\qquad \qquad \qquad \qquad \qquad \text{for any $v: S \rightarrow \bar{E}$, $v|\partial s = x$.}
\end{aligned}
\end{equation}
We find it convenient to make a more specific choice.

\begin{setup} \label{th:product-section}
Suppose that are given a section $\beta_{N(\delta M)}$ which is transverse to the zero-section, and such that $B_{N(\delta M)} = \beta_{N(\delta M)}^{-1}(0)$ is disjoint from the one-periodic orbits of $(H_{\bar{M},t})$. Because of \eqref{eq:can-be-extended}, we can find a section $\beta_{N(\delta E)}$ which equals $(w/|w|^2)\, \beta_{N(\delta M)}$ over any $|w| \geq 1/2$, and whose zero-set (transversally cut out) is $B_{N(\delta E)} = W \times B_{N(\delta M)}$.
\end{setup}

Setup \ref{th:product-section} has the following useful consequence.
\begin{equation} \label{eq:two-w}
\mybox{
Take a one-periodic orbit in $\delta E|W_{\geq 1/2} = W_{\geq 1/2} \times \delta M$, and write it as $x(t) = (x_W(t),x_{\bar{M}}(t))$. Then 
\[
w_{N(\delta E)}(x) = -\mathrm{wind}(x_W) + w_{N(\delta M)}(x_{\bar{M}}).
\]
Here, for clarity we have written $w_{N(\delta M)}$ for the quantity \eqref{eq:winding-number} associated to $x_{\bar{M}}$ using $\beta_{N(\delta M)}$; and $\mathrm{wind}(x_W)$ is the standard winding number of a loop in $W_{\geq 1/2}$ around the central hole.
}
\end{equation}
From here on, definition of $\mathit{CF}^*_q(H_t,J_t)$ follows exactly the same method as in \eqref{eq:relative-floer}. Note, however, that the resulting relative Floer cohomology group depends not only on $r$, but also on the choice of Hamiltonian.

We will very briefly discuss the situation for continuation maps and PSS maps. When constructing a continuation map, one uses Hamiltonians $(F_{s,t})$ and $(H_{s,t})$ which, because of the general prescription from Setup \ref{th:fibration-2}, come with accompanying $(F_{\bar{M},s,t})$, $(H_{\bar{M},s,t})$. Given a map $u = (u_W,u_{\bar{M}}): \bR \times S^1 \rightarrow W \times \delta M = \delta E$, one considers the topological energy \eqref{eq:topological-energy} only for the fibrewise component. Correspondingly, one only considers the curvature quantity obtained by applying \eqref{eq:curvature}
to $(F_{\bar{M},s,t})$, $(H_{\bar{M},s,t})$. The fibrewise continuation-admissibility condition says that
\begin{equation} \label{eq:fibrewise-admissible-2}
A_{N(\delta E)}^{\mathit{fibrewise}}(x_-) - A_{N(\delta E)}^{\mathit{fibrewise}}(x_+) < 1 + K^{\mathit{min}}(F_{\bar{M},s,t}, H_{\bar{M},s,t}).
\end{equation}
In parallel, the fibrewise admissibility condition for PSS maps is
\begin{equation} \label{eq:fibrewise-admissible-3}
A_{N(\delta E)}^{\mathit{fibrewise}}(x) < K^{\mathit{min}}(F_{\bar{M},s,t}, H_{\bar{M},s,t}).
\end{equation}

\subsection{The explicit construction\label{subsec:concrete}}
We now describe the specific choice of Hamiltonian underlying the definition of our first Floer cohomology group, namely $\mathit{HF}^*_q(\bar{p},1)$. This all takes place under the constraints from \eqref{eq:time-2}, but is additionally designed so that the insights from Section \ref{subsec:compute-hf} carry over.
%
%
\begin{align}
& \label{eq:e-hessian}
\mybox{Fix a small $\epsilon>0$. Take a time-dependent $(H_{\bar{E},t})$ (as in Setup \ref{th:fibration-2}) which agrees with $(1-\epsilon)H^{\mathit{rot}}$ both on $\delta E = W \times \delta M$ and on $W_{\geq 1/2} \times \bar{M}$. This means that all points of $\{0\} \times \delta M \subset \delta E$ are critical points of $H_{\bar{E},t}$ for any $t$. We additionally require that along that submanifold, the Hessian in direction of $N(\delta E)$ should be zero.}
\\ & \label{eq:e-hessian-1b}
\mybox{Take a Morse function $H_{\bar{M}}'$ on the fibre (as in Setup \ref{th:tangent-hamiltonians}), such that for each critical point lying on $\delta M$, the Hessian in normal direction to that submanifold is positive definite.}
\\ &
\label{eq:e-hessian-2}
\mybox{Take a function $H_{\bar{E}}'$ which agrees with $H_{\bar{M}}'$ both on $W \times \delta M$ and $W_{\geq 1/2} \times \bar{M}$. Moreover, at every critical point lying in $\{0\} \times \delta M$, the Hessian in direction normal to $N(\delta E)$ should be negative definite.
}
\end{align}
Introduce another positive constant $\epsilon' \ll \epsilon$, and define
\begin{equation} \label{eq:relative-2}
H_{2,t} = \begin{cases} 
H_{\bar{E},t} + \epsilon' H_{\bar{E}}' & \text{on $\bar{E}|W_{\leq 1/2}$,} \\
\epsilon (H_2 \circ \phi_{-t}^{\mathit{rot}}) + H^{\mathit{rot}} + \epsilon' H_{\bar{M}}'
& \text{on $W_{\geq 1/2} \times \bar{M}$;}
\end{cases}
\end{equation}
This is a special case of \eqref{eq:r2-hamiltonian}, for $r_2 = 1$. The only nontrivial aspect is fibrewise admissibility, which is a consequence of the following observation.

\begin{lemma} \label{th:contained}
Any one-periodic orbit of $(H_{2,t})$ lying in $\delta E = W \times \bar{M}$ is constant in $\bar{M}$-direction; and its action in the sense of \eqref{eq:partial-action} is
\begin{equation}
A_{N(\delta E)}^{\mathit{fibrewise}}(x) = \epsilon' H'_{\bar{M}}(x)-1. 
\end{equation}
\end{lemma}

\begin{proof}
We have
\begin{equation} \label{eq:h2-on-the-divisor}
H_{2,t}|W \times \delta M = H^{\mathit{rot}} + \epsilon(H_2 \circ \phi_{-t}^{\mathit{rot}}) + \epsilon' (H_{\bar{M}}'|\delta M).
\end{equation}
From this, it is clear that the one-periodic orbits are indeed constant in $\bar{M}$-direction. 
In fact, there are two kinds of such orbits (compare Lemma \ref{th:h1-h2}):
\begin{align}
& \label{eq:n-action}
\mybox{Constant orbits $(0,x)$, for $x \in \delta M$ a critical point of $H_M'$. Because of the negative definiteness assumption in \eqref{eq:e-hessian-2}, these orbits have $w_{N(\delta E)} = -1$. In \eqref{eq:partial-action-2} one can then choose $v$ constant,
and then $A_{N(\delta E)}^{\mathit{fibrewise}} = H_{\bar{M}}(x) + w_{N(\delta E)}(0,x)$ as desired (note that the $W$-component of the Hamiltonian, given by the first two terms in \eqref{eq:h2-on-the-divisor}, does not contribute to this computation).}
\\
& \label{eq:n-action-2}
\mybox{Orbits $(\phi_t^{\mathit{rot}}(w_2),x)$, where $w_2$ is as in Lemma \ref{th:morse-function}, and $x$ is as in \eqref{eq:n-action}. The Hessian in normal direction is positive definite by \eqref{eq:e-hessian-1b}. However, in \eqref{eq:two-w} we have $\mathrm{wind}(x_W) = 1$, and hence again $w_{N(\delta E)} = -1$. In \eqref{eq:partial-action-2} one can choose $v$ to be constant in $\bar{M}$-direction, which because of Setup \ref{th:product-section} means that $v \cdot (\delta E)_{\zeta} = \emptyset$. Hence, the outcome is ultimately the same as in the previous case. 
}
\end{align}
\end{proof}

In spite of all the restrictions we have imposed on the Hamiltonian, and corresponding ones on the almost complex structure, one can still achieve regularity of the moduli spaces of Floer trajectories. (For Floer trajectories remaining in $W \times \delta M$, the fact that all our choices respect the product decomposition is not an obstruction to generic regularity inside that manifold; and then, regularity in $\bar{E}$ comes from an appropriate version of Lemma \ref{th:using-admissibility}, which itself relies on \eqref{eq:action-energy-index-3} and automatic regularity.) This gives us the relative Floer complex $\mathit{CF}^*_q(H_{2,t},J_t)$ as in \eqref{eq:relative-floer}. Its cohomology is the desired $\mathit{HF}^*_q(\bar{p},1)$. (One can use a version of Lemma \ref{th:deform-h} to prove that this Floer cohomology group is independent of the various technical details that enter into its definition; that is not really necessary for our purpose, since relative Floer cohomology appears in our argument mainly as an intermediate technical tool, not as a significant invariant in itself).

To define $\mathit{HF}^*_q(\bar{p},r)$ for $r = 1+\epsilon$, one replaces \eqref{eq:relative-2} with the Hamiltonian
\begin{equation} \label{eq:relative-3}
H_{3,t} = \begin{cases} 
H_{\bar{E}} + \epsilon' H_{\bar{E}}' & \text{on $\bar{E}|W_{\leq 1/2}$,} \\
\epsilon (H_3 \circ \phi_{-t}^{\mathit{rot}}) + H^{\mathit{rot}} + \epsilon' H_{\bar{M}}'
& \text{on $W_{\geq 1/2} \times \bar{M}$.}
\end{cases}
\end{equation}
Lemma \ref{th:contained} carries over without any problems, and that establishes fibrewise admissibility; the rest of the construction is as before. (One can extend the definition to $\mathit{HF}^*_q(\bar{p},r)$ for any $r \in (1,2)$ by modifying the Hamiltonian by a further rotation over $W_{\geq \rho_3}$, but those groups are all isomorphic to each other, as can easily be seen by a continuation map argument.) 

Next, consider the construction of the continuation map 
\begin{equation}
C_q^{1,1+\epsilon}: \mathit{HF}^*_q(\bar{p},1) \longrightarrow \mathit{HF}^*_q(\bar{p},1+\epsilon).
\end{equation}
One can use an $(s,t)$-dependent family of Hamiltonians which interpolates between $(H_{2,t})$ and $(H_{3,t})$ in a straightforward way, along the same lines as in the proofs of Lemma \ref{th:continuation-inclusion} and \ref{th:continuation-inclusion-2}. In particular, the $\bar{M}$-part of the Hamiltonian remains $\epsilon H_{\bar{M}}'$ throughout, which means that the curvature bound in \eqref{eq:fibrewise-admissible-2} is actually zero. Lemma \ref{th:contained} and its analogue for \eqref{eq:relative-3} then guarantee that the fibrewise admissibility condition \eqref{eq:fibrewise-admissible-3} is satisfied. On a much more concrete level, the same argument as in Lemma \ref{th:continuation-inclusion-2} applies, showing that the continuation map is an inclusion on the chain level, whose cokernel is generated by one-periodic orbits associated to the critical point $w_3$. The cohomology of the resulting quotient complex can be computed by reducing it to the Floer complex for $\epsilon H_{\bar{M},t}'$, and hence to Example \ref{th:transverse-min}, with the same provisos concerning gradings as in \eqref{eq:third-orbits}. This completes the proof of Lemma \ref{th:les-3} (strictly speaking, we have proved this for $r = 1+\epsilon$, but as explained above, the extension to $r \in (1,2)$ is straightforward).

The same observations apply to the construction of PSS maps \eqref{eq:pss-r-2}. In \eqref{eq:fibrewise-admissible-3} the curvature term will be of order $\epsilon'$, since it depends on deforming $\epsilon' H_{\bar{M}}'$ to zero; as the $A_N{(\delta E)}^{\mathit{fibrewise}}$ are close to $-1$, that admissibility condition \eqref{eq:fibrewise-admissible-3} is obviously satisfied. 

\subsection{Compactification and a Gromov-Witten invariant}
Our remaining task is to prove Lemma \ref{th:kernel-of-pss-2}; or to say it more appropriately, to check that the argument from \cite[Lemma 8.4]{seidel16} goes through in the relative setting. This not particularly difficult, but requires us to set up the geometry of the relevant moduli spaces so that it fits in with the rest of our setting.

\begin{setup} \label{th:fibration-3}
Given the situation from Setup \ref{th:fibration-2}, we construct a compact symplectic manifold $\cornerbar{E}$ with a map \eqref{eq:cornerbar-p} as follows. Taking $\bC|_{\geq 1/2} = \{|w| \geq 1/2\} \cup \{\infty\} \subset \bC| = \bC \cup \{\infty\}$, one sets 
\begin{equation} \label{eq:cornerbar-e}
\cornerbar{E} = E \cup_{(W_{\geq 1/2} \times \bar{M})} (\bC|_{\geq 1/2} \times \bar{M}),
\end{equation}
which comes with an obvious map $\cornerbar{p}$ to $\bC|$. The exceptional divisor is similarly
\begin{equation}
\delta E| = \delta E \cup_{(W_{\geq 1/2} \times \delta M)} (\bC|_{\geq 1/2} \times \delta M) \iso \bC| \times \delta M,
\end{equation}
where the second identification comes from \eqref{eq:outside-trivialization-2}. We fix a complex structure on the normal bundle to $\delta M$. Carry this over to a complex structure for $N(\delta E|)|(\bC_{\geq 1/2} \times \delta M)$ in the obvious way, and then choose an extension to all of $N(\delta E|)$, denoted by $J_{N(\delta E|)}$.

The symplectic form $\omega_{\bar{E}}$ cannot extend to $\cornerbar{E}$, and we remedy that by a suitable cutoff. Take a rotation-invariant symplectic form $\omega_{\bC|} \in \Omega^2(\bC|)$ which agrees with $\omega_W$ on $|w| \leq \rho_\infty$, for some $\rho_\infty$ which is smaller than, but sufficiently close to, $1$. Set
\begin{equation}
\omega_{\cornersubbar{E}} = \begin{cases} \omega_{\bar{E}} & \text{on $\bar{E}|W_{\leq 1/2}$}, \\ 
\omega_{\bC|} + \omega_{\bar{M}} & \text{elsewhere.}
\end{cases}
\end{equation}
\end{setup}

\begin{remark}
When first discussing this in Section \ref{sec:intro}, the examples were taken from algebraic geometry; whereas in the current more technical considerations, we have constructed a space \eqref{eq:cornerbar-e} with very explicit constraints on its symplectic geometry. Still, making the connection between the two is not particularly hard. Given an algebro-geometric Lefschetz pencil, one first blows up the base locus to get a variety fibered over the projective line. One removes the fibre at $\infty$ (assumed to be smooth), and deforms the symplectic form to bring it into the framework from Setup \ref{th:fibration-2}. The resulting compactification \eqref{eq:cornerbar-e} will be deformation equivalent to the original fibered variety; that suffices for our purposes, since the only piece of algebraic geometry we have used is the computation of Gromov-Witten invariants.
\end{remark}

\begin{lemma}
The normal bundle $N(\delta E|)$ (of $\delta E| = \bC| \times \delta M|$ inside $\cornerbar{E}$) is 
\begin{equation} \label{eq:twisted-normal}
N(\delta E|) \iso \scrO_{\bC|}(-1) \boxtimes N(\delta M).
\end{equation}
\end{lemma}

This is a straightforward consequence of \eqref{eq:can-be-extended}: the isomorphism given there combines with the obvious isomorphism $N(\delta E|) | (\bC|_{\geq 1/2} \times \delta M) \iso N(\delta M)$ to give \eqref{eq:twisted-normal}.

\begin{lemma}
The symplectic class of $\cornerbar{E}$ satisfies
\begin{equation} \label{eq:omega-class}
[\omega_{\cornersubbar{E}} - \cornerbar{p}^*\omega_{\bC|}]  = [\delta E|] + [\cornerbar{E}_\infty].
\end{equation}
\end{lemma}

Since \eqref{eq:omega-class} is correct after restricting to $H^2(E|)$, all one has to do is to check that the two sides have equal integral over some class in $H_2(\cornerbar{E})$ which has nonzero degree over $\bC|$. For that, one takes a constant section $\bC| \times \{x\}$, for $x \in \delta M|$. By \eqref{eq:twisted-normal}, the intersection number of $\delta E|$ and that section is $-1$, and the rest is obvious.

\begin{setup}
The class of almost complex structures on $\cornerbar{E}$ we will allow is characterized as follows. Start by choosing $J_{\bar{M}}$ on $\bar{M}$ as in Setup \ref{th:tangent-hamiltonians}, and its restriction $J_{\delta M}$. Our almost complex structure $J$ must agree with $i \times \delta J_{\delta M}$ on $\delta E| \iso \bC \times \delta M$, and with $i \times J_{\bar{M}}$ on $\bC|_{\geq \rho_\infty} \times \bar{M}$. Moreover, on the normal bundle to $\delta E|$, it should reproduce the previously fixed $J_{N(\delta E|)}$.

For Hamiltonians, start by taking $H_{\bar{M}}$ as in Setup \ref{th:tangent-hamiltonians}, and its restriction $H_{\delta M}$. Also, take a function $H_{\bC}$ on the base, such that near infinity, the associated Hamiltonian vector field is the infinitesimal rotational vector field at some speed $\theta \in \bR$. Then, we want $H = H_{\bC|} + H_{\delta M}$ on $\bC \times \delta M$, and similarly $H = H_{\bC|} + H_{\bar{M}}$ on $\bC_{\geq \rho_\infty} \times \bar{M}$. Finally, on $N(\delta E|)$ we want the same behaviour as in Setup \ref{th:tangent-hamiltonians}.
\end{setup}

Take a time-dependent Hamiltonian $(H_t)$ of this kind. By assumption, the Hamiltonian vector field rotates the fibres over $\bC|_{\geq \rho_\infty}$ with some speed $\theta_t$, and we require that the total rotation number satisfies
\begin{equation} \label{eq:overall-speed}
r = \frac{1}{2\pi}\int_0^1 \theta_t \mathit{dt} \notin \bZ.
\end{equation}
Take a one-periodic orbit that is not contained in the fibre over $\infty$. This is automatically disjoint from $\bC|_{\geq \rho_\infty} \times M$, hence can be thought of as lying inside our original $\bar{E}$. We will stick to the assumption of considering only those orbits that are nullhomologous in $\bar{E}$. For those, we write down a slightly more general form of \eqref{eq:partial-action}, namely
\begin{equation}
\begin{aligned}
& 
A_{N(\delta E)}^{\mathit{fibrewise}}(x) = - \int_S v^*(\omega_{\cornersubbar{E}} - \cornerbar{p}^* \omega_{\bC|}) + \int_{S^1} H_{\bar{M},t}(x(t))\, \mathit{dt} + \big(v|(S \setminus \partial S) \cdot (\delta E| + \cornerbar{E}_\infty)\big),
\\ &
\text{where $v: S \rightarrow \cornerbar{E}$ has the same boundary behaviour as in \eqref{eq:partial-action}.}
\end{aligned}
\end{equation}
If we choose the bounding surface $v$ to take lie in $\bar{E}$, then that just reproduces \eqref{eq:partial-action}. On the other hand, it is independent of $v$, by \eqref{eq:omega-class}, and the two facts together show that the general formula is correct.

Because our topic of discussion involves a version of PSS maps, we will be working with maps defined on the thimble \eqref{eq:thimble}. This means that we choose $(F_{s,t},H_{s,t},J_{s,t})$ on $\cornerbar{E}$ as in \eqref{eq:plusminus-limit-2}.
Suppose that we are given a map $u = (u_{\bC|}, u_{\bar{M}}): T \longrightarrow \delta E| = \bC| \times \delta M$, whose limit is a one-periodic orbit of the kind discussed above. The analogue of \eqref{eq:relative-energy-fibre} says that
\begin{equation} \label{eq:relative-energy-fibre-2}
\begin{aligned} &
E^{\mathit{top}}(u_{\bar{M}}) = A_{N(\partial E)}^{\mathit{fibrewise}}(x) + (\nu_u \cdot \text{zero-section}) + (u \cdot \cornerbar{E}_\infty), \\
& 
\text{where $\nu_u$ is a section $u^*N(\delta E|)$ with the same asymptotic behaviour as in \eqref{eq:relative-energy}.}
\end{aligned}
\end{equation}
Now suppose that $x$ is nondegenerate. Then we have a linearized operator (or rather, a class of such operators) $D_{N(\delta E|),u}$ in direction normal to $\delta E|$. The formula for its index $I_{N(\delta E|)}(u)$ is as in \eqref{eq:normal-index}, and that combines with \eqref{eq:relative-energy-fibre-2} to give
\begin{equation} \label{eq:aei}
E^{\mathit{top}}(u_{\bar{M}}) = A_{N(\partial E)}^{\mathit{fibrewise}}(x) + \half I_{N(\delta E|)}(u) + (u \cdot \cornerbar{E}_\infty).
\end{equation}

\begin{remark}
As a sanity check on \eqref{eq:aei}, consider what happens if we glue (in a purely topological sense) a sphere $w = (w_{\bC|}, w_{\bar{M}}): S^2 \rightarrow \bC| \times \delta M$ into $u$. The left hand side changes by the area of $w_{\bar{M}}$, which by \eqref{eq:omega-class} is $w \cdot ([\delta E|] + [\cornerbar{E}_\infty])$. Those two terms exactly match the changes in $I_{N(\delta E|)}(u)$ and $u \cdot \cornerbar{E}_\infty$.
\end{remark}

Suppose that $(F_{s,t},H_{s,t})$ satisfy the fibrewise version of the PSS-admissibility condition, introduced in \eqref{eq:fibrewise-admissible-3}. If $u$ is a solution of the associated Cauchy-Riemann equation, we combine \eqref{eq:aei} with an appropriate version of \eqref{eq:energy-bound} to obtain
\begin{equation} \label{eq:eeeee}
0 < \half I_{N(\delta E|)}(u) + (u \cdot \cornerbar{E}_\infty).
\end{equation}
Lemma \ref{th:automatic-regularity-2} also works in this context. The consequence of that and \eqref{eq:eeeee} is:

\begin{corollary} \label{th:final-regularity}
If \eqref{eq:fibrewise-admissible-3} is satisfied and we are considering maps with $u \cdot \cornerbar{E}_\infty = 1$, any associated operator $D_{N(\delta E|),u}$ is surjective.
\end{corollary}

We extend Setup \ref{th:product-section} to the compactification $\cornerbar{E}$, as follows.

\begin{setup} \label{th:product-section-2}
Take $\beta_{N(\delta M)}$ as in Setup \ref{th:product-section}. Then there is a section $\beta_{N(\delta E|)}$ of the normal bundle to $\delta E| \subset \cornerbar{E}$ which equals $(w/|w|^2) \beta_{N(\delta M)} = \bar{w}^{-1} \beta_{N(\delta M)}$ over $\bC|_{\geq 1/2} \times \delta M$, and whose zero-set is $(\bC| \times B(\delta M)) \cup -(\infty \times \delta M)$, with the minus sign denoting the reversal of the usual orientation. The section $\beta_{N(\delta E|)}$ gives rise to perturbed versions of $\delta E|$, denoted as usual by $(\delta E|)_{\zeta}$.
\end{setup}

We are now to ready assemble all the ingredients into the desired construction. Let's assume, in the notation from \eqref{eq:overall-speed}, that $r>0$. On a much more technical level, we assume that the analogue of the locally split condition from Definition \ref{th:local-splitting} holds. Fix some $s_* \in \bR$ and consider maps
\begin{equation} \label{eq:through-infinity}
\left\{
\begin{aligned}
& u: T \longrightarrow \cornerbar{E}, \\
& u \cdot \cornerbar{E}_\infty = 1, \\
& u(\infty) \in \cornerbar{E}_\infty, \\
& u(s_*,0) \in \cornerbar{E}_{w_*}, \\
& \textstyle\lim_{s \rightarrow -\infty} u(s,\cdot) = x(t) \in \bar{E},
\end{aligned}
\right.
\end{equation}
which are solutions of the PSS equation associated to some choice of $(F_{s,t}, H_{s,t})$, satisfying \eqref{eq:fibrewise-admissible-3}. For a solution contained in $\delta E|$, Corollary \ref{th:final-regularity} ensures surjectivity of the linearized operator in normal direction (the intersection conditions in \eqref{eq:through-infinity} do not have a component in normal direction to $\delta E|$, hence do not affect this statement), and that is sufficient in order for the standard transversality methods to go through. One can encode the counting of isolated solutions $u$ in a Floer cocycle in $\mathit{CF}^2_q(H_t,J_t)$, defined as follows:
\begin{equation} \label{eq:provisional}
\text{$x$-coefficient } = \sum_u \pm q^{u \cdot (\delta E|)_{\zeta}}.
\end{equation}
We will not go through all the details of the construction, since it follows ideas that have been thoroughly explored in \cite{seidel16}; but we will touch on the salient points.

First, let's check that the powers of $q$ on the right hand side really define an element of $\mathit{CF}^*_q(H_t,J_t)$, meaning that 
\begin{equation} \label{eq:lies-in}
u \cdot (\delta E|)_{\zeta} \geq w_{N(\delta E)}(x). 
\end{equation}
For a solution $u$ not contained in $\delta E|$, but whose limit is contained in $\delta E|$, we have an analogue of \eqref{eq:m-topological}. This says that $u \cdot (\delta E|)_{\zeta} - w_{N(\delta E)}(x)$ is positive (the sum of a positive contributions from the limit and from each of the points of $u \cdot \delta E|$). For a solution contained in $\delta E|$, one can use Setup \ref{th:product-section-2} to rewrite the index formula as $\half I_N(u) = - w_{N(\delta E)}(x) + u \cdot (\bC \times B(\delta M) - \{\infty\} \times \delta M)$, where the intersection numbers are taken inside $\delta E| = \bC| \times \delta M$; or equivalently, as $\half I_N(u) = -w_{N(\delta E)}(x) + u \cdot (\delta E|)_{\zeta}$, where now the intersection number takes place in $\cornerbar{E}$. This and Corollary \ref{th:final-regularity} complete the verification of \eqref{eq:lies-in}. 

The other issue is whether \eqref{eq:provisional} is actually a cocycle. This involves looking at undesirable points that can occur in the compactification of moduli spaces of dimension $\leq 1$. One of them is a splitting into two components (a Floer trajectory, and a map on the thimble) where the intermediate periodic orbit lies in $\cornerbar{E}_\infty$, see Figure \ref{fig:degenerate}. In that case, one can prove that the sequence converging to that limit must have an additional intersection point with $\cornerbar{E}_\infty$, thereby contradicting the intersection number assumption in \eqref{eq:through-infinity} (technically, the argument is based on \cite[Lemma 7.7]{seidel16}). Another, potentially more problematic, limiting configuration is shown in Figure \ref{fig:degenerate-2}. This can again be ruled out by showing that a sequence converging to that limit must have intersection number $\geq 2$ with $\cornerbar{E}_\infty$. This once more uses \cite[Lemma 7.7]{seidel16}, but in this case additionally depends on $r>0$.
\begin{figure}
\begin{centering}
\includegraphics{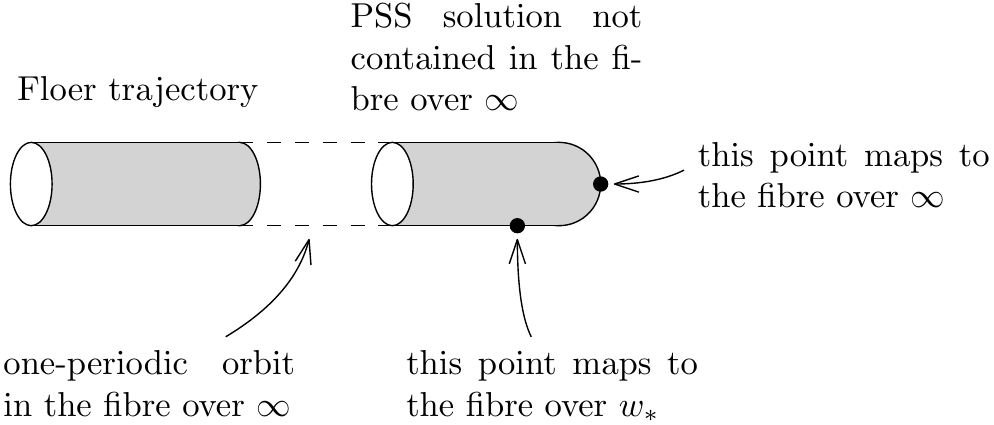}
\caption{\label{fig:degenerate}A hypothetical limit point of the space from \eqref{eq:through-infinity}, which we want to rule out.}
\end{centering}
\end{figure}
\begin{figure}
\begin{centering}
\includegraphics{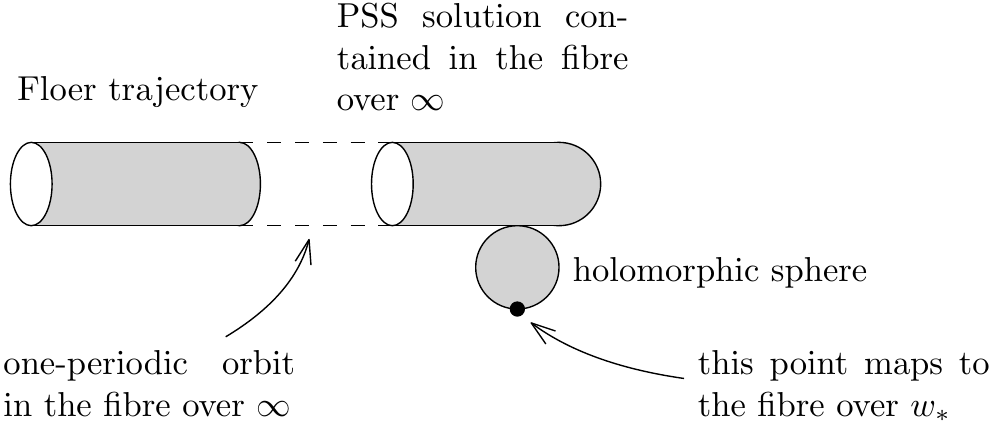}
\caption{\label{fig:degenerate-2}A more complicated variant of Figure \ref{fig:degenerate}.}
\end{centering}
\end{figure}

Actually, the cocycle \eqref{eq:provisional} is not exactly what we're aiming for. Instead, we want to consider a parametrized version where $s_*$ can move. For the purpose of exposition, one can divide the consequences of using such a parametrized moduli space into two parts.

\begin{lemma} \label{th:right}
In the situation above, the class in $\mathit{HF}_q^2(\bar{p},r)$ of the cocycle from \eqref{eq:provisional} is the image under $\mathit{PSS}_q^r$ of $z^{(1)}|\bar{E}$, where $z^{(1)}$ is the Gromov-Witten invariant from \eqref{eq:z1}.
\end{lemma}

\begin{lemma} \label{th:left}
If we moreover assume that $r>1$, the cocycle defined by \eqref{eq:provisional} is nullhomologous.
\end{lemma}

This can be applied to the specific Hamiltonians used in Section \ref{subsec:concrete} to define $\mathit{HF}^*_q(\bar{E},1+\epsilon)$, and then, the two Lemmas taken together imply Lemma \ref{th:kernel-of-pss-2}. Since \cite[Lemma 8.4]{seidel16} already contains both parts of the argument, we will only offer the briefest of summaries here. For Lemma \ref{th:left}, one moves $s_*$ to the left, and shows that the resulting parametrized moduli space has no limit points with $s_* \rightarrow -\infty$. In the simplest case, such limits would have the form indicated in Figure \ref{fig:degenerate-3}, but that contradicts the intersection assumption in the original moduli space. The assumption $r>1$ rules out more complicated limits, such as the version of Figure \ref{fig:degenerate-3} where the intermediate periodic orbit, and possibly also the entire PSS component of the limit, lies in $\cornerbar{E}_\infty$. Along the same lines, Lemma \ref{th:right} involves moving $s_*$ to the right, and a bubbling argument for $s_* \rightarrow +\infty$, with the relevant limits shown in Figure \ref{fig:degenerate-5}.
\begin{figure}
\begin{centering}
\includegraphics{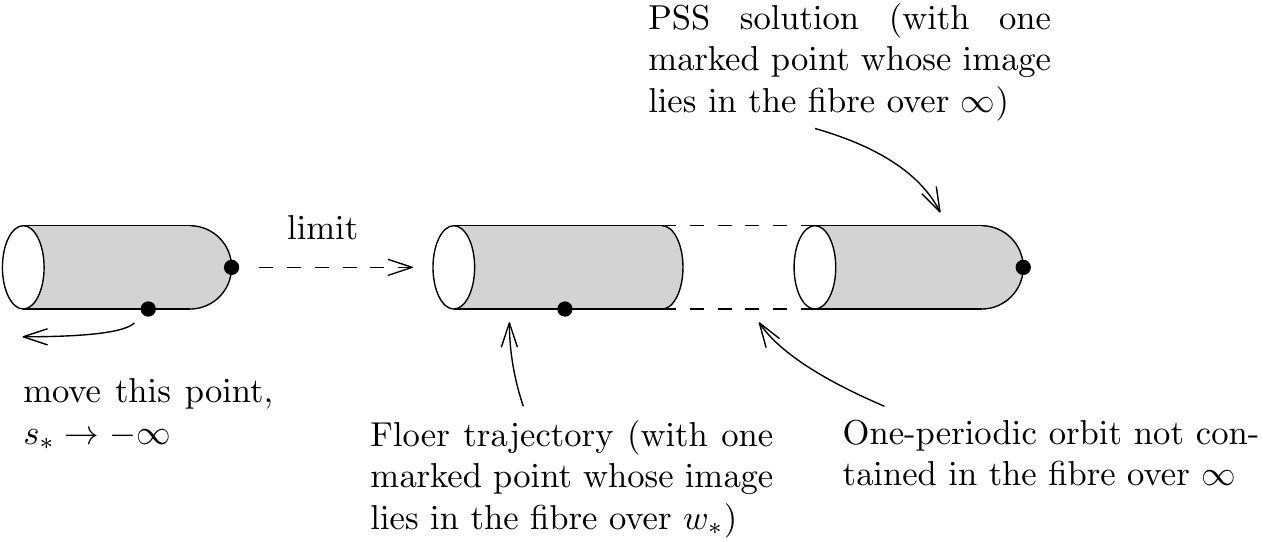}
\caption{\label{fig:degenerate-3}One of the limit points that have to be ruled out to derive Lemma \ref{th:left}.}
\end{centering}
\end{figure}
\begin{figure}
\begin{centering}
\includegraphics{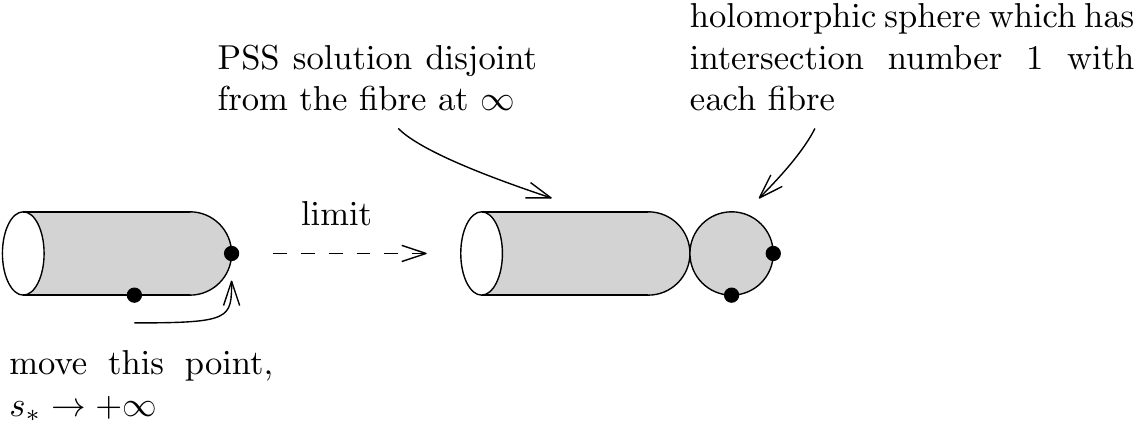}
\caption{\label{fig:degenerate-5}The bubbling process underlying the proof of Lemma \ref{th:right}.}
\end{centering}
\end{figure}

\section{\label{sec:general}$B$-fields}
This section contains the proof of Theorem \ref{th:main}. We will approach this pragmatically, in a way which remains parallel to the special case discussed before, hence requiring only a minimum of additional explanations. The basic insight comes from \cite[Section 3]{seidel15}, where it was shown that one can always introduce a $B$-field twist for which the twisted Gromov-Witten invariants satisfy \eqref{eq:write-z1}. Pursuing that direction would lead to a result which is strictly parallel to Theorem \ref{th:1-variable}, for the correspondingly twisted Fukaya category. However, we're really looking for a statement that applies to the Fukaya category in the standard sense. The idea is to carry out the $B$-field twist argument in a Novikov ring with several variables, where the resulting twisted theory can be specialized back to the ordinary Fukaya category. This is not how we have stated things in Section \ref{subsec:multivariable}, so a bit of extra work has to be done in order to connect the two formulations. After that, we will be rather brief concerning the proof, since that repeats our previous argument, with extra (topologically defined) weights when counting pseudo-holomorphic curves.

\subsection{The formal framework\label{subsec:b-field}}
Let's assume that \eqref{eq:connected} holds, which gives us the preferred homology class $A_*$. From Section \ref{subsec:multivariable} we use the lattice $H$ and its sublattice $H_0^0$ (the degree $0$ part of $H_0$). The dual lattices are
\begin{equation}
\begin{aligned}
& H^\vee = \mathit{Hom}(H,\bZ) = H^2(\cornerbar{E};\bZ)/\mathit{torsion} \iso \bZ^{r+2}, \\
& (H_0^0)^\vee = \mathit{Hom}(H_0^0,\bZ) \iso \bZ^r.
\end{aligned}
\end{equation}
To clarify the situation, note that these sit in a short exact sequence
\begin{equation}
0 \rightarrow \bZ [\delta E|] \oplus {\textstyle\frac{1}{d}}\bZ([\bar{M}] + [\delta E|]) \longrightarrow H^\vee \longrightarrow (H_0^0)^\vee \rightarrow 0,
\end{equation}
where $d$ is the maximal divisibility of $[\bar{M}] + [\delta E|] \in H^\vee$. For concreteness, let's fix classes
\begin{equation} \label{eq:divisors} 
D_1,\dots,D_r \in H^2(\cornerbar{E};\bZ), \quad
D_k \cdot A_* = 0, 
\end{equation}
which give rise to a basis of $(H_0^0)^\vee$ (it is easy to see that such classes exist: start with any basis of $(H_0^0)^\vee$, lift the basis elements to $H^2(\cornerbar{E};\bZ)$, and then add integer multiples of $\delta E|$ or $\bar{M}$ to achieve the desired intersection number). To these classes we associate formal variables $q_1,\dots,q_r$, which we combine with the usual $q$ to form the ring $\bQ[q_1^{\pm 1},\dots,q_r^{\pm 1}][[q]]$. Take
\begin{equation} \label{eq:basic-q}
Q = q^{[\delta E|]} q_1^{D_1} \cdots q_r^{D_r}
\in  Hom(H,\bQ[q^{\pm 1}, q_1^{\pm 1},\dots,q_r^{\pm 1}]^\times).
\end{equation}
Because this homomorphism is multiplicatively-valued, we write $Q^A = q^{\delta E| \cdot A} q_1^{D_1 \cdot A} \cdots q_r^{D_r \cdot A}$ for its evaluation on $A \in H$. Our $B$-fields are formal sums
\begin{equation} \label{eq:b-field}
B  = b \,[\delta E|] + b_1 D_1 + \cdots + b_r D_r, \quad b,b_1,\dots,b_r \in q\bQ[q_1^{\pm 1},\dots,q_r^{\pm 1}][[q]].
\end{equation}
the correspondingly modified version of \eqref{eq:basic-q} is $Q\exp(B)$. One can think of $Q^A \exp(B \cdot A)$ intuitively as $\exp(-\Omega_B \cdot A)$, for
\begin{equation} \label{eq:basic-omega}
-\!\Omega_B = \log(Q) + B = (\log(q) +b )[\delta E|] + (\log(q_1) + b_1) D_1 + \cdots + (\log(q_r) + b_r) D_r.
\end{equation}
We will also use its derivative
\begin{equation} \label{eq:derivative-of-omega-b}
-\!\partial_q \Omega_B = q^{-1}[\delta E|] + \partial_q B = (q^{-1} + \partial_q b) [\delta E|] + (\partial_q b_1) D_1 + \cdots + (\partial_q b_r) D_r.
\end{equation}

\begin{lemma} \label{th:specialize}
There is a unique substitution $(q,q_1,\dots,q_r) \mapsto (g(q),g_1(q),\dots,g_r(g))$, for $g \in q+q^2\bQ[[q]]$, $g_k \in 1 + q\bQ[[q]]$, which transforms $Q\exp(B)$ into $q^{[\delta E|]}$. 
\end{lemma}

\begin{proof}
It is convenient to write $g(q) = qe^{\phi(q)}$, $g_k(q) = e^{\phi_k(q)}$ for $\phi,\phi_k \in \bQ[[q]]$, and to think in terms of \eqref{eq:basic-omega}. The desired equations are
\begin{equation} \label{eq:desired-order-by-order}
\begin{aligned}
& \phi(q) = -b(q e^{\phi(q)},e^{\phi_1(q)},\dots,e^{\phi_r(q)}), \\
& \phi_1(q) = -b_1(q e^{\phi(q)},e^{\phi_1(q)},\dots,e^{\phi_r(q)}), \\
& \dots \\
& \phi_r(q) = -b_r(q e^{\phi(q)},e^{\phi_1(q)},\dots,e^{\phi_r(q)}).
\end{aligned}
\end{equation}
Because $b,b_1,\dots,b_r$ have no $q$-constant term, the $q^m$ term on the right hand side of each equation depends only on the terms of $\phi(q),\phi_1(q),\dots,\phi_r(q)$ of $q$-order $<m$. This allows us to solve \eqref{eq:desired-order-by-order} order-by-order.
\end{proof}

The relevant version of the Gromov-Witten invariant \eqref{eq:zk} is
\begin{equation}
z^{(k)}_B  = \sum_{\bar{M} \cdot A = k} z_A Q^A \exp(B \cdot A),
\end{equation}
which for the cases of interest to us, yields
\begin{equation}
\begin{aligned}
& z^{(1)}_B \in q^{-1}[\delta E|] + H^2(\cornerbar{E};\bQ[q_1^{\pm 1},\dots,q_r^{\pm 1}])[[q]], \\
& z^{(2)}_B \in H^0(\cornerbar{E};\bQ[q_1^{\pm 1},\dots,q_r^{\pm 1}])[[q]] \iso
\bQ[q_1^{\pm 1},\dots,q_r^{\pm 1}][[q]].
\end{aligned}
\end{equation}
The replacement for \eqref{eq:as} is the following condition (a generalization of \cite[Assumption 3.1]{seidel15}):
\begin{equation} \label{eq:choose-b-field}
\begin{aligned}
& -\!\partial_q \Omega_B = \psi_B z^{(1)}_B - \eta_B [\bar{M}]  
\in q^{-1}[\delta E|] + H^2(\cornerbar{E};\bQ[q_1^{\pm 1},\dots,q_r^{\pm 1}])[[q]],
\\[-.2em] & \qquad \qquad
\text{with } \psi_B \in 1 + q\bQ[q_1^{\pm 1},\dots,q_r^{\pm 1}][[q]], \;\;
\eta_B \in \bQ[q_1^{\pm 1},\dots,q_r^{\pm 1}][[q]].
\end{aligned}
\end{equation}
There is a certain freedom in achieving this property: it is preserved under
\begin{equation} \label{eq:beta-change}
\begin{aligned}
& (B,\psi_B,\eta_B) \longmapsto (B|_{q \mapsto \gamma} + \log(\gamma/q)[\delta E|] , 
(\psi_B|_{q \mapsto \gamma}) \partial_q\gamma, (\eta_B|_{q \mapsto \gamma}) \partial_q\gamma) \\
& \qquad \qquad \qquad \qquad \qquad \qquad \text{for $\gamma \in q + q^2\bQ[q_1^{\pm 1},\dots,q_r^{\pm 1}][[q]]$.}
\end{aligned}
\end{equation}
We will need a version of \cite[Lemma 3.2]{seidel15}:

\begin{lemma} \label{th:choose-b-field}
One can choose \eqref{eq:b-field} so that \eqref{eq:choose-b-field} holds. Moreover, such a $B$ is unique up to transformations \eqref{eq:beta-change}.
\end{lemma}

\begin{proof}
We will show that there is a unique solution with $\psi_B = 1$. After that, the general result follows immediately by applying \eqref{eq:beta-change}. Expand 
\begin{equation}
\begin{aligned}
& B = qB_1 + q^2B_2 + \cdots, \\
& \eta_B = \eta_{B,0} + q\eta_{B,1} + \cdots.
\end{aligned}
\end{equation}
Suppose that we have chosen $B_1,\dots,B_{k-1}$ and $\eta_{B,0},\dots,\eta_{B,k-2}$, so that \eqref{eq:choose-b-field} holds modulo $q^{k-1}$ (with $\psi_B = 1$; this is automatically true for $k = 1$). One has
\begin{equation}
\text{$q^{k-1}$-term of } z_B^{(1)} = (B_k \cdot A_*) [\delta E|] + \text{ (terms depending on the previous choices).}
\end{equation}
The desired equation at order $q^{k-1}$ is therefore
\begin{equation} \label{eq:solve-for-bk}
\begin{aligned}
& kB_k - (B_k \cdot A_*)[\delta E|] = \text{(terms depending on previous choices) } - \eta_{B,k-1}[\bar{M}] \\ & \qquad \qquad \qquad \qquad \qquad \qquad \qquad \qquad \qquad
\in H^2(\cornerbar{E};\bQ[q_1^{\pm 1},\dots,q_r^{\pm 1}]).
\end{aligned}
\end{equation}
One can choose $\eta_{B,k-1}$ so that the right hand side lies in the $\bQ[q_1^{\pm 1},\dots,q_r^{\pm 1}]$-submodule generated by $[\delta E|],\, D_1,\dots,D_r$. The endomorphism $\Phi_k(X) = X - \frac{1}{k} (X \cdot A_*)[\delta E|]$ of that submodule is invertible for all $k \geq 2$, with inverse $\Phi_{-1-k}$. Hence, there is a unique solution for $B_k$ in \eqref{eq:solve-for-bk}.
\end{proof}

Assuming that \eqref{eq:choose-b-field} holds, one considers the analogue of \eqref{eq:schwarz-eq},
\begin{equation} \label{eq:weird-schwarz}
S_q f_B + 8z^{(2)}_B \psi_B^2 + \left(\eta_B - \frac{\psi_B'}{\psi_B}\right)' + \frac12 \left( \eta_B - \frac{\psi_B'}{\psi_B} \right)^2 = 0,
\end{equation}
where the differential and Schwarzian operator are with respect to $q$. This is an equation for an $f_B \in q\bQ[q_1^{\pm 1},\dots,q_r^{\pm 1}][[q]]$ whose $q^1$-term is invertible in $\bQ[q_1^{\pm 1},\dots,q_r^{\pm 1}]$. Turning to $A_\infty$-structures, one can define a version of \eqref{eq:b-algebra} over $\bQ[q_1^{\pm 1},\dots,q_r^{\pm 1}]$, by separating the topological contributions of the different holomogy classes of holomorphic polygons. Moreover, taking the twist \eqref{eq:b-field} into account, one has a version of the relative Fukaya category, which yields a deformation $\scrB_B$ defined over $\bQ[q_1^{\pm 1},\dots,q_r^{\pm 1}][[q]]$. Here is the direct analogue of Theorem \ref{th:1-variable}:

\begin{theorem} \label{th:1-variable-twisted}
Suppose that the topological assumptions \eqref{eq:connected} and \eqref{eq:simply-connected} hold, and choose $B$ so that it satisfies \eqref{eq:choose-b-field}. Let $f_B$ be a solution of \eqref{eq:weird-schwarz}. Then, there is a quasi-isomorphic model for $\scrB_B$, in which all $A_\infty$-operations have coefficients in $\bQ[f_B,q_1^{\pm 1},\dots,q_r^{\pm 1}] \subset \bQ[q_1^{\pm 1},\dots,q_r^{\pm 1}][[q]]$.
\end{theorem}

Lemma \ref{th:specialize} gives us a change of variables which turns $\scrB_B$ into the ordinary relative Fukaya category $\scrB_q$. In more formal language, if one views that change of coordinates as a map of rings
\begin{equation}
\begin{aligned}
& G_B: \bQ[q_1^{\pm 1},\dots,q_r^{\pm 1}][[q]] \longmapsto \bQ[[q]], \\
& G_B(q) = g, \; G_B(q_k) = g_k,
\end{aligned}
\end{equation}
then
\begin{equation}
\scrB_q \iso \scrB_B \otimes_{G_B} \bQ[[q]].
\end{equation}
As an immediate consequence, we get:

\begin{corollary} \label{th:b}
In the situation of Theorem \ref{th:1-variable-twisted}, there is a quasi-isomorphic model for $\scrB_q$ in which all $A_\infty$-operations have coefficients in $\bQ[f,g_1^{\pm 1},\dots,g_r^{\pm 1}] \subset \bQ[[q]]$, for $g_k = G_B(q_k)$ and $f = G_B(f_B)$.
\end{corollary}

Let's go back to the intrinsically defined Novikov rings from Section \ref{subsec:multivariable}. There are maps
\begin{equation}
\begin{aligned}
& K_B: \Lambda_{\geq k}^i \longrightarrow t^{i/2} q^k \bQ[q_1^{\pm 1},\dots,q_r^{\pm 1}][[q]], \\
& K_B(q^A) = q^{\delta E| \cdot A} q_1^{D_1 \cdot A} \cdots q_r^{D_r \cdot A} t^{\bar{M} \cdot A} \exp(B \cdot A).
\end{aligned}
\end{equation}
If Lemma \ref{th:specialize} applies, we have a commutative diagram that encodes the relation between this map and the previously defined \eqref{eq:specialize}:
\begin{equation} 
\xymatrix{
\Lambda_{\geq k}^i \ar[rr]^-{K_B} 
\ar[dr]_-{K}
&& 
t^{i/2}q^k\bQ[q_1^{\pm 1},\dots,q_r^{\pm 1}][[q]]
\ar[dl]^-{G_B}
\\ &
t^{i/2} q^k\bQ[[q]].
&
}
\end{equation}
If \eqref{eq:choose-b-field} holds, there is a commutative diagram which generalizes \eqref{eq:kappa-diagram}:
\begin{equation} \label{eq:kappa-diagram-2}
\xymatrix{
\Lambda_{\geq 0}^i \ar[d]_-{\partial_{z^{(1)}_\Lambda}} \ar[rr]^-{K_B} && 
t^{i/2}\bQ[q_1^{\pm 1},\dots,q_r^{\pm 1}][[q]] \ar[d]^-{\partial_{B,t}} \\
\Lambda_{\geq 0}^{i+2} \ar[rr]_-{K_B} &&
t^{i/2+1} \bQ[q_1^{\pm 1},\dots,q_r^{\pm 1}][[q]],
}
\end{equation}
where 
\begin{equation}
\partial_{B,t} = t\psi_B^{-1}( \partial_q + (i/2) \eta_B).
\end{equation}
This implies that under $K_B$, the $g_{\Lambda,k}$ from \eqref{eq:g-lambda-k} get mapped to functions whose $q$-derivative is zero. In fact, if we assume that $D_1,\dots,D_r$ give rise to a basis of $(H_0^0)^\vee$ which is dual to $(A_1,\dots,A_r)$, then $K_B(g_{\Lambda,k}) = q_k$. By combining this with \eqref{eq:kappa-diagram-2}, one sees that the functions $g_1,\dots,g_r$ from Theorem \ref{th:main} coincide with those of the same name in Corollary \ref{th:b}. Similarly, if $f_\Lambda$ satisfies \eqref{eq:weird-schwarz-0}, then $f_B = K_B(f_\Lambda)$ is a solution of
\begin{equation}
S_{B,t} f_B = 8z^{(2)}_B,
\end{equation}
where $S_{B,t}$ is the Schwarzian associated to $\partial_{B,t}$. One sees easily that this equation is equivalent to \eqref{eq:weird-schwarz}, hence $f_B$ can be chosen to coincide with the function of that name which appears in Theorem \ref{th:1-variable-twisted}. Finally, applying \eqref{eq:kappa-diagram-2}, one sees that the functions $f$ from Theorem \ref{th:main} and Corollary \ref{th:b} agree. In that way, Corollary \ref{th:b} implies Theorem \ref{th:main}.

\subsection{Hamiltonian Floer theory with $B$-fields}
We now start our stroll through the various theories encountered in our argument, and see how each of them should be adapted to the presence of $B$-fields. The most straightforward instance is Hamiltonian Floer cohomology. While one can formulate a notion of $B$-field in an intrinsic way, we will stick to the kind of ad hoc language adopted in Section \ref{subsec:b-field}. To make things particularly clear, we will work with a closed symplectic manifold $\bar{M}$, with $c_1(\bar{M}) = 0$ and integral $[\omega_{\bar{M}}]$. Take $(H_t)$ with nondegenerate one-periodic orbits. We choose a codimension two submanifold $\Omega$ representing $[\omega_{\bar{M}}]$, and further codimension two submanifolds $D_1,\dots,D_r$; all of them are assumed to be disjoint from the one-periodic orbits of $(H_t)$. Our $B$-field (realized on the cycle level, rather than just in cohomology) is a formal weighted sum
\begin{equation} \label{eq:cycle-b}
B = b \Omega + b_1 D_1 + \cdots + b_r D_r, \;\; b,b_1,\dots,b_r \in q\bQ[q_1^{\pm 1},\dots,q_r^{\pm 1}][[q]].
\end{equation}
Similarly, we define the formal expression $Q = q^{\Omega} q_1^{D_1} \cdots q_r^{D_r}$.
The Floer complex is then defined by
\begin{equation} \label{eq:cf-b}
\begin{aligned}
& \mathit{CF}^*_{q,q^{-1},B}(H_t,J_t) = \bigoplus_x \bQ[q_1^{\pm 1},\dots,q_r^{\pm 1}]((q)) x, \\
&
\text{$(x_-)$-coefficient of } d_{q,q^{-1},B}\,x_+ = \sum_u \pm Q^u \exp(B \cdot u).
\end{aligned}
\end{equation}
Here, $Q^u = q^{\Omega \cdot u} q_1^{D_1 \cdot u} \cdots q_r^{D_r \cdot u}$ and $B \cdot u = b (\Omega \cdot u) + \cdots + b_r (D_r \cdot u)$. The sum over $u$ in \eqref{eq:cf-b} makes sense since $\exp(B \cdot u) \in 1 + q\bQ[q_1^{\pm 1},\dots,q_r^{\pm 1}][[q]]$, meaning that the growth of powers of $q$ is controlled by $q^{\Omega \cdot u}$. Denote the resulting cohomology by $\mathit{HF}^*_{q,q^{-1},B}(H_t,J_t)$. It is independent of the choice of submanifold representatives up to canonical isomorphism, for the same reason as in Remark \ref{th:unique-hf}. The formula for $d_{q,q^{-1},B}$ also makes sense in the context of relative Floer cohomology, with $\Omega$ replaced by $\delta M$, since the new term $\exp(u \cdot Z)$ does not introduce negative powers of $q$. Under the same additional requirements as in Section \ref{sec:relative-floer}, this leads to groups $\mathit{HF}^*_{q,B}(H_t,J_t)$ defined over $\bQ[q_1^{\pm 1},\dots,q_r^{\pm 1}][[q]]$.

$B$-field twisted Hamiltonian Floer cohomology is generally straightforward to develop. The only aspect that may deserve discussion concerns $q$-differentiation operations. Consider  
\begin{equation} \label{eq:dqdq}
\begin{aligned}
& \partial_q d_{q,q^{-1},B}: \mathit{CF}^*_{q,q^{-1},B}(H_t,J_t) \longrightarrow \mathit{CF}^{*+1}_{q,q^{-1},B}(H_t,J_t), \\
& (\partial_q d_{q,q^{-1},B})(c) = \partial_q (d_{q,q^{-1},B}c) - d_{q,q^{-1},B}(\partial_q c).
\end{aligned}
\end{equation}
Here, $\partial_q$ acts on the Floer complex in the obvious way, given the basis \eqref{eq:cf-b}. Concretely,
\begin{equation}
\text{$(x_-)$-coefficient of } (\partial_q  d_{q,q^{-1},B})(x_+) = \sum_u  \pm \big((q^{-1}\Omega + \partial_q B) \cdot u\big) Q^u \exp(B \cdot u).
\end{equation}
For generic $(J_t)$, one can interpret $\sum_u \pm (q^{-1}\Omega + \partial_q B) \cdot u$ as follows. It counts pairs $(u,\theta)$, where $u$ is an isolated Floer trajectory and $\theta \in S^1$ a point such that $u(0,\theta) \in \Omega$, with multiplicity $q^{-1} + \partial_q b$; and similarly for any $u(0,\theta) \in D_k$, with multiplicities $\partial_q b_k$. This is a special case of a general operation, by which a pseudo-cycle $Z$ in $\bar{M}$ of codimension $c$, with coefficients in $\bQ[q_1^{\pm 1}, \cdots, q_r^{\pm 1}]((q))$, gives rise to a chain map for generic $(J_t)$,
\begin{equation}
L_Z: \mathit{CF}^*_{q,q^{-1},B}(H_t,J_t) \longrightarrow \mathit{CF}^{*+c-1}_{q,q^{-1},B}(H_t,J_t).
\end{equation}
In our case, the cycle that occurs matches what we previously saw in \eqref{eq:derivative-of-omega-b}:
\begin{equation} \label{eq:z-cycle}
Z = (q^{-1} + \partial_q b)\Omega + (\partial_q b_1)D_1 + \cdots + (\partial_qb_r) D_r.
\end{equation} 

The entire discussion carries over without any issues to the case of noncompact manifolds $\bar{E}$. To construct connections generalizing those in \eqref{eq:hamiltonian-connections}, \eqref{eq:hamiltonian-connections-2}, one assumes that the image of \eqref{eq:z-cycle} under a suitable PSS map is nullhomologous. All the material from \cite{seidel16,seidel17} then carries over, always using weights as in \eqref{eq:cf-b} when counting any kind of holomorphic curve (for the parts involving Gromov-Witten invariants, the $B$-field is always chosen to be defined on $\cornerbar{E}$).

\subsection{Relative Fukaya category with $B$-fields}
One can twist Fukaya categories in an analogous way, with a few additional complications. Let's work with a closed $\bar{M}$ as before (and we additionally choose a trivialization of the canonical bundle), together with a symplectic divisor $\delta M$ Poincar{\'e} dual to $[\omega_{\bar{M}}]$. Take a finite collection of closed Lagrangian submanifolds $(V_1,\dots,V_m)$ in $M = \bar{M} \setminus \delta M$, which are Lagrangian branes in an appropriate sense: namely, 
\begin{equation}
\mybox{
they are relatively exact, meaning that $[\delta M] = [\omega_{\bar{M}}] \in H^2(\bar{M},V_i;\bR)$; they are also graded, with respect to our choice of trivialization of the canonical bundle; and they come equipped with a choice of {\em Spin} structure.
}
\end{equation}
The definition of $B$-field in this context is a more elaborate version of \eqref{eq:cycle-b}:
\begin{equation} \label{eq:cycle-b-2}
\mybox{
$B = b\, \delta M + b_1 D_1 + \cdots + b_r D_r$, $b,b_1,\dots,b_r \in q\bQ[q_1^{\pm 1},\dots,q_r^{\pm 1}][[q]]$. We assume that $D_l$ restricts to the trivial class in $H^2(V_i;\bZ)$ for any $(i,l)$.}
\end{equation}
Based on the last-mentioned topological assumption, we would like to make further choices. 
\begin{equation} 
\mybox{
For each $(k,i)$, we choose a codimension two submanifold $\Delta_{l,i} \subset [0,1/2] \times \bar{M}$, such that $\Delta_{l,i} \cap (\{0\} \times \bar{M})$ (transverse intersection) is disjoint from $V_i$, while $\Delta_{l,i} \cap ([1/4,1/2] \times \bar{M}) = [1/4,1/2] \times D_k$.
}
\end{equation}
Choose, for any $(i,j)$ in $\{1,\dots,m\}$, a time-dependent Hamiltonian $(H_{i,j,t})$ whose vector field $(X_{i,j,t})$ is parallel to $\delta M$. A generic choice ensures that:
\begin{equation} \label{eq:disjoint-chord}
\mybox{For any $X_{i,j,t}$-chord $x: [0,1] \rightarrow M$ from $V_i$ to $V_j$, the point $(t,x(t))$, $0 \leq t \leq 1/2$, does not lie on any $\Delta_{l,i}$, and similarly $(t,x(1-t))$ does not lie on any $\Delta_{l,j}$.}
\end{equation}

\begin{figure}
\begin{centering}
\includegraphics{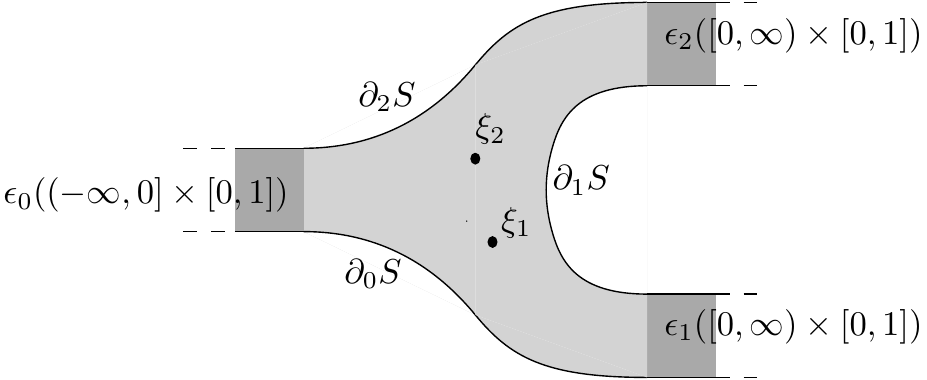}
\caption{\label{fig:surface}Punctured surfaces $S$ with extra interior marked points, which appear in the definition of the relative Fukaya category.}
\end{centering}
\end{figure}%

In the definition of the Fukaya $A_\infty$-structure \cite{sheridan11b, perutz-sheridan20, seidel20}, one uses surfaces which are $(d+1)$-pointed discs with an additional choice of strip-like ends $\epsilon_0,\dots,\epsilon_d$, and (ordered and pairwise disjoint) interior marked points $\xi_1,\dots,\xi_e$, for some $e \geq 0$, see Figure \ref{fig:surface}; this is subject to the stability condition $d+2e \geq 2$. Given $i_0,\dots,i_d \in \{1,\dots,m\}$, one considers maps
\begin{equation} \label{eq:u-map}
\left\{
\begin{aligned}
& u: S \longrightarrow \bar{M}, \\
& u(\partial_k S) \subset V_{i_k}, \\
& \textstyle \lim_{s \rightarrow -\infty} u(\epsilon_0(s,\cdot)) = x_0 \;\; \text{ an $(X_{i_0,i_d,t})$-chord for $(V_{i_0},V_{i_d})$}, \\
& \textstyle \lim_{s \rightarrow +\infty} u(\epsilon_k(s,\cdot)) = x_k \;\; \text{ an $(X_{i_{k-1},i_k,t})$-chord for $(V_{i_{k-1}},V_{i_k})$, if $k>0$}, \\
& u^{-1}(\delta M) = \{\xi_1,\dots,\xi_e\}, \\
& u \cdot \delta M = e. \\
\end{aligned}
\right.
\end{equation}
(These are of course subject to a suitable Cauchy-Riemann equation, but that's not relevant for the purely topological considerations here.) Here, $\partial_k S$ are the connected components of $\partial S$, numbered as in Figure \ref{fig:surface}. Let's equip $S$ with smooth embeddings which have pairwise disjoint images, and which parametrize a tubular neighbourhood of $\partial S$, compatibly with the ends:
\begin{equation} \label{eq:boundary-neighbourhoods}
\left\{
\begin{aligned}
& \delta_0,\dots,\delta_d: \bR \times [0,1/4] \longrightarrow S, \\
& \delta_k^{-1}(\partial S) = \delta_k^{-1}(\partial_k S) = \bR \times \{0\}, \\
& \delta_0(s,t) = \epsilon_0(s,t) \text{ for $s \ll 0$,} \\
& \delta_0(s,t) = \epsilon_1(s,t) \text{ for $s \gg 0$,} \\
& \delta_1(s,t) = \epsilon_1(-s,1-t) \text{ for $s \ll 0$,} \\
& \delta_1(s,t) = \epsilon_2(s,1-t) \text{ for $s \gg 0$,} \\
& \cdots \\
& \delta_d(s,t) = \epsilon_0(-s,1-t) \text{ for $s \gg 0$.}
\end{aligned}
\right.
\end{equation}
This gives rise to codimension two submanifolds $\Delta_{S,1},\dots, \Delta_{S,r} \subset S \times \bar{M}$, namely
\begin{equation} \label{eq:big-delta}
\Delta_{S,l} \cap (\{z\} \times \bar{M}) = \begin{cases} 
\Delta_{l,i_k} \cap (\{t\} \times \bar{M}) 
& \text{for $z = \delta_k(s,t)$, $0 \leq t \leq 1/4$,} \\
D_l & \text{otherwise.}
\end{cases}
\end{equation}
Note that this condition implies
\begin{equation} \label{eq:delta-on-the-ends}
\text{for $s \ll 0$,} \quad
\Delta_{S,l} \cap (\{\epsilon_0(s,t)\} \times \bar{M}) = \begin{cases}
\Delta_{l,i_0} \cap (\{t\} \times \bar{M}) & t \in [0,1/4], \\ 
D_l & t \in [1/4,3/4], \\
\Delta_{l,i_d} \cap (\{1-t\} \times \bar{M}) & t \in [3/4,1];
\end{cases}
\end{equation}
and correspondingly for the other ends. Given that and a map \eqref{eq:u-map}, one has a well-defined intersection number
\begin{equation} \label{eq:delta-intersect}
(\mathit{id},u) \cdot \Delta_{S,l} \in \bZ,
\end{equation}
which (intuitively) counts points $z \in S$ such that $(z,u(z)) \in \Delta_{S,l}$. Because of the boundary conditions and the assumptions on $\Delta_{S,l} \cap (\{\delta_k(s,0)\} \times \bar{M})$ in \eqref{eq:big-delta}, there can be no such points with $z \in \partial S$. Equally, because of the asymptotic condition on $u$, the assumption \eqref{eq:disjoint-chord}, and \eqref{eq:delta-on-the-ends}, there can also be no such points with $z$ sufficiently close to $\infty$ on the $\epsilon_0$ end. Similar observations apply to all ends. Hence, the intersection is proper, which justifies \eqref{eq:delta-intersect}. Of course, in order for these numbers to make sense consistently, one has to assume that \eqref{eq:boundary-neighbourhoods} depend smoothly on moduli, and are compatible with gluing together surfaces along their ends. When defining the twisted version $\scrB_{q,B}$ of the relative Fukaya category, we replace the standard expression $q^{u \cdot \delta M}$ with
\begin{equation}
q^{u \cdot \delta M} q_1^{(\mathit{id},u) \cdot \Delta_{S,1}} \cdots
q_r^{(\mathit{id},u) \cdot \Delta_{S,r}} \exp\big(b \,(u \cdot \delta M) +
b_1 \, ((\mathit{id},u) \cdot \Delta_{S,1}) + \cdots + b_r \, ((\mathit{id},u) \cdot \Delta_{S,r})\big).
\end{equation}
When taking the $q$-derivative, one picks up an additional multiplicative factor $(q^{-1} + \partial_q b) (\delta M \cdot u) + (\partial_q b_1) ((\mathit{id},u) \cdot \Delta_{S,1}) + \cdots + (\partial_q b_r) ((\mathit{id},u) \cdot \Delta_{S,1})$. As before, this allows us to relate the Kaledin class to the image of \eqref{eq:z-cycle} under an closed-open string map.

In our applications, $\bar{M}$ is a fibre of a Lefschetz fibration, and we take the $D_i$ to come from cohomology classes in $\bar{E}$. As a consequence, the cohomological condition from \eqref{eq:cycle-b-2} is satisfied for the vanishing cycles, leading to the definition of a twisted version $\scrB_B$ of the category $\scrB_q$ from Section \ref{subsec:directed-2}, which already appeared in Section \ref{subsec:b-field}. When considering the Fukaya category of $\bar{E}$ itself, as in Section \ref{subsec:fukaya-of-lefschetz}, we only use the (contractible) Lefschetz thimbles as objects, so the cohomological condition is automatically satisfied, giving us the analogue $\scrA_B$ of $\scrA_q$. The basic properties of those categories carry over without any issues. In particular, the $q$-derivative argument hinted at above leads to a $B$-field twisted version of Proposition \ref{th:first-co} (and its partial improvement in Proposition \ref{th:first-co-2}), following the same method as in \cite{seidel21b}.

\subsection{Homological algebra}
The algebraic theory in Sections \ref{sec:outline}--\ref{sec:connections} used $\bQ$-coefficients, and $\bQ[[q]]$ for deformation theory. For the $B$-field twist, that needs to be adapted to work with coefficients in $\bQ[q_1^{\pm 1},\dots,q_r^{\pm 1}]$, respectively $\bQ[q_1^{\pm 1},\dots,q_r^{\pm 1}][[q]]$. For simplicity, we will formulate our discussion for $A_\infty$-algebras (see Remark \ref{th:semisimple}; the same applies here).

Let $R$ be a commutative ring of finite global dimension. When considering $A_\infty$-algebras $\scrA$ over $R$, we assume that each graded piece is a free $R$-module
(this is sensible because unbounded complexes of free $R$-modules are $K$-projective). A cohomological unit for such an algebra is a cocycle $e \in \scrA^0$ such that $Re \subset \scrA^0$ and $\scrA^0/Re$ are free, and which becomes an identity element in $H^*(\scrA)$. This property means that 
\begin{equation} \label{eq:right-mu}
\mu^2_{\scrA}(\cdot,e): \scrA \longrightarrow \scrA
\end{equation}
is a quasi-isomorphism, and therefore invertible up to chain homotopy. We also know that \eqref{eq:right-mu} is chain homotopic to its square. Therefore, it must be chain homotopic to the identity. The same holds for left multiplication with $e$. For the notion of strict unit, one replaces the last condition by the standard requirements on $\mu_{\scrA}^d(\cdots,e,\dots)$, and retains all the other assumptions. The relevant version of \cite[Lemma 2.1]{seidel04} is:

\begin{lemma} \label{th:unify}
Let $\scrA$ be a cohomologically unital $A_\infty$-algebra, such that $R e \rightarrow \scrA$ is split-injective as a map of chain complexes. Then there is a strictly unital $A_\infty$-structure $\scrA^{(1)}$, living on the same graded $R$-module and with the same differential, and an $A_\infty$-homomorphism $\scrF: \scrA \rightarrow \tilde{\scrA}$ whose linear term is the identity.
\end{lemma}

\begin{proof}
Write the chain complex as a sum $\scrA = R e \oplus \bar{\scrA}$. We know that $\mu^2_{\scrA}(\cdot,e)$ and $(-1)^{|a|}\mu^2_{\scrA}(e,\cdot)$ are homotopic to the identity. In sign conventions that are suitable for the $A_\infty$-world, this means that we have maps
\begin{equation}
\begin{aligned}
& h^{\mathit{left}}: \scrA \longrightarrow \scrA[-1], 
\;\; \mu^1_{\scrA} h^{\mathit{left}}(a) + h^{\mathit{left}} \mu^1_{\scrA}(a) = \mu^2_{\scrA}(e,a) - (-1)^{|a|} a, \\
& h^{\mathit{right}}: \scrA \longrightarrow \scrA[-1], \;\;
\mu^1_{\scrA} h^{\mathit{right}}(a) + h^{\mathit{right}} \mu^1_{\scrA}(a) = \mu^2_{\scrA}(a,e) - a.
\end{aligned}
\end{equation}
Note that the difference $h^{\mathit{left}}(e) - h^{\mathit{right}}(e) \in \scrA^{-1}$ is a cocycle. One can change $h^{\mathit{left}}$ by a chain map $\scrA \rightarrow \scrA[-1]$ which is zero on $\bar{\scrA}$, in order to make that difference equal to zero. Take
\begin{equation}
\left\{
\begin{aligned}
& \scrF^2: \scrA \otimes \scrA \longrightarrow \scrA[-1], \\
& \scrF^2\,|\,\bar\scrA \otimes \bar\scrA = 0, \\
& \scrF^2\,|\, \{e\} \times \scrA = h^{\mathit{left}}, \\
& \scrF^2\,|\, \bar\scrA \times \{e\} = h^{\mathit{right}}.
\end{aligned}
\right.
\end{equation}
If we then define $\mu_{\scrA^{(1)}}^2$ so that $\scrF^2$ satisfies the $A_\infty$-homomorphism equation from $(\mu^1_{\scrA},\,\mu^2_{\scrA})$ to $(\mu^1_{\scrA},\, \mu^2_{\scrA^{(1)}})$, the outcome is that $\mu^2_{\scrA^{(1)}}(e,a) = (-1)^{|a|}a$, $\mu^2_{\scrA^{(1)}}(a,e) = a$, as desired. The subsequent steps are as in \cite[Lemma 2.1]{seidel04}.
\end{proof}

We have used \cite[Lemma 2.1]{seidel04} in the proof of Lemma \ref{th:directed}. The version with $B$-fields would rely on Lemma \ref{th:unify} instead; the split-injectivity condition is always satisfied in our case, because of the way we set up the Floer chain complex for a vanishing cycle, compare \eqref{eq:cf-vk}. A related issue that appears in the proof of Lemma \ref{th:directed-2} is addressed by the following:

\begin{lemma} \label{th:reduced}
Let $\scrA$ be a strictly unital $A_\infty$-algebra over $R$. Then, the inclusion of the reduced Hochschild complex $\mathit{CC}^{\mathit{red},*}(\scrA,\scrA)$ (the complex of Hochschild cochains which vanish on $e$) into the entire Hochschild complex $\mathit{CC}^*(\scrA,\scrA)$ is a quasi-isomorphism.
\end{lemma}

\begin{proof}
We introduce a decreasing filtration $F^p \mathit{CC}^*(\scrA,\scrA)$, which consists of those $\gamma$ which become zero if we insert $e$ into one of the last $p$ entries. Clearly, $\mathit{CC}^{\mathit{red},*}(\scrA,\scrA) = \bigcap_p F^p \mathit{CC}^*(\scrA,\scrA)$. The induced filtration $F^p \mathit{CC}^*(\scrA,\scrA)/\mathit{CC}^{\mathit{red},*}(\scrA,\scrA)$ of the quotient $\mathit{CC}^*(\scrA,\scrA)/\mathit{CC}^{\mathit{red},*}(\scrA,\scrA)$ is complete. To see that, it is convenient to pick a splitting (of $R$-modules) $\scrA = Re \oplus \bar{\scrA}$. Then, ignoring the gradings to simplify the notation,
\begin{equation}
\frac{\mathit{CC}^*(\scrA,\scrA)}{\mathit{CC}^{*,\mathit{red}}(\scrA,\scrA)} =
\prod_{r \geq 1,\; p_0,\dots,p_r \geq 0} \!\!\! \mathit{Hom}(\bar\scrA^{\otimes p_r} \otimes R e \otimes \bar\scrA^{\otimes p_{r-1}} \otimes \cdots \otimes R e \otimes \bar\scrA^{\otimes p_0}, \scrA),
\end{equation}
and our filtration is by $p_0 \geq p$. Hence, it is sufficient to prove that the successive quotients of our filtration are acyclic. Consider the degree $-1$ endomorphism of the Hochschild complex defined by 
\begin{equation} 
(h^p \gamma)^d(a_d,\dots,a_1) = (-1)^{\|a_1\|+\cdots+\|a_p\|} \gamma^{d+1}(a_d,\dots,a_{p+1},e,a_p,\dots,a_1)
\end{equation}
(the formula is understood to mean zero if $d<p$). For $\gamma \in F^p \mathit{CC}^*(\scrA,\scrA)$, we have
\begin{equation}
\begin{aligned}
& (\delta h^p \gamma + h^p \delta \gamma)^d(a_d,\dots,a_1) \\ & \quad  = \sum_{i>0,\, j \geq p} \pm
\mu^{d-j+1}_{\scrA}(a_d,\dots,a_{i+j+1},\gamma^{j+1}(a_{i+j},e,a_{i+p},\dots,a_{i+1}),a_i,\dots,a_1).
\end{aligned}
\end{equation}
Because of the $i>0$ restriction, the right hand side lies in $F^{p+1}\mathit{CC}^*(\scrA,\scrA)$. Hence, our formula defines a contracting homotopy for $F^p \mathit{CC}^*(\scrA,\scrA)/F^{p+1}\mathit{CC}^*(\scrA,\scrA)$.
\end{proof}

%

In the theory from Section \ref{sec:noncommutative-divisors}, when talking about $A_\infty$-algebras over $R$, one always wants the inclusion of such a subalgebra to be a split injective map of graded $R$-modules (something that is never an issue in our applications). The proof of Lemma \ref{th:turn-into-inclusion} relies on the acyclicity of the complexes
\begin{equation} \label{eq:c-revisited}
F^{r-1} \frakc / F^r c = \mathit{Hom}\big(\scrB^{\otimes r}/\scrA^{\otimes r}, \{\scrB \rightarrow \tilde{\scrB}\big).
\end{equation} 
Here, the differential comes from the differentials on $\scrB$ and $\tilde{\scrB}$, together with the given map $\scrB \rightarrow \tilde{\scrB}$. That map is a quasi-isomorphism by assumption, so $\{\scrB \rightarrow \tilde{\scrB}\}$ is acyclic. Since $\scrB^{\otimes r}/\scrA^{\otimes r}$ is a complex of free $R$-modules, it follows that \eqref{eq:c-revisited} is acyclic as well. In other words, it is not necessary to first pass to the cohomology of $\scrA$ or $\scrB$, as we had chosen to do (for expository reasons) in \eqref{eq:b-filtered}, \eqref{eq:g-filtered}, \eqref{eq:p-filtered}.

Concerning the theory of $A_\infty$-connections (Section \ref{sec:connections}), no modifications or additional explanations are necessary: Lemmas \ref{th:kaledin-2} and \ref{th:unique-gauge} only require unique divisibility by positive integers, hence work well over $\bQ[q_1^{\pm 1},\dots,q_r^{\pm 1}]$.


\end{document}